\documentclass[11pt]{article}

 \usepackage{amscd,amsmath,amsfonts,amssymb,amsthm,%cite
 ,mathrsfs,color} 
\usepackage{dsfont} \usepackage{leftidx}
\usepackage{tikz} \usepackage{enumerate}  \usepackage{hyperref} %\usepackage[backref]{hyperref}
 \usepackage{appendix}

 \usepackage[left=2.5cm, right=2.5cm, top=2.5cm, bottom=2.5cm]{geometry}

\usepackage{geometry}
\usepackage{float}
\usepackage{subfigure}
\usepackage{graphicx} \usepackage{xcolor}
\numberwithin{equation}{section} % Equations numbered according to sections

\usepackage{amsthm}

\theoremstyle{plain}
\newtheorem{theorem}{Theorem}[section]
\newtheorem{lemma}[theorem]{Lemma}

\newtheorem{proposition}[theorem]{Proposition}

\theoremstyle{definition}
\newtheorem{definition}[theorem]{Definition}
\newtheorem{example}[theorem]{Example}

\newtheorem{remark}[theorem]{Remark}%[section]

\begin{document}
%%%%%%%%%%%%%%%%%%%%%%%%%%%%%%%%%%%%%%%%%%%%%%%%%%%%%%%%%%%%%%%%%%%%%%%%%%%%%%%%%%%%%%%%%%%%%%%%%%%

\title{\bf{ 
%Eigenvalue  statistics for deformed GinUEs I: Critical edge
Critical edge statistics for deformed GinUEs%Deformed complex Ginibre ensembles 
}}

\author{
Dang-Zheng Liu\footnotemark[1] ~ and     Lu Zhang\footnotemark[2]}
\renewcommand{\thefootnote}{\fnsymbol{footnote}}
\footnotetext[1]{%CAS Key Laboratory of Wu Wen-Tsun Mathematics, 
School of Mathematical Sciences, University of Science and Technology of China, Hefei 230026, P.R.~China. E-mail: dzliu@ustc.edu.cn}
\footnotetext[2]{School of Mathematical Sciences, University of Science and Technology of China, Hefei 230026, P.R.~China. E-mail: zl123456@mail.ustc.edu.cn}

%\footnotetext[3]{ Keywords:  Deformed ensembles, Ginibre matrices, edge statistics}

%\Keywords{keyword1, Keyword2, Keyword3, Keyword4} Keywords Random matrices · Spiked models · Extreme eigenvalue statistics ·
%Gaussian fluctuations · Ginibre matrices
%Mathematics Subject Classification 15A52 · 60F05  Mathematics Subject Classification (2000) 60B20

 \maketitle
%--------------------------------------------------------------------------------------------------
 \begin{abstract}
 For the complex Ginibre ensemble subjected to an additive perturbation by a deterministic normal matrix $X_0$, we establish that under specific spectral conditions on 
$X_0$, only two distinct types of local spectral statistics emerge at the spectral edge: GinUE statistics and critical statistics, which respectively correspond to regular   and quadratically vanishing spectral points. The critical statistics, as a non-Hermitian analogue of Pearcey statistics in random matrix theory,  describes a novel point process on the complex plane.
This identifies the third (and likely final) universal statistics  in non-Hermitian random matrix theory, after  the established GinUE bulk and edge universality classes,  and represents the primary  achievement of this paper.

 \end{abstract}

\tableofcontents

\section{Introduction and main results}

%%%%%%%%%%%%%%%
\subsection{Introduction}

We  begin our study by specifying the deformed complex Ginibre ensembles (GinUE for short)
 \begin{equation} \label{defmat}
 X:=X_0 +\sqrt{\frac{\tau}{N}}G,  
 \end{equation}
where $X_0$   is a deterministic  matrix, and    $G=[g_{ij}]_{i,j=1}^N$ is a random  matrix  whose entries   are i.i.d.  complex normal   variables with mean 0 and variance 1. 
 Or, to be more exact,  we have 
 \begin{definition}  \label{GinU} 
A random  complex  
$N\times N$ matrix  $X$,    is said to belong to the deformed complex  Ginibre ensemble with mean $X_0$ and  time $\tau>0$, denoted by GinUE$_{N}(X_0)$,  if  the joint probability density function for  matrix entries  is given by   
\begin{equation}\label{model}
P_{N}(X_0;X)=    \Big(\frac{N}{\pi\tau}\Big)^{N^2}\ 
e^{ -\frac{N}{\tau} {\rm Tr} (X-X_0)(X-X_0)^*}.
\end{equation}

\end{definition}

One  motivation behind the study of the deformed model    \eqref{defmat}   comes mostly
from the general effort toward the understanding of the effect of a perturbation %with small rank 
on the spectrum of a large-dimensional random matrix.   Another is  that  
 eigenvalues  of non-Hermitian matrices may  present    instability, unlike eigenvalues of Hermitian matrices; see \cite{BC16} and \cite{LZ}  for  a more detailed explanation.
Although   the deformed GUEs  or the general deformed Hermitian models   have  been  
quite well understood (see e.g. \cite{CP16} and references therein), not much is known about   the  local  eigenvalue statistics  in the non-Hermitian situation.     When  the perturbation strength  is above a certain threshold (supercritical regime),    extreme eigenvalues will stay away from the bulk   and outlier phenomena  have been   well  studied \cite{BR,BC16,Ta13}. 
Meanwhile, when  the perturbation strength  is near or below   a certain threshold (subcritical and critical  regimes), in one  previous paper \cite{LZ}  we investigate  the  finite-rank  perturbation effect and prove that 
 the edge statistics   is characterized  by a new class of determinantal point processes, for 
which  correlation kernels  depend  only on geometric multiplicity of eigenvalue and  can be expressed 
in terms of the repeated erfc integrals.   Together,  all  these  results  establish     a non-Hermitian analogue of the BBP phase  transition for the largest eigenvalue of spiked random matrices \cite{BBP}.  %,  named after     the famous work of  Baik, Ben~Arous and P{\'e}ch{\'e} .  
Our goal is to   identify all the possible local  eigenvalue statistics at the edge and in the bulk  of the spectrum in a   series of papers.  Related to those   is  a  key  result   from  \cite[Proposition 1.3]{LZ}, which  expresses correlation functions as matrix integrals  in term of auto-correlations  of  characteristic polynomials. This  indeed  
plays a central role in our studying eigenvalue statistics   of  the deformed GinUE.

 Historically, % In the non-perturbative  case,  that is,  $X_0=0$ in \eqref{defmat},  
 the study of  non-Hermitian random matrices  was first initiated  by  Ginibre  \cite{Gi}  for those  matrices with i.i.d. real/complex standard  Gaussian entries, and then was extended to  i.i.d.  entries.   At a macroscopic  level,     the limiting spectral   measure was     the famous \textit{circular law}, which  is   a uniform measure on the unit disk,   after  a long list of works including Bai \cite{Ba},  Girko \cite{Gir},   G\"{o}tze and Tikhomirov \cite{GT}, Pan and Zhou  \cite{PZ},  and Tao and Vu \cite{TV10}.    See a recent survey \cite{BC12} and references therein for more details. 
  However, at  a   microscopic level,     local  eigenvalue statistics  in the bulk and at the  soft edge  of the spectrum were  revealed    first  for the   Ginibre ensembles   \cite{BS, FH,Ka},  with the help of  exact eigenvalue  correlations.     These Ginibre statistics  are
  conjectured to hold true even  for    i.i.d.   random matrices,   although  the proof  seems much more difficult  than  the Hermitian analogy.  Later  in \cite{TV15},  Tao and Vu  established  a four moment match theorem  in the bulk and at the  soft edge, % as a non-Hermitian version of  their previous work   \cite{TV11},  
both in the real and complex cases.  
    Cipolloni,  Erd\H{o}s and  Schr\H{o}der  recently  remove the four moment
matching condition from \cite{TV15} and prove   the edge universality for any   random matrix with  i.i.d.  entries  under the assumption of higher moments \cite{CES}.   Very recently,   Maltsev and Osman  \cite{MO23} prove   the bulk universality   in the    i.i.d. case.
%Bulk Universality for Complex non-Hermitian Matrices with Independent and Identically Distributed Entries
%But their  method   doesn't  seem to work in the bulk. 
% For more details about local universality  of  non-Hermitian random matrices,  see \cite{CES,TV15} and references  therein. 
   See    \cite{BF23a,BF23b}   for a comprehensive review   on non-Hermitian random matrices.% given by Byun and  Forrester    

 For the deformed model  \eqref{defmat},    the  circular law still holds  under a  small low rank perturbation from $X_0$, see   \cite[Corollary 1.12]{TV10} as the culmination of work by many authors and a survey \cite{BC12} for further details. 
 Particularly   when  $X_0$ has bounded rank and bounded operator norm,  such a perturbation can  create outliers  outside the unit disk  \cite{Ta13}.  Since then,  this outlier phenomenon has been  extensively  studied  in \cite{BC16,BZ,BR,COW,OR}.  Furthermore,  the fluctuations of  outlier eigenvalues are  investigated respectively by 
 Benaych-Georges and Rochet   for   deformed invariantly matrix ensembles \cite{BR},  and   %in the sense of the single ring theorem,
 by Bordenave and Capitaine for deformed i.i.d. random matrices \cite{BC16}.   These fluctuations,  due to non-Hermitian structure,  become much more complicated, and indeed  highly depend on the shape of the Jordan canonical form of the perturbation.     It is worth emphasizing that 
  when  $G$  takes a  matrix-valued Brownian motion, \eqref{defmat}  becomes a dynamical random matrix  and remains something of a mystery; see  \cite{BD,BGN14, BGN15}.

 Our main goal is to   identify all the possible local  edge  statistics  
 for  the $n$-point correlation functions, which is defined as 
\begin{equation}\label{Correlation functionDef}
R_N^{(n)}(X_0;z_1,\cdots,z_n):=\frac{N!}{(N-n)!}
\int\cdots\int P_N(X_0;z_1,\cdots,z_N)
{\rm d}z_{n+1}\cdots{\rm d}z_N,
\end{equation}
with   $P_N(X_0;z_1,\cdots,z_N)$ being  the joint eigenvalue   density function  of  the deformed  model $X$ given   in Definition \ref{GinU},   under certain restriction    on  the initial matrix $X_0$. 
Closely related to  this is the limiting   spectral measure of the deformed model. Its characterization can be given by  using tools from free probability,    the Brown measure of free circular Brownian motion  $x_0 + \sqrt{\tau}c$, where  $c$  is   a circular variable,   freely independent of 
a general  operator  $x_0$.    The Brown measure and corresponding density formula has been extensively studied in   \cite{BYZ, BCC13,  EJ,  HZ23, Zh21}.  The exact density  will  play a crucial   role  in the study of local eigenvalue statistics for non-Hermitian random matrices.   Recently,  \cite{EJ}  Erd\H{o}s and Ji  further  observe  a remarkable phenomenon:  \\   
{\textit{ “The density of the Brown measure has one of the following two types of behavior
around each point on the boundary of its support - either (i) sharp cut,   or (ii) quadratic decay at certain critical points on the boundary’’}}. \\This dichotomy can also be observed in the deformed model  \eqref{defmat} under certain restriction  on  the initial matrix $X_0$.

  {\bf Notation.} For arbitrary   sequences
  of complex numbers   $x_N$ and of real positive numbers $b_N$, the notation $x_N=O(b_N)$   means that   $|x_N|/ b_N \leq C$ for some   positive constant $C$, while   $x_N=o(b_N)$   means that $ x_N/b_N \to 0$ as $ N \to  +\infty$.

 %https://arxiv.org/pdf/1202.0644.pdf 
%Theorem 1.4 has Brown measure formula for the case g is a normal operator which is your case
%my paper 
%Brown measure of the sum of an elliptic operator and a free random variable in a finite von Neumann algebra
% has Brown measure valide for general case x+c,  where x could be non-normal. Besides mentioning the support, maybe you could also mention existence of Brown measure density formula.
%my other paper
%The Brown measure of a sum of two free random variables, one of which is triangular elliptic
%Serban Belinschi, Zhi Yin, Ping Zhong
%Section 7 has additional properties of x+c: the density of the Brown measure of x+c has an optimal upper bound 1/pi; and the Brown measure is absolutely continuous
%Erdos-Ji's work was heavily built on my papers 2 3. They played with the Brown measure formula and studied the edge behavior. 

%%%%%%%%%%%%

\subsection{Main results}

Throughout the present paper, we just    specify  $X_0$  as a diagonal   form, although we  can select it  as a normal matrix  due to the invariance 
    of the GinUE matrix,  
\begin{equation}\label{A0 form}
X_0:={\rm diag}\left(a_1\mathbb{I}_{r_1},\cdots,a_t\mathbb{I}_{r_t},
z_0\mathbb{I}_{r_0},A_{t+1}
\right), 
\end{equation}
where $t$ is a fixed  non-negative integer, $z_0$ is a spectral parameter, 
and $z_0, a_1,\cdots, a_t$ are distinct complex numbers, all of which  are independent of $N$. $A_{t+1}$ is a normal $r_{t+1}\times r_{t+1}$ matrix independent of $N$ such that it does not include $z_0$ and $\{ a_{\alpha}|\alpha=1,\cdots,t \}$ as eigenvalues. Here, $r_0$ and $r_{t+1}$ are fixed non-negative integers which are independent of $N$. We also assume that for $\alpha=1,\cdots,t,$
 \begin{equation}\label{ralpha N}
r_{\alpha}=c_{\alpha}N+R_{\alpha,N},
\quad R_{\alpha,N}=O(1),
\end{equation}
where $c_{\alpha}$ is positive and independent of $N$, while $R_{\alpha,N}$ may depend on $N$. 
Obviously,  we have 
\begin{equation}
\sum_{\alpha=1}^{t} c_{\alpha} N+
\sum_{\alpha=1}^{t} R_{\alpha,N}=N-(r_0+r_{t+1}),
\end{equation}
and %As $\sum_{\alpha=1}^{t} R_{\alpha,N}+r_0+r_{t+1}=O(1)$, we must have
\begin{equation}\label{calpha sum}
\sum_{\alpha=1}^t  c_{\alpha}=1, \quad 
\sum_{\alpha=1}^{t} R_{\alpha,N}=-(r_0+r_{t+1}).
\end{equation}

For any given complex numbers $a_1,\cdots,a_t$, let's  introduce  a  probability measure  on the complex plane, which is just the limiting spectral measure of $X_0$, defined   by
\begin{equation}\label{nu}
d\nu(z) :=    \sum_{\alpha=1}^t c_{\alpha} \delta(z-a_{\alpha}).
\end{equation}
Set  
\begin{equation}\label{parameter}
P_{00}(z_0):= \int  \frac{1}{|z-z_0|^2} d\nu(z), \quad P_0\equiv P_{0}(z_0):=  
  \int  \frac{z-z_0}{|z-z_0|^4} d\nu(z),
 \end{equation}
%then for   the limiting spectral measure of the complex random matrix $X$ defined in \eqref{defmat} denoted by $\mu_{\infty}$, 
% the support of $\mu_{\infty}$ is
%\begin{equation}\label{supp mu} 
%\mathrm{Supp}(\mu_{\infty})=\Big\{z_0\in \mathbb{C}: P_{00}(z_0)\geq 1 \Big\},
%\end{equation}
%with  the boundary curve  characterized by the equation $P_{00}(z_0)=1$,
%see e.g.  \cite[Proposition 1.2]{BC16}.
and \begin{equation}\label{parameter2}
 % P_0\equiv P_{0}(z_0):=  \int  \frac{z-z_0}{|z-z_0|^4} d\nu(z)
 P_1\equiv P_{1}(z_0):=  
  \int  \frac{1}{|z-z_0|^4} d\nu(z), \quad 
   P_2\equiv P_{1}(z_0):=  
  \int  \frac{(z-z_0)^2}{|z-z_0|^6} d\nu(z),\quad \chi=\frac{P_{2}(z_0)}{P_{1}(z_0)}.
\end{equation}
These    quantities  are  crucial  to characterize   exact functional form  of the  limiting  spectral measure for the deformed matrix $X$ in Definition  \ref{defmat} and  to distinguish spectral   properties.  For instance,   assuming $\tau=1$ for convenience,   the support of  the limiting spectral measure  $\mu_{\infty}$ for $X$ 
\begin{equation} \mathrm{Supp}(\mu_{\infty}):=\Big\{z_0\in \mathbb{C}: P_{00}(z_0)\geq 1 \Big\};\end{equation}
see e.g.  \cite[Proposition 1.2]{BC16}. %\cite{BC16}, 
We will  verify the following  conclusions  in  the present and subsequent papers.   
\begin{itemize}
\item [(i)] {\bf Bulk  universality}   holds at every point in  the bulk regime  
\begin{equation}  \Big\{z_0\in \mathbb{C}: P_{00}(z_0)> 1\Big\}. \end{equation}
%as   that  in GinUE.
\item [(ii)]  {\bf There are only two types of  edge statistics}  on the  spectral  boundary:    
%\begin{equation} \{z_0\in \mathbb{C}: P_{00}(z_0)= 1 \}, \end{equation}
{\bf GinUE    edge statistics}  at the regular  edge  point $z_0$ such that  $P_{00}(z_0)= 1$ and $P_{0}(z_0) \neq  0$, and   
  {\bf critical  edge statistics} at the critical edge point  $z_0$ such that  $P_{00}(z_0)= 1$ and  $P_{0}(z_0) =  0$.  The latter is a new  kind of point process, but  seems not determinantal; see   \cite{EJ}  for   the same classification of spectral edges   for  
the density of Brown measure of free circular Brownian motion.

\end{itemize}    It seems most interesting  to investigate  local statistics   
at the  critical  edge,  so we focus on it    in the present paper,  and leave     GinUE    edge statistics     and bulk universality in the subsequent papers.

To state  our main results we need to introduce a matrix integral.  
 For an $n\times n$ strictly upper triangular matrix $T$, an $n\times n$ complex matrix $Y$ and an $n\times r_0$ complex matrix $W$,  define a function 
 \begin{multline}\label{FTGY}
F(W,T,Y)=-\hat{\tau}{\rm Tr}(WW^*)-{\rm Tr}\Big(
K_2(T,W)
\Big(
\frac{1}{2}K_2(T,W)+K_1(T)
\Big)
\Big)
+{\rm Tr}\big( K_3(T) \big)
\\
-\frac{1}{2}{\rm Tr}(YY^*)^2
+
 \overline{\chi}{\rm Tr}\big( Y^*Y\hat{Z}^2 \big)
 +{\rm Tr}\big( Y\hat{Z}Y^*\hat{Z}^* \big)
 +\chi{\rm Tr}\Big( YY^*\big(\hat{Z}^*\big)^2 \Big)
 +\hat{\tau}{\rm Tr}\big( YY^* \big),
\end{multline}
where   $\hat{\tau}\in \mathbb{R}$,   $\chi\in \mathbb{C}$  and 
\begin{equation}\label{K1T}
\begin{aligned}
K_1(T)&=\overline{\chi} T^2  + \chi (T^*)^2+T^*T +\big( \hat{Z}+\chi \hat{Z}^* \big)^*T +T^* \big( \hat{Z}+\chi \hat{Z}^* \big) 
    \\
&+
\frac{1}{2}T\big( \hat{Z}+\chi \hat{Z}^* \big)^*+\frac{1}{2}\big( \hat{Z}+\chi \hat{Z}^* \big)T^*
+\chi(\hat{Z}^*)^2+2\hat{Z}\hat{Z}^*+\overline{\chi}\hat{Z}^2,
%&K_1(T)=\overline{\chi}\Big( \hat{Z}T+\frac{1}{2}T\hat{Z}+T^2 \Big)
%+\chi\Big( T^*\hat{Z}^*+\frac{1}{2}\hat{Z}^*T^*+(T^*)^2 \Big)
%+T^*T     \\
%&+T^*\hat{Z}+\hat{Z}^*T
%+\frac{1}{2}\big( T\hat{Z}^*+\hat{Z}T^* \big)
%+\chi(\hat{Z}^*)^2+2\hat{Z}\hat{Z}^*+\overline{\chi}\hat{Z}^2,
%\chi(\hat{Z}^*)^2+2\hat{Z}\hat{Z}^*+\overline{\chi}\hat{Z}^2
\end{aligned}
\end{equation}
\begin{equation}\label{K2T}
K_2(T,W)=TT^*+WW^*+\frac{1}{2}
\big( 
\hat{Z}+\chi\hat{Z}^*
 \big)T^*+
 \frac{1}{2}
T\big( 
\hat{Z}+\overline{\chi}\hat{Z}^*
 \big)^*,
\end{equation}
and
\begin{equation}\label{K3T}
\begin{aligned}
&K_3(T)=-\frac{1}{4}T
\big( \chi(\hat{Z}^*)^2+2\hat{Z}\hat{Z}^*+\overline{\chi}\hat{Z}^2 \big)
T^* \\
&+(1-|\chi|^2)\Big( \hat{Z}T+\frac{1}{2}T\hat{Z}+T^2 \Big)
\Big( \hat{Z}T+\frac{1}{2}T\hat{Z}+T^2 \Big)^*
-\hat{\tau}TT^*.
\end{aligned}
\end{equation}
%At the same time, we introduce an Hermitian matrix variable G and the matrix variable Y without any symmetry, and define the matrix function 
%with real parameter $\hat{\tau}$ and complex parameter $\chi$ with module less or equal than 1:
%
For $ \hat{Z}={\rm diag}\big(
\hat{z}_1,\cdots,\hat{z}_n
\big), $ we further define a matrix integral 
\begin{multline}\label{I0critical2}
I_n\big( \hat{Z} \big)
=\frac{1}{(\sqrt{2\pi})^n}\frac{ 1}{\pi^{n(n+r_0)+\frac{n(n+1)}{2}}}
\int \int \int
\left(
\det\begin{bmatrix}
\hat{Z} & -Y^*  \\ Y  &\hat{Z}^*
\end{bmatrix}
\right)^{r_0-n}
\\
\times
\det\!\big( h(W,T,Y) \big)\,\exp\{ F(W,T,Y) \}
{\rm d}W {\rm d}T{\rm d}Y,
\end{multline}
where 
\begin{equation}\label{h12}
h(W,T,Y)=\begin{bmatrix}
(Y^*Y)\otimes \mathbb{I}_n+F_{1,1} &
F_{1,2}+\big(
\hat{Z}Y^*+\chi Y^*\hat{Z}^*
\big)\otimes \mathbb{I}_n  \\
-F_{1,2}^*-\big(
\hat{Z}^*Y+\overline{\chi} Y\hat{Z}
\big)\otimes \mathbb{I}_n
& (YY^*)\otimes \mathbb{I}_n+F_{1,1}^*
\end{bmatrix},
\end{equation}
with
\begin{equation}\label{F112}
F_{1,1}=\mathbb{I}_n\otimes (WW^*)
-\overline{\chi}\hat{Z}^2 \otimes \mathbb{I}_n-
\hat{Z}\otimes \big(
\overline{\chi}\hat{Z}+\hat{Z}^*+
T^*+\overline{\chi}T
\big)
\end{equation}
and
\begin{equation}\label{F122}
F_{1,2}=Y^*\otimes \Big(
\chi\big(
\hat{Z}^*+T^*
\big)+\big(
\hat{Z}+T
\big)
\Big).
\end{equation}

Now, we are ready to  formulate the  main result about      limits of  correlation functions near the critical edge point.
\begin{theorem}\label{2-complex-correlation critical2}
Let $R_{N}^{(n)}\left( X_0; z_1,\cdots,z_{n} \right)$ be  the $n$-point correlation function  for the  deformed  ensemble $ {\mathrm{GinUE}}_{N}(X_0)$  with $\tau=1+N^{-\frac{1}{2}}\sqrt{P_1}\hat{\tau}$.  Under  the assumption
\eqref{A0 form} on $X_0$,  let $\chi$ be given in \eqref{parameter2},   if  $z_0$ is  a  critical boundary point, that is,  $P_{00}(z_0)=1$ and 
$P_{0}(z_0)=0$, then 
as $N\to \infty$ we have,    %the following     hold 
uniformly for   all 
$\hat{z}_{1}, \ldots, \hat{z}_{n} $
in a  compact subset of $\mathbb{C}$,
\begin{multline} \label{corre2edgecomplex critical03}
 \big( NP_1 \big)^{-\frac{n}{2}}
R_N^{(n)}\Big(
X_0;z_0+\big( NP_1 \big)^{-\frac{1}{4}}\hat{z}_1
,\cdots,
z_0+\big( NP_1 \big)^{-\frac{1}{4}}\hat{z}_n
\Big)
 \\ =\prod_{1\leq i<j\leq n}\big| \hat{z}_i-\hat{z}_j \big|^2\,
e^{-\frac{1}{2}{\rm Tr}
\big(
\overline{\chi}\hat{Z}^2+2\hat{Z}\hat{Z}^*+\chi(\hat{Z}^*)^2
\big)^2+\overline{\chi}{\rm Tr}(\hat{Z}^3\hat{Z}^*)
+\frac{3}{2}{\rm Tr}\big(( \hat{Z}\hat{Z}^*)^2\big) 
+\chi{\rm Tr}\big( \hat{Z}(\hat{Z}^*)^3 \big)
}
\\
\times
e^{-\frac{1}{2}n\hat{\tau}^2
-\hat{\tau}{\rm Tr}\big(
\chi (\hat{Z}^*)^2+\hat{Z}\hat{Z}^*
+\overline{\chi}\hat{Z}^2
\big)}
I_n\big(\hat{Z} \big)+O\big(
N^{-\frac{1}{4}}
\big).
\end{multline}

\end{theorem}

         Several comments  on the above theorem %and future questions 
         can be  listed    as follows.   
         
%         \begin{remark}
%If we turn to consider the finite-rank perturbation effect of the critical point $z_0$, that is,    $r_0>0$ in  \eqref{A0 form}, then 
%the same  result holds as in \eqref{corre2edgecomplex critical02} after similar calculations as in the case $z_0$.  So  the essential phenomenon  near  the critical  spectral edge   is  a point process whose correlation functions are   characterized by  the family of multivariate functions   on the RHS of  \eqref{corre2edgecomplex critical03}.
%\end{remark}

\begin{example}
As an interesting  example, for $a\geq 0$ set  \begin{equation} X_0=\mathrm{diag}\big(\overbrace{a, \ldots,a}^{N/2}, \overbrace{-a, \ldots,-a}^{N/2}\big) \end{equation}
 with   an
equal number of  $\pm a$,     then $t=2, c_1=c_2=1/2, a_1=-a_2=a$  in \eqref{nu},   so solve  the two equations
     \begin{equation} \frac{1}{2|z_0-a|^2} +\frac{ 1}{2|z_0+a|^2} =1,   \quad 
\frac{ a-z_0}{2|z_0-a|^4} +\frac{ -a-z_0}{2|z_0+a|^4} =0,
%\quad P_1=\sum_{\alpha=1}^t \frac{c_{\alpha}}{|z_0-a_{\alpha}|^2} \geq 1,
\end{equation}
  and we see that    there is a critical point $z_0=0$  only   when 
   $a=1$. This  is  a simplest example in which  the critical edge point  appears. 
   Actually,   as $a$ goes from zero  to infinity,  it    illustrates an evolution of the eigenvalue density %initiated from  one   
   to two   disconnected 
  domains which  separate   at some critical value.   
  This example was previously investigated in \cite{BGN15}, where a novel universality pattern emerged in the vicinity of the spectral separation for the averaged product of two characteristic polynomials. Notably, this universality class shares the same scaling laws as those established in Theorem \ref{2-complex-correlation critical2}. To our knowledge, this  constitutes the first detectable signature of a new universal statistics in non-Hermitian random matrices. However, the analysis  there focused exclusively on the averaged product of two characteristic polynomials, leaving more physically relevant correlation functions including even the one-point level density unexplored. Crucially, our methodology extends beyond these limitations: it enables the systematic study of asymptotic behaviors for averaged products of arbitrarily many characteristic polynomials, as demonstrated in the broader framework of Theorem  \ref{2-complex-correlation critical2}.
  %  This example has been investigated  in  \cite{BGN15} and  
%  a novel  universality pattern    arises  in the vicinity of the  separation for  averaged product of two characteristic polynomials, with the same scalings as  in the above theorem.  As far as we know,  this  visible  signal  of a new universal statistics  in non-Hermitian random matrices   was first observed  in  \cite{BGN15}. However, there  only the quantity  of  averaged product of two characteristic polynomials  is well studied,  more interesting correlation functions,    even for  one-point level density, are not investigated. Actually, our method  can be used to study asymptotics for the  averaged product of any finite   characteristic polynomials in more examples as in  Theorem 
% \ref{2-complex-correlation critical2}.    % , in which vanishing  at  a quadratic rate.  
\end{example}

  \begin{figure}\label{RealComplexGrapph2}
\subfigure[$z_0=0$, $X_{0}=\mathrm{diag}\big(\mathbb{I}_{N/2}, -\mathbb{I}_{N/2}\big)$, $P_1=1$]{
\label{ComplexPerturb4Sim}
\includegraphics[width=8.1cm,height=5.1cm]{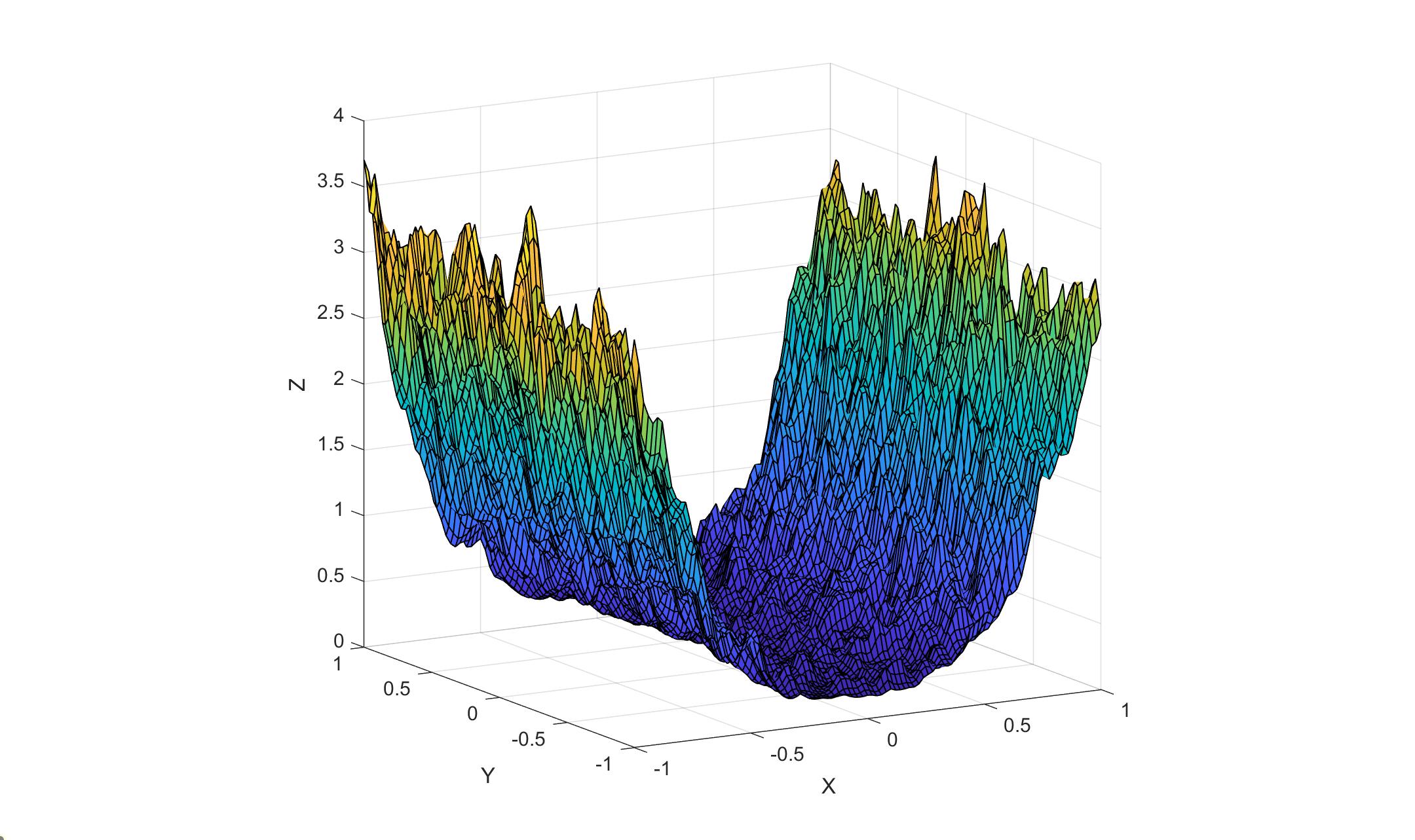}
}\ \ 
\subfigure[$\hat{\tau}=0$, $\chi=1$, $\hat{z}=x+{\rm i}y$]{
\label{ComplexNonPerturb4}
\includegraphics[width=8.1cm,height=5.1cm]{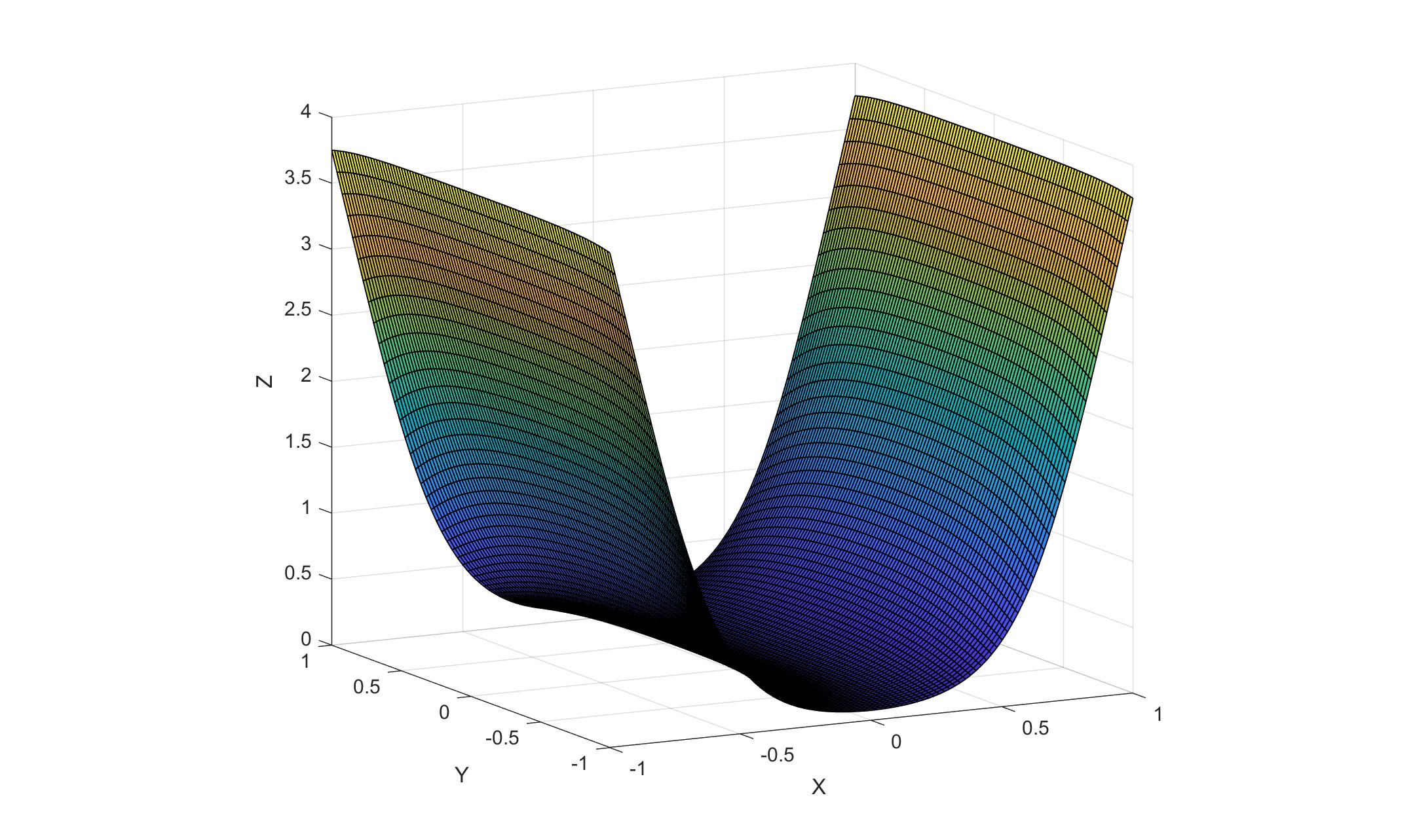}
}

\caption{Plot (a) shows the simulation surface graph of the microscopic rescaled level density $N^{-\frac{1}{4}}R_N^{(1)}\big(
X_0;z_0+N^{-\frac{1}{4}}\hat{z}
\big)$ near the point $z_0=0$ with 5000 samples and $N=3000$, and with $X_{0}=\mathrm{diag}\big(\mathbb{I}_{N/2}, -\mathbb{I}_{N/2}\big)$. Plot (b) shows the surface graph of the  limiting one-point function 
%$e^{-\frac{1}{2}(\hat{z}+\overline{\hat{z}})^4+\hat{z}^2 |\hat{z}|^2+\frac{3}{2}|\hat{z}|^4+\overline{\hat{z}}^2|\hat{z}|^2}I_1(\hat{z}) $ in \eqref{I0critical2}
  with $r_0=0$ and $\hat{z}=x+{\rm i}y$,  as a limit of Plot (a).}
\end{figure}
\begin{remark}
In the special case of $n=1$, the limiting one-point function reduces to    

\begin{equation}
e^{-\frac{1}{2}(\bar{\chi}\hat{z}^2+2|\hat{z}|^2+  \chi \overline{\hat{z}}^2)^2+|\hat{z}|^2 (\chi\bar{\hat{z}}^2 +\frac{3}{2}|\hat{z}|^2+\bar{\chi} \hat{z}^2)} e^{-\frac{1}{2}\hat{\tau}^2-\hat{\tau}(\chi\bar{\hat{z}}^2 +|\hat{z}|^2+\bar{\chi} \hat{z}^2)} I_1(\hat{z}),
 \end{equation}
 where
 \begin{multline}
 I_1(\hat{z})=\frac{1}{\sqrt{2\pi}\pi^{r_0+2}}
 \int \int
 \Big(
 \big||y|^2+WW^*-2\overline{\chi}\hat{z}^2-|\hat{z}|^2\big|^2
 +4|y|^2|\chi\overline{\hat{z}}+\hat{z}|^2
 \Big)
 \\
 \times
 \big(
 |\hat{z}|^2+|y|^2
 \big)^{r_0-1}
 \exp\!\Big\{ 
 -\frac{1}{2}(WW^*)^2-  (\hat{\tau}+\chi\bar{\hat{z}}^2 +2|\hat{z}|^2+\bar{\chi} \hat{z}^2)WW^*
  \Big\}\\
  \times \exp\!\Big\{ 
 -\frac{1}{2}|y|^4+(\hat{\tau}+\chi\bar{\hat{z}}^2 +|\hat{z}|^2+\bar{\chi} \hat{z}^2)|y|^2
  \Big\}
{\rm d}W {\rm d}y,
%e^{-\frac{1}{2}(\bar{\chi}\hat{z}^2+2|\hat{z}|^2+  \chi \overline{\hat{z}}^2)^2+|\hat{z}|^2 (\chi\bar{\hat{z}}^2 +\frac{3}{2}|\hat{z}|^2+\bar{\chi} \hat{z}^2)} e^{-\frac{1}{2}\hat{\tau}^2-\hat{\tau}(\chi\bar{\hat{z}}^2 +|\hat{z}|^2+\bar{\chi} \hat{z}^2)},
 \end{multline}
 with complex number $y$ and $1\times r_0$ complex matrix $W$.  See Figure 1 for the finite-size level density  and its limit.
\end{remark}

\begin{remark}

In principle we  should   also   deal with certain     non-normal matrix perturbation, for instance,  \begin{equation}X_0=\mathrm{diag}\big(\overbrace{J_{2}(a), \ldots,J_{2}(a)}^{N/4}, \overbrace{J_{2}(-a), \ldots,J_{2}(-a)}^{N/4}\big),  \end{equation}
 where   the $2\times 2$ Jordan block     
 \begin{equation} J_{2}(a)=  \begin{bmatrix}
a & 1  \\
0 & a
\end{bmatrix}.\end{equation}  In this case one of the main difficulties  is to verify the  corresponding maximum lemmas in   Section   \ref{Sectmax}. 
Besides, we  can probably  investigate a more general spectral  measure $\nu$ associated with a normal matrix $X_0$ in \eqref{A0 form}.
 
\end{remark}

\begin{remark}
In the special case of $r_0=0$, the matrix variable $W=0,$ and the matrix $h(0,T,Y)$ in  \eqref{h12} can be rewritten as a product of two matrices 
\begin{equation*}
h(0,T,Y)=\begin{bmatrix}
\hat{Z}\otimes \mathbb{I}_n & -Y^*\otimes \mathbb{I}_n
  \\ Y\otimes \mathbb{I}_n  &\hat{Z}^*\otimes \mathbb{I}_n
\end{bmatrix}
\begin{bmatrix}
K_{1,1}
& Y^*\otimes \mathbb{I}_n \\
-Y\otimes \mathbb{I}_n &
K_{1,1}^*
\end{bmatrix}
\end{equation*}
where $K_{1,1}=-\mathbb{I}_n\otimes \big(
\overline{\chi}(\hat{Z}+T)+\hat{Z}^*+T^*
\big)-\overline{\chi}\hat{Z}\otimes \mathbb{I}_n.$
Therefore, if $r_0=0$, in \eqref{I0critical2} we have
\begin{equation*}
\left(
\det\begin{bmatrix}
\hat{Z} & -Y^*  \\ Y  &\hat{Z}^*
\end{bmatrix}
\right)^{-n}
\det\!\big( h(0,T,Y) \big)
=\det\begin{bmatrix}
K_{1,1}
& Y^*\otimes \mathbb{I}_n \\
-Y\otimes \mathbb{I}_n &
K_{1,1}^*
\end{bmatrix}.
\end{equation*}
The integral $I_n(\hat{Z})$ in \eqref{I0critical2} and \eqref{corre2edgecomplex critical03} reduces to
\begin{equation}\label{I0critical20}
I_n\big( \hat{Z} \big)
=\frac{ 1}{(\sqrt{2\pi})^n\pi^{n^2+\frac{n(n+1)}{2}}}
\int \int
\det\begin{bmatrix}
K_{1,1}
& Y^*\otimes \mathbb{I}_n \\
-Y\otimes \mathbb{I}_n &
K_{1,1}^*
\end{bmatrix}\,\exp\{ F(0,T,Y) \}
{\rm d}T{\rm d}Y,
\end{equation}
where $F(W,T,Y)$ is defined in \eqref{FTGY}.

\end{remark}
%      \begin{figure}[h]
%%\vspace{1cm}
%\begin{center}
%%\begin{minipage}[c]{0.5\textwidth}
%%\centering\includegraphics[width=3.5cm,height=3.5cm]{a0.7.jpg}%ŸÍÔÚÇ°ÃæÀšºÅÖÐÐŽÍŒÆ¬Ãû
%%%\caption{$\theta=1.2,1,2,1$}
%%\renewcommand{\figurename}{ass}
%%\label{}
%%\end{minipage}%
%%\begin{minipage}[c]{0.5\textwidth}
%%\centering\includegraphics[width=3.5cm,height=3.5cm]{a1.2.jpg}%ŸÍÔÚÇ°ÃæÀšºÅÖÐÐŽÍŒÆ¬Ãû
%%%\caption{$\theta=1,1,1$}
%%\renewcommand{\figurename}{ass}
%%\end{minipage}
% \begin{minipage}[c]{0.5\textwidth}
%\centering\includegraphics[width=3.5cm,height=3.cm]{a1.jpg}%ŸÍÔÚÇ°ÃæÀšºÅÖÐÐŽÍŒÆ¬Ãû
%%\caption{critical}
%\renewcommand{\figurename}{ass}
%\end{minipage}
%\end{center}
%\end{figure}

 \subsection{Proof sketch and article structure% of the proof
 } \label{sec1.3}
%\qquad
Combining     integral representations  for the  $n$-point    correlation functions,   established  in our previous paper \cite[Proposition 1.3]{LZ}, with    duality formulas for     auto-correlation functions of characteristic polynomials in \cite{Gr, LZ240},  we  reformulate  the matrix integral on the right-hand side (RHS for short)  of \eqref{algebraequa}  below  
into a tractable analytical framework. This restructured form is now amenable to asymptotic analysis through Laplace's method, as detailed in  Proposition   \ref{foranlysis} in Section  \ref{sect2} where   the function $f(\boldsymbol{T},Y,\boldsymbol{Q})$
 – though intricate in structure – plays a pivotal role.

Following the verification of two key maximum principles for   $f(\boldsymbol{T},Y,\boldsymbol{Q})$ in Section   \ref{Sectmax}, we reduce the integral’s dominant contribution to a localized region (Proposition \ref{RNn delta}). This localization, governed by the constrained integration domain in Eq.  \eqref{integrationregionprop}, necessitates a precise Taylor expansion of $f(\boldsymbol{T},Y,\boldsymbol{Q})$, whose technical implementation constitutes the core challenge addressed in Section  \ref{sect3.1}.

 The crucial final step involves analyzing the reduced matrix integral  $I_{N,\delta}$ in  Proposition \ref{RNn delta}.
 To unravel the complexity of highly non-trivial structure for $f(\boldsymbol{T},Y,\boldsymbol{Q})$, we devise a series of changes of matrix variables. These recast the original integration domain  \eqref{integrationregionprop}  into a simplified parametrization with minimal degrees of freedom; see 
  \eqref{integral region critical non0},   \eqref{S critical non0} and \eqref{R critical non0} in  Section  \ref{sect3i}. Central to this reformulation are the  
   triangular matrices  $\{T_{\alpha}\}$, equivalently,  
    diagonal matrices  $\{T_{{\rm d},\alpha}\}$ and   strictly upper triangular  matrices  $\{T_{{\rm u},\alpha}\}$ as in \eqref{Talpha 2}. 
    Through the sequential  change of variables   
\begin{equation} 
T_{{\rm d},1},  T_{{\rm d},2}, \ldots,  T_{{\rm d},t} \longmapsto  \widetilde{T}_{{\rm d},1}, \widetilde{T}_{{\rm d},2},\ldots, \widetilde{T}_{{\rm d},t}   \longmapsto  \Gamma_1, \widetilde{T}_{{\rm d},2},\ldots, \widetilde{T}_{{\rm d},t} 
%T_{{\rm u},\alpha}=\frac{1}{\sqrt{f_{\alpha}}}\hat{T}_{{\rm u},\alpha},    \quad \alpha=1,\cdots,t.
\end{equation}
and 
\begin{equation}
T_{{\rm u},1}, T_{{\rm u},2}, \ldots,  T_{{\rm u},t}   \longmapsto \hat{T}_{{\rm u},1},    \hat{T}_{{\rm u},2}, \ldots, \hat{T}_{{\rm u},t} \longmapsto G_1, G_2, \ldots, G_t  \longmapsto G_1, \hat{G}_2, \ldots, G_t,
\end{equation}
we correspondingly  do very fine-tuning  in $f(\boldsymbol{T},Y,\boldsymbol{Q})$, like  

\begin{equation}
\begin{aligned}
(G_1, Y) &\mapsto N^{-\frac{1}{4}}(G_1, Y), \quad
(\widehat{G}_2, \Gamma_1) \mapsto N^{-1}(\widehat{G}_2, \Gamma_1), \quad Q_{t+1} \mapsto N^{-\frac{1}{2}}Q_{t+1}\\
%Q_{t+1} &\mapsto N^{-\frac{1}{2}}Q_{t+1}, \\
\widetilde{T}_{d,\alpha} &\mapsto N^{-\frac{1}{2}}\widetilde{T}_{d,\alpha} \,(\alpha = 2, \ldots, t), 
G_\beta \mapsto N^{-\frac{1}{2}}G_\beta \,(\beta = 3, \ldots, t).
\end{aligned}
\end{equation}

%This will  be a very long journey  of calculation 
%and  may    indicate one of the reasons   why   critical statistics at the spectral edge is  so complicated. 

This elaborate analytical process not only underscores the profound technical challenges inherent in the edge-critical statistics,  but also illuminates a fundamental source of their complexity. Our systematic resolution unfolds through three rigorously structured   subsections: 
 In {\bf Section \ref{sect3.1}},   we do a Taylor expansion of the function $f(\boldsymbol{T},Y,\boldsymbol{Q})$ in  \eqref{fTY}   in seven  steps,  with a lot of changes of variables, notations    and calculations.    
  In {\bf Section \ref{sect3.2}} we do Taylor expansion of a determinant $\det\big(\widehat{L}_1+\sqrt{\gamma_N}\widehat{L}_2\big)$. 
  In {\bf Section \ref{sect3.3}}, combine the Taylor expansions  in   {\bf Section \ref{sect3.1}}  and  {\bf Section \ref{sect3.2}},  we   give a  complete proof.

\section{Concentration reduction} \label{sect2}
%%%%%%%%%%%%%%%%%%%%%%%%%%%%%%%%%%%%%%%%%%%%%%%%%%%%%%%%%%%%%%%%%%%%%%%%%55

%%%%%%%%%%%%%%%%%%%%%%%%%%%%%%%%%%%%%%%%%%%%%%%%%%%%%%%%%%%%%%%%%%%%%%%%%55
\subsection{Notation} \label{sectnotation}

%\quad\ \;
Throughout the proof of Theorem \ref{2-complex-correlation critical2}, the sizes of all matrices and matrix variables will be chosen to be finite, independent of $N$.

  Denote  the conjugate, transpose,  and  conjugate transpose of a complex matrix $A$ by  $\overline{A}$,  $A^t$ and $A^*$  respectively.   Besides, the following symbols will be used in the subsequent sections.  
Denote the tensor product of an $m\times n$  matrix  $A:=[a_{i,j}]$  and  a $p\times q$ matrix  $B$ by a block matrix 
$$
A \otimes B=\left[\begin{smallmatrix}
a_{11}B   &\cdots  &  a_{1n} B   \\ \vdots  & \ddots &   \vdots   \\
a_{m1}B   &\cdots  &  a_{mn} B   
\end{smallmatrix}\right],$$
and denote the Hilbert-Schmidt norm of a complex matrix $M$ by $\|M\|:=\sqrt{{\rm{Tr}}(MM^*)}$. The notation $M=O(A_N)$ for some number sequence $A_N$  means that each matrix element of $M$ has the same order  of $ A_N$. In other words, there exists a positive constant $C$ independent of $N$, such that $|M_{i,j}|\leq C A_N$.

%%%%%%%%%%%%%%%%%%%%%%%%%%%%%%%%%%%%%%%%%%%%%%%%%%%%%%%%%%%%%%%%%%%%%%%%%%%%%5
\subsection{Translation reduction} \label{Translation reduction}
\quad\ \;First, recall the assumptions \eqref{defmat} and \eqref{A0 form}. By taking a translation transformation of the random matrix $X$, that is, $
X\mapsto X-a_1\mathbb{I}_N,$ the eigenvalues $\lambda_1,\cdots,\lambda_N$ of $X$ change to $\lambda_1-a_1,\cdots,\lambda_N-a_1$. Combining the definition of the $n$-point correlation function $R_N^{(n)}$ in \eqref{Correlation functionDef}, we  have
\begin{equation}\label{CorrelationTranslationReduction}
R_N^{(n)}(X_0;z_1,\cdots,z_n)=R_N^{(n)}(X_0-a_1\mathbb{I}_N;z_1-a_1,\cdots,z_n-a_1).
\end{equation}

Moreover,  %as $N$ tends to infinity, 
if $z_0$ is an edge point of the limiting spectral measure of the random matrix $X$, then $z_0-a_1$ is an edge point of the limiting spectral measure of the random matrix $X-a_1\mathbb{I}_N$. Therefore, without loss of generality, in \eqref{A0 form} we can assume that 
\begin{equation}\label{CorrelationTranslationReduction0}
a_1=0.
\end{equation}
We will see that   this assumption  ensures  a key  integral representation of the $n$-correlation function; cf. Proposition \ref{intrep} below.  

\subsection{Matrix variables} \label{sectMatrixvariables}
\quad\ \;To prove Theorem \ref{2-complex-correlation critical2}, we need a matrix integral representation for the correlation function. To this end, recalling \eqref{A0 form}, we need to introduce several matrix variables $Y$, $Q_0$, $Q_{t+1}$ and the upper triangular matrix variable $T_{\alpha}$ with non-negative diagonal elements for $\alpha=1,\cdots,t$, which take values in $\mathbb{C}^{n\times n}$, $\mathbb{C}^{n\times r_0}$, $\mathbb{C}^{n\times r_{t+1}}$ and $\mathbb{C}^{n\times n}$ respectively.

Also, for the matrix variables $Y$, $Q_0$ and $Q_{t+1}$, denote
\begin{equation}
Y:=[y_{i,j}]_{i,j=1}^n,\quad
Q_0:=\big[q^{(0)}_{a,b}\big], \quad Q_{t+1}:=\big[q^{(t+1)}_{c,d}\big],
\end{equation}
where $ a,c=1,\cdots,n, b=1,\cdots,r_0,   d=1,\cdots,r_{t+1}$.
 On the other hand, rewrite $T_{\alpha}$ as a sum of  a diagonal matrix $\sqrt{T_{{\rm d},\alpha}}$ and a strictly upper triangular  matrix ${T_{{\rm u},\alpha}}$:
\begin{equation}\label{Talpha 2}
T_{\alpha}=\sqrt{T_{{\rm d},\alpha}}+T_{{\rm u},\alpha},
\end{equation} 
where
\begin{equation}\label{Tdalpha}
T_{{\rm d},\alpha}:={\rm diag}
\big(
t_{1,1}^{(\alpha)},\cdots,t_{n,n}^{(\alpha)}
\big)\geq 0\ \ 
\text{and}
\ \ T_{{\rm u},\alpha}:=\left[\begin{smallmatrix}
0 & t_{1,2}^{(\alpha)} & \cdots & t_{1,n}^{(\alpha)} \\
& 0 & \cdots & t_{2,n}^{(\alpha)}  \\
&& \ddots & \vdots \\
&&& 0 
\end{smallmatrix}\right].
\end{equation}

Further, denote the relevant volume elements by
\begin{multline}\label{VolumeElement1}
{\rm d}Y=\prod_{i,j=1}^n {\rm d}^2 y_{i,j},
{\rm d}Q_0=\prod_{a=1}^n\prod_{b=1}^{r_0} {\rm d}^2 q^{(0)}_{a,b},
{\rm d}Q_{t+1}=\prod_{c=1}^n\prod_{d=1}^{r_{t+1}} {\rm d}^2 q^{(t+1)}_{c,d}, {\rm d}T_{\alpha}=\prod_{i=1}^n {\rm d}t_{i,i}^{(\alpha)}
\prod_{i<j}^n {\rm d}^2 t_{i,j}^{(\alpha)},
%\\{\rm d}T_{\alpha}=\prod_{i=1}^n {\rm d}t_{i,i}^{(\alpha)}
%\prod_{i<j}^n {\rm d}^2 t_{i,j}^{(\alpha)},\quad
%\alpha=1,\cdots,t,
\end{multline}
where  $\alpha=1,\cdots,t$, and  for a complex variable $y:=\Re y+\sqrt{-1}\Im y$ the volume element 
${\rm d}^2 y:={\rm d}\Re y{\rm d}\Im y$ is the Lebesgue measure on the complex plane.

Later, we need to use the singular value decomposition for the matrix variable $Y$: 
\begin{equation}\label{singularcorre}
Y=U\sqrt{\Lambda}V,\quad
\Lambda={\rm diag}\left( \lambda_1,\cdots,\lambda_n \right),\quad
\lambda_1\geq \cdots \geq \lambda_n \geq 0,
\end{equation}
and the Jacobian determinant reads 
\begin{equation}\label{singularjacobian}
{\rm d}Y=\pi^{n^2}\Big( \prod_{i=1}^{n-1}i! \Big)^{-2}
\prod_{1\leq i<j\leq n}(\lambda_{j}-\lambda_{i})^2 
{\rm d}\Lambda{\rm d}U{\rm d}V,
\end{equation}
where 
$U$ and $ V$ are chosen from the unitary group  $\mathcal{U}(n)$ with the Haar probability measures ${\rm d}U_1$ and ${\rm d}U_2$ respectively.

Also for simplicity, we will write the bold-faced symbols $\boldsymbol{T}$ and $\boldsymbol{Q}$ to represent the matrix variables $(T_1,\cdots,T_t)$ and $(Q_0,Q_{t+1})$ respectively. Finally, introduce 
\begin{multline}\label{integrationregionprop}
\Omega:=\bigg\{ \big(\boldsymbol{T}, Y, \boldsymbol{Q}\big)\Big|
T_{\alpha},Y\in \mathbb{C}^{n\times n},\ Q_0\in \mathbb{C}^{n\times r_0},\ 
Q_{t+1}\in \mathbb{C}^{n\times r_{t+1}},\ 
\\
T_{\alpha}\ \text{satisfying}\ \eqref{Talpha 2}\ \text{and}\ \eqref{Tdalpha}
\ \text{for}\ \alpha=1,\cdots,t,\ \ 
\sum_{\alpha=1}^t
T_{\alpha}T_{\alpha}^*+Q_0Q_0^*+Q_{t+1}Q_{t+1}^*\leq \mathbb{I}_n
 \bigg\}.
\end{multline}

\subsection{Integral representation} \label{sectIntegral representation}
\quad\ \;Firstly, introduce the following scaled spectral points:% which play the roles as parameters in the $n$-point correlation function:
\begin{equation}\label{zi}
 Z:=\mathrm{diag}(z_1, \ldots,z_n)=z_0 I_{n}+N^{-\frac{1}{4}}\hat{Z}, \ \ \hat{Z}:=\mathrm{diag}(\hat{z}_1, \ldots,\hat{z}_n).
\end{equation} 
Also, for $\alpha=1,\cdots,t$, we introduce  some  $2n\times 2n$ matrices  
\begin{equation}\label{A alpha}
E_{\alpha}:=
\begin{bmatrix}
\sqrt{\gamma_N}(Z-a_{\alpha}\mathbb{I}_n)  &
-Y^* \\
Y & \sqrt{\gamma_N}
\big(Z^*-\overline{a}_{\alpha}\mathbb{I}_n
\big)
\end{bmatrix},
\end{equation}
where \begin{equation}\label{yammaN}
\gamma_N:=\frac{N}{N-n}.
\end{equation}
Denote the $n\times \widetilde{r}_{t+1}$ matrix $\widetilde{Q}_{t+1}$ and the $\widetilde{r}_{t+1}\times \widetilde{r}_{t+1}$ matrix $\widetilde{A}_{t+1}$ by
\begin{equation}\label{Q convenience} 
\widetilde{Q}_{t+1}:=\left[ Q_0,Q_{t+1}
\right],\quad
\widetilde{A}_{t+1}:={\rm diag}\left( z_0\mathbb{I}_{r_0},A_{t+1}
\right),\quad  
\widetilde{r}_{t+1}:=r_0+r_{t+1}.
\end{equation}
 Further, denote the $2n(nt+\widetilde{r}_{t+1})\times 2n(nt+\widetilde{r}_{t+1})$ matrices $\widehat{L}_1$ and $\widehat{L}_2$ by
\begin{equation}\label{L1hat}
\widehat{L}_1:={\rm diag}\big(
E_1\otimes \mathbb{I}_n,\cdots,
E_t\otimes \mathbb{I}_n,
\widehat{B}_{t+1}
\big) 
\end{equation}
and
\begin{equation}\label{L0hat}
\widehat{L}_2:=
\left[
\begin{smallmatrix}
%%%%%%%%%%%%%%%%%%%%%%%%%%%%%%%%%%%%%%%%%%%%%%%%%%%%%%%%
\left[
\left[\begin{smallmatrix}
a_{\alpha}\mathbb{I}_n & \\ & \overline{a}_{\alpha}\mathbb{I}_n
\end{smallmatrix}\right]
\otimes
(T_{\alpha}^*T_{\beta}) \right]_{\alpha,\beta=1}^t
&
\left[
\left[\begin{smallmatrix}
a_{\alpha}\mathbb{I}_n & \\ & \overline{a}_{\alpha}\mathbb{I}_n
\end{smallmatrix}\right]
\otimes
(T_{\alpha}^*\widetilde{Q}_{t+1}) \right]_{\alpha=1}^t
   \\
\left[
\begin{smallmatrix}
\mathbb{I}_n \otimes 
(\widetilde{A}_{t+1}\widetilde{Q}_{t+1}^*
T_{\beta})  &  \\
 &  
\mathbb{I}_n  \otimes 
 (\widetilde{A}_{t+1}^*\widetilde{Q}_{t+1}^*
T_{\beta})
\end{smallmatrix}\right]_{\beta=1}^t
&
\left[\begin{smallmatrix}
\mathbb{I}_n  \otimes 
(\widetilde{A}_{t+1}\widetilde{Q}_{t+1}^*\widetilde{Q}_{t+1})
 &  \\
& 
\mathbb{I}_n  \otimes 
(\widetilde{A}_{t+1}^*\widetilde{Q}_{t+1}^*\widetilde{Q}_{t+1})
\end{smallmatrix}\right]
%%%%%%%%%%%%%%%%%%%%%%%%%%%%%%%%%%%%%%%%%%%%%%%%%%%%%%%%%%%%%
\end{smallmatrix}\right],
\end{equation}
where  a $2n\widetilde{r}_{t+1}\times 2n\widetilde{r}_{t+1}$ matrix  
\begin{equation}\label{B0hat}
\widehat{B}_{t+1}:=
\begin{bmatrix}
\sqrt{\gamma_N}
\big(
Z\otimes \mathbb{I}_{\widetilde{r}_{t+1}}
-\mathbb{I}_n\otimes \widetilde{A}_{t+1}
\big)  &
-Y^*\otimes \mathbb{I}_{\widetilde{r}_{t+1}} \\
 Y\otimes \mathbb{I}_{\widetilde{r}_{t+1}} &  \sqrt{\gamma_N}
\big(
Z^*\otimes \mathbb{I}_{\widetilde{r}_{t+1}}
-\mathbb{I}_n\otimes \widetilde{A}_{t+1}^*
\big)
\end{bmatrix}.
\end{equation}

The starting point is    \cite[Proposition 1.3]{LZ},  which is reformulated as follows for convergence. 
\begin{proposition} (\cite{LZ})\label{intrep}
For the model \eqref{defmat}, assume that $X_0:={\rm diag}(A_0,0_{(N-r)\times (N-r)})$, where $A_0$ is any $r\times r$ complex matrix.  
When  $r+n\leq N$,  the $n$-point correlation function    is  given by  
\begin{equation}\label{algebraequa}
\begin{aligned}
&R_{N}^{(n)}\left( X_0; z_1,\cdots,z_{n} \right)=\frac{1}{C_{N,\tau}}%\frac{Z_{N-n,r}}{Z_{N,r}}\frac{N!}{(N-n)!} \Big( \frac{N-n}{N} \Big)^{\frac{1}{2}(N-n)(N+1+n)}
 e^{-\frac{N}{\tau}
\sum_{k=1}^{n}\left| z_k \right|^2} 
\prod_{1\leq i<j\leq n}\left| z_i-z_j \right|^2
\\
&\times  %e^{-\frac{N}{\tau}\sum_{k=1}^{n}\left| z_k \right|^2} 
 \int_{\mathcal{M}_{n,r}} \big(\det(\mathbb{I}_r-Q^*Q) \big)^{N-n-r}
e^{\frac{N}{\tau}h(Q)}\,
{ {\mathbb{E}}_{\mathrm{ GinUE}_{N-n}( \widetilde{X}_0)}}\Big[\prod_{i=1}^n
\Big| \det\Big( \sqrt{\frac{N}{N-n}}z_i-\widetilde{X} \Big) \Big|^2\Big] {\rm d}Q,
\end{aligned}
\end{equation}
where  
 \begin{equation*}\label{Qintegraldomain}
\mathcal{M}_{n,r}=\left\{ Q\in \mathbb{C}^{n\times r} \big| Q^*Q\leq\mathbb{I}_r\right\},
\end{equation*}
\begin{equation}\label{CN}
C_{N,\tau}= 
\tau^{nN-\frac{n(n-1)}{2}}
\pi^{n(r+1)} N^{-\frac{1}{2}n(n+1)} (N-n)^{-n(N-n)}     \prod_{k=N-n-r}^{N-r-1}  k !,
\end{equation}
\begin{equation}\label{OmegaQdenotion}
h(Q)%{\rm Tr}\left( \widetilde{A}_0\widetilde{A}_0^* \right)
%\sum_{i=1}^n \sum_{k,l=1}^rq_{ik}\overline{q_{il}} (\overline{z_i} a_{kl}+ z_{i} \overline{a_{lk}} )- \\
%& \sum_{j=1}^n\sum_{l_1,l_2,k_1,k_2=1}^r \overline{q_{jl_1}}q_{jl_2} a_{k_1l_1}\overline{a_{k_2l_2}} \Big( \delta_{k_1,k_2}-\sum_{i=1}^{j-1} q_{ik_1}\overline{q_{ik_2}} \Big).
%\sum_{i=1}^n\sum_{l_1,l_2,k_1,k_2}^r
%q_{ik_1}\overline{q_{ik_2}}\Big( \delta_{l_1,l_2}-\sum_{j=1}^i\overline{q_{jl_1}}q_{jl_2} \Big)
%a_{k_1l_1}\overline{a_{k_2l_2}}
= %{\rm Tr}\big( QA_0 Q^* Z^*+ QA_{0}^*  Q^* Z\big)
\sum_{i=1}^n\left( z_i\overline{(QA_0 Q^* )_{i,i}}+\overline{z_i}(QA_0 Q^* )_{i,i}\right)+
\sum_{1\leq i<j\leq n}|(QA_0 Q^* )_{i,j}|^2 -{\rm Tr}\big(Q^* Q A_{0}^* A_0\big),
\end{equation} 
and
\begin{equation}\label{DeltaA0}
\widetilde{X}_0:={\rm diag}(\widetilde{A}_0,0_{(N-n-r)\times (N-n-r)}),\ \ 
\text{where}\ \ 
\widetilde{A}_0:=\sqrt{N/(N-n)}\sqrt{\mathbb{I}_r-Q^*Q}A_0\sqrt{\mathbb{I}_r-Q^*Q}.
\end{equation}

\end{proposition}

For the model \eqref{defmat}, under unitary transformation, the eigenvalue statistics of the random matrix $X$ remain unchanged.  Note that 
we have assumed that $a_1=0$  in Section \ref{Translation reduction}, 
so by applying a unitary transformation, we can take the mean matrix $X_0$ defined in \eqref{A0 form} as the following new form:
$$
X_0:={\rm diag}(A_0,0_{n\times n}),
$$
where $A_0$ takes the form
\begin{equation}\label{A0New form}
A_0={\rm diag}\left(a_1\mathbb{I}_{\widetilde{r}_1},\cdots,a_t\mathbb{I}_{r_t},
z_0\mathbb{I}_{r_0},A_{t+1}
\right),\ \ 
\text{where}\ \ \widetilde{r}_1:=r_1-n.
\end{equation}
From \eqref{ralpha N}, we have $\widetilde{r}_1=c_1 N+\widetilde{R}_{1,N}$, where $\widetilde{R}_{1,N}:=R_{1,N}-n$. As the symbol $r_1$ does not appear in \eqref{algebraequa}, for simplicity, we will still use the symbols $r_1$ and $R_{1,N}$ to represent $\widetilde{r}_1$ and $\widetilde{R}_{1,N}$ respectively. According to this convention, in Proposition \ref{intrep} we have
\begin{equation}\label{New r}
r=\sum_{\alpha=0}^{t+1} r_{\alpha}=N-n.
\end{equation}

Now we can apply Proposition \ref{intrep} directly with $r=N-n$ and obtain an integral representation for the $n$-point correlation function. 

\begin{proposition}\label{foranlysis}
\begin{equation}\label{foranlysisequ} 
R_N^{(n)}(X_0;z_1,\cdots,z_n)=
C_{N,n}
\int_{\Omega} \det\Big(
\widehat{L}_1+\sqrt{\gamma_N}\widehat{L}_2
\Big)\exp\{ Nf(\boldsymbol{T},Y,\boldsymbol{Q}) \}
{\rm d}Y
{\rm d}Q_0{\rm d}Q_{t+1} \prod_{\alpha=1}^t
{\rm d}T_{\alpha},
\end{equation}
where $\gamma_N$, the integration domain $\Omega$   
and the volume elements are  respectively given   in \eqref{yammaN}, \eqref{integrationregionprop}     and \eqref{VolumeElement1},  and  \begin{multline}\label{fTY}
f(\boldsymbol{T},Y,\boldsymbol{Q})=\sum_{\alpha=1}^t
c_{\alpha}\log\det(T_{\alpha}T_{\alpha}^*)
+N^{-1}\sum_{\alpha=1}^t
\sum_{j=1}^n (R_{\alpha,N}-j)\log\big(
t_{j,j}^{(\alpha)}
\big)
+\frac{1}{\tau}h(\boldsymbol{T},\boldsymbol{Q})
\\-
(\tau\gamma_N)^{-1}
{\rm Tr}\big(
YY^*
\big)+\sum_{\alpha=1}^tc_{\alpha}\log\det(E_{\alpha})+
N^{-1}\sum_{\alpha=1}^t
(R_{\alpha,N}-n)\log\det\big(
E_{\alpha}
\big),
\end{multline}
with 
\begin{multline}\label{hQrewrite}
h(\boldsymbol{T},\boldsymbol{Q})=
\sum_{\alpha=1}^t\left(-f_{\alpha}{\rm Tr}\big(
T_{\alpha}T_{\alpha}^*
\big)
+N^{-\frac{1}{4}}
\left(
\overline{a}_{\alpha}{\rm Tr}\big(
\hat{Z} T_{\alpha}T_{\alpha}^*
\big)+a_{\alpha}{\rm Tr}\big(
\hat{Z}^* T_{\alpha}T_{\alpha}^*
\big)
\right) 
\right)        \\
+
|z_0|^2\Big(
\sum_{\alpha=1}^t
{\rm Tr}\big(
T_{\alpha}T_{\alpha}^*
\big)+{\rm Tr}\big(
Q_0Q_0^*
\big)+{\rm Tr}\big(
Q_{t+1}Q_{t+1}^*
\big)
\Big)         \\       
+N^{-\frac{1}{4}}
\Big(
\overline{z}_0{\rm Tr}\big(
\hat{Z} Q_0Q_0^*
\big)+z_0{\rm Tr}\big(
\hat{Z}^* Q_0Q_0^*
\big)
+
{\rm Tr}\big(
\hat{Z} Q_{t+1}A_{t+1}^*Q_{t+1}^*
\big)+{\rm Tr}\big(
\hat{Z}^* Q_{t+1}A_{t+1}Q_{t+1}^*
\big)
\Big)  \\
+\sum_{1\leq i < j\leq n}
\Big |
\Big(
\sum_{\alpha=1}^t a_{\alpha}T_{\alpha}T_{\alpha}^*
+z_0Q_0Q_0^*+Q_{t+1}A_{t+1}Q_{t+1}^*
\Big)_{i,j}
\Big|^2          \\
-{\rm Tr}Q_{t+1}
\big(
z_0\mathbb{I}_{r_{t+1}}-A_{t+1}
\big)^*
\big(
z_0\mathbb{I}_{r_{t+1}}-A_{t+1}
\big)
Q_{t+1}^*. 
\end{multline}
Here $R_{\alpha,N}$ is given  in \eqref{ralpha N}, 
\begin{equation}\label{falpha}
f_{\alpha}:=|z_0-a_{\alpha}|^2, \quad \alpha=1,\cdots,t,
\end{equation}
 and 
\begin{equation} \label{norm-1}
C_{N,n}:=
\Big( \frac{N-n}{\pi\tau} \Big)^{n N}
\frac{N^{\frac{n(n+1)}{2}}}{\tau^{\frac{n(n+1)}{2}} \pi^n \prod_{i=1}^{n-1}i!}
e^{-\frac{N}{\tau}\sum_{k=1}^n|z_k|^2}
\prod_{1\leq i <j \leq n}|z_i-z_j|^2
\prod_{\alpha=1}^t
\frac{\pi^{nr_{\alpha}-\frac{n(n-1)}{2}}}
{\prod_{j=1}^n(r_{\alpha}-j)!}.
\end{equation}
\end{proposition}

\begin{proof}
We first deal with the inner expectation. Using of  the duality formula   
in \cite[Eqn(40)]{Gr} or \cite[Corollary 2.8]{LZ240}  gives rise to 
\begin{equation}\label{EGinUE dual}
{ {\mathbb{E}}_{\mathrm{ GinUE}_{N-n}( \widetilde{X}_0)}}\Big[\prod_{i=1}^n
\Big| \det\Big( \sqrt{  \gamma_{N}}z_i-\widetilde{X} \Big) \Big|^2\Big]
=
\Big( \frac{N-n}{\pi\tau} \Big)^{n^2} 
\int\     
 D_Y\ e^{-\frac{N-n}{\tau}{\rm Tr}(YY^*)}{\rm d}Y ,
\end{equation}
where $\gamma_N$ is defined in \eqref{yammaN} and 
 \begin{equation}
 D_Y=\det\begin{bmatrix}
\sqrt {\gamma_{N}}\mathbb{I}_{r}\otimes Z -\widetilde{A}_0\otimes \mathbb{I}_{n} 
 &   -\mathbb{I}_{r} \otimes Y^*    \\
\mathbb{I}_{r} \otimes Y &   
\sqrt{ \gamma_{N}}\mathbb{I}_{r}\otimes Z^* -\widetilde{A}_0^*\otimes \mathbb{I}_{n} 
\end{bmatrix}.
 \end{equation}
Recalling   $\widetilde{A}_0$ in  \eqref{DeltaA0}, multiply 
$\mathbb{I}_2\otimes \big( (\mathbb{I}_r-Q^*Q)^{-\frac{1}{2}}\otimes \mathbb{I}_n \big)$
 and its inverse on both sides, we obtain 
\begin{equation}\label{First}
\begin{aligned}
&D_Y=\det\begin{bmatrix}
\sqrt{  \gamma_{N}}\Big(\mathbb{I}_{r}\otimes Z -A_0(\mathbb{I}_r-Q^*Q)\otimes \mathbb{I}_{n} \Big)
 &   -\mathbb{I}_{r} \otimes Y^*    \\
\mathbb{I}_{r} \otimes Y &   
\sqrt{  \gamma_{N}}\Big(\mathbb{I}_{r}\otimes Z^* -A_0^*(\mathbb{I}_r-Q^*Q)\otimes \mathbb{I}_{n} \Big)
\end{bmatrix} .      
\end{aligned}
\end{equation}
To simplify the above determinant further, we need to divide the matrix  variable $Q$  into  a block form,  according to the structure  of $A_0$ as in \eqref{A0 form}
\begin{equation}\label{Q0 division}
Q=\left[Q_1,\cdots,Q_t
,Q_0,Q_{t+1}
\right],
\end{equation}
where $Q_{\alpha}$ is of size $n\times r_{\alpha}$ for $\alpha=0,1,\cdots,t+1$. %Besides, set \eqref{Q convenience}. 

After elementary matrix operations, we see  from \eqref{First} that  
\begin{equation}\label{L1 add L2}
\begin{aligned}
D_Y  =\det\Big(
L_1+\sqrt{ \gamma_{N}}L_2
\Big),
\end{aligned}
\end{equation}
where
\begin{equation}\label{L1}
L_1={\rm diag}\left( B_{1,r_{1}},
\cdots,B_{t,r_t}, B_{t+1}
\right),
\end{equation}
with
\begin{equation}\label{B alpha}
B_{\alpha,r_{\alpha}}=\begin{bmatrix}
\sqrt{  \gamma_{N}}\big( \mathbb{I}_{r_{\alpha}}\otimes 
(Z-a_{\alpha} \mathbb{I}_n) \big)
& -\mathbb{I}_{r_{\alpha}}\otimes Y^*   \\  \mathbb{I}_{r_{\alpha}}\otimes Y  &
\sqrt{ \gamma_{N} }\big( \mathbb{I}_{r_{\alpha}}\otimes 
(Z^*-\overline{a}_{\alpha}  \mathbb{I}_n) \big)
\end{bmatrix},
\end{equation}
for $\alpha=1,\cdots,t$,
and
\begin{equation}\label{B0}
B_{t+1}=\begin{bmatrix}
\sqrt{  \gamma_{N}}\big( \mathbb{I}_{\widetilde{r}_{t+1}}\otimes 
Z-\widetilde{A}_{t+1}\otimes \mathbb{I}_n \big)
& -\mathbb{I}_{\widetilde{r}_{t+1}}\otimes Y^*   \\  
\mathbb{I}_{\widetilde{r}_{t+1}}\otimes Y  &
\sqrt{  \gamma_{N}}\big( \mathbb{I}_{\widetilde{r}_{t+1}}\otimes 
Z^*-\widetilde{A}_{t+1}^* \otimes \mathbb{I}_n \big)
\end{bmatrix},
\end{equation}
while 
\begin{equation}\label{L2 another}
L_2=
\left[\begin{smallmatrix}
\mathbb{I}_{2n\sum_{\alpha=1}^t r_{\alpha}} && \\
& \widetilde{A}_{t+1}\widetilde{Q}_{t+1}^* \otimes \mathbb{I}_n & \\
&& \widetilde{A}_{t+1}^*\widetilde{Q}_{t+1}^* \otimes \mathbb{I}_n
\end{smallmatrix}\right]
\widetilde{L}_2
\left[\begin{smallmatrix}
\mathbb{I}_{2n\sum_{\alpha=1}^t r_{\alpha}} && \\
& \widetilde{Q}_{t+1} \otimes \mathbb{I}_n & \\
&& \widetilde{Q}_{t+1} \otimes \mathbb{I}_n
\end{smallmatrix}\right],
\end{equation}
with
\begin{equation}\label{tilde L2}
\widetilde{L}_2=\begin{bmatrix}
\left[
\begin{smallmatrix}
a_{\alpha}Q_{\alpha}^*Q_{\beta}\otimes \mathbb{I}_n &  \\
& \overline{a}_{\alpha}Q_{\alpha}^*Q_{\beta}\otimes \mathbb{I}_n
\end{smallmatrix}\right]_{\alpha,\beta=1}^t
&
\left[
\begin{smallmatrix}
a_{\alpha}Q_{\alpha}^*\otimes \mathbb{I}_n &  \\
& \overline{a}_{\alpha}Q_{\alpha}^*\otimes \mathbb{I}_n
\end{smallmatrix}\right]_{\alpha=1}^t   \\
\left[\begin{smallmatrix}
Q_{\beta}
\otimes \mathbb{I}_n &  \\
& Q_{\beta}
\otimes \mathbb{I}_n
\end{smallmatrix}\right]_{\beta=1}^t 
&
\left[\begin{smallmatrix}
\mathbb{I}_n
\otimes \mathbb{I}_n &  \\
& \mathbb{I}_n
\otimes \mathbb{I}_n
\end{smallmatrix}\right]
\end{bmatrix}.
\end{equation}

In order to  simplify $D_Y$ in \eqref{L1 add L2} again,  we use the QR decomposition for  rectangular matrices

\begin{equation}\label{Qalpha QR}
Q_{\alpha}=T_{\alpha} U_{\alpha}, \quad \alpha=1,2,\ldots,t,
\end{equation}
where   $U_{\alpha}$ is  an  $n\times r_{\alpha}$ complex matrix  such that $U_{\alpha}U_{\alpha}^*=\mathbb{I}_n$
and  $T_{\alpha}$ is an $n\times n$ upper triangular matrix with  non-negative  diagonal elements.
Here we assume that all  $r_{\alpha}\geq n$ since  $r_{\alpha}$ tends to infinity as $N \rightarrow +\infty$. 
Equivalently,    $U_{\alpha}$ can be taken from   the first $n$ rows of a $r_{\alpha}\times r_{\alpha}$ unitary matrix $\hat{U}_{\alpha}$ \begin{equation}\label{Qalpha QR 1}
Q_{\alpha}=\big[ T_{\alpha},0 \big]\hat{U}_{\alpha}.
%\quad\hat{U}_{\alpha}\in \mathcal{U}(r_{\alpha});
\end{equation}
%Now we return to the calculations from \eqref{L1 add L2} to \eqref{tilde L2}. 
Hence,    for $\alpha,\beta=1,\cdots,t$,
 \begin{equation}\label{Qalpha time Qbeta}
Q_{\alpha}^*Q_{\beta}=
\hat{U}_{\alpha}^*\begin{bmatrix}
T_{\alpha}^*T_{\beta} & 0 \\
0 & 0
\end{bmatrix}\hat{U}_{\beta},
\end{equation}
from which  we  can rewrite
\begin{small}
\begin{equation}\label{tilde L2 extract}
\begin{aligned}
\widetilde{L}_2=
%&\left[\begin{smallmatrix}
%\left[\begin{smallmatrix}
%\hat{U}_{1}^*  &  \\
% & \hat{U}_{1}^*
%\end{smallmatrix}\right]\otimes \mathbb{I}_n
%&&&   \\
%& \ddots &&    \\
%&& 
%\left[\begin{smallmatrix}
%\hat{U}_{t}^*  &  \\
% & \hat{U}_{t}^*
%\end{smallmatrix}\right]\otimes \mathbb{I}_n
% &
%\\
%&&& \mathbb{I}_{2n^2}
%\end{smallmatrix}\right]      
%%%%%%%%%%%%%%%%%%%%%%%%%%%%%%%%%%%%%%%%%%%%%%%%%%%%%%%%%%%%%%%%%%%
\widetilde{U}_{n,n}^*\left[
\begin{smallmatrix}
\left[
\begin{smallmatrix}
a_{\alpha}
\left[\begin{smallmatrix}
T_{\alpha}^*T_{\beta} & 0  \\ 0 & 0
\end{smallmatrix}\right]\otimes \mathbb{I}_n
 &  \\
& \overline{a}_{\alpha}
\left[\begin{smallmatrix}
T_{\alpha}^*T_{\beta} & 0  \\ 0 & 0
\end{smallmatrix}\right]\otimes \mathbb{I}_n
\end{smallmatrix}\right]_{\alpha,\beta=1}^t
&
\left[\begin{smallmatrix}
a_{\alpha}\left[\begin{smallmatrix}
T_{\alpha}^* \\
0 
\end{smallmatrix}\right]\otimes \mathbb{I}_n  &  \\
 & \overline{a}_{\alpha} 
 \left[\begin{smallmatrix}
T_{\alpha}^* \\
0 
\end{smallmatrix}\right]
 \otimes \mathbb{I}_n
\end{smallmatrix}\right]_{\alpha=1}^t
   \\
\left[
\begin{smallmatrix}
\begin{bmatrix}
T_{\beta} & 0 
\end{bmatrix}\otimes \mathbb{I}_n  &  \\
 &  
\begin{bmatrix}
T_{\beta} & 0 
\end{bmatrix}\otimes \mathbb{I}_n 
\end{smallmatrix}\right]_{\beta=1}^t
&
\left[\begin{smallmatrix}
\mathbb{I}_n
\otimes \mathbb{I}_n &  \\
& \mathbb{I}_n
\otimes \mathbb{I}_n
\end{smallmatrix}\right]
\end{smallmatrix}\right]\widetilde{U}_{n,n},
%\\
%&\times
%\left[\begin{smallmatrix}
%\left[\begin{smallmatrix}
%\hat{U}_{1}  &  \\
% & \hat{U}_{1}
%\end{smallmatrix}\right]\otimes \mathbb{I}_n
%&&&   \\
%& \ddots &&    \\
%&& 
%\left[\begin{smallmatrix}
%\hat{U}_{t}  &  \\
% & \hat{U}_{t}
%\end{smallmatrix}\right]\otimes \mathbb{I}_n
% &
%\\
%&&& \mathbb{I}_{2n^2}
%\end{smallmatrix}\right].
\end{aligned}
\end{equation}
\end{small}
where 
 the unitary matrix 
\begin{equation}
\widetilde{U}_{n,r}=\mathrm{diag} 
\bigg( 
\left[
 \begin{matrix}
\hat{U}_{1}  &  \\
 & \hat{U}_{1}
\end{matrix}
\right] \otimes \mathbb{I}_n
, \ldots, \left[
 \begin{matrix}
\hat{U}_{t}  &  \\
 & \hat{U}_{t}
\end{matrix}
\right] \otimes \mathbb{I}_n, \mathbb{I}_{2nr}\bigg).
\end{equation}

Notice the  commutativity $ \widetilde{U}_{n, \widetilde{r}_{t+1}} L_1= L_1\widetilde{U}_{n, \widetilde{r}_{t+1}}$, 
\begin{small}
\begin{equation}\label{invariant1}
\begin{aligned}
&\left[\begin{smallmatrix}
\mathbb{I}_{2n\sum_{\alpha=1}^t r_{\alpha}} && \\
& \widetilde{A}_{t+1}\widetilde{Q}_{t+1}^* \otimes \mathbb{I}_n & \\
&& \widetilde{A}_{t+1}^*\widetilde{Q}_{t+1}^* \otimes \mathbb{I}_n
\end{smallmatrix}\right]
 \widetilde{U}_{n,n}^*
=       \widetilde{U}_{n, \widetilde{r}_{t+1}}^* 
\left[\begin{smallmatrix}
\mathbb{I}_{2n\sum_{\alpha=1}^t r_{\alpha}} && \\
& \widetilde{A}_{t+1}\widetilde{Q}_{t+1}^* \otimes \mathbb{I}_n & \\
&& \widetilde{A}_{t+1}^*\widetilde{Q}_{t+1}^* \otimes \mathbb{I}_n
\end{smallmatrix}\right]
\end{aligned}
\end{equation}
\end{small}
and
\begin{small}
\begin{equation}\label{invariant2}
\begin{aligned}
 \widetilde{U}_{n,n}
\left[\begin{smallmatrix}
\mathbb{I}_{2n\sum_{\alpha=1}^t r_{\alpha}} && \\
& \widetilde{Q}_{t+1} \otimes \mathbb{I}_n & \\
&&\widetilde{Q}_{t+1} \otimes \mathbb{I}_n
\end{smallmatrix}\right] 
=              
&\left[\begin{smallmatrix}
\mathbb{I}_{2n\sum_{\alpha=1}^t r_{\alpha}} && \\
& \widetilde{Q}_{t+1} \otimes \mathbb{I}_n & \\
&& \widetilde{Q}_{t+1}\otimes \mathbb{I}_n
\end{smallmatrix}\right]
  \widetilde{U}_{n, \widetilde{r}_{t+1}},
\end{aligned}
\end{equation}
\end{small}
 we see  from \eqref{L2 another} and \eqref{tilde L2} that   \eqref{L1 add L2}   changes to 
 \begin{equation}\label{L1 add L2hat}
D_Y=\det\left(L_1+\sqrt{\gamma_N}\widetilde{L}'_2\right),
\end{equation}
where
\begin{equation}\label{L2hat}
\widetilde{L}'_2=\left[
\begin{smallmatrix}
\left[
\begin{smallmatrix}
a_{\alpha}
\left[\begin{smallmatrix}
T_{\alpha}^*T_{\beta} & 0  \\ 0 & 0
\end{smallmatrix}\right]\otimes \mathbb{I}_n
 &  \\
& \overline{a}_{\alpha}
\left[\begin{smallmatrix}
T_{\alpha}^*T_{\beta} & 0  \\ 0 & 0
\end{smallmatrix}\right]\otimes \mathbb{I}_n
\end{smallmatrix}\right]_{\alpha,\beta=1}^t
&
\left[\begin{smallmatrix}
a_{\alpha}
\big(
\left[\begin{smallmatrix}
T_{\alpha}^* \\
0 
\end{smallmatrix}\right]
\widetilde{Q}_{t+1}
\big)
\otimes \mathbb{I}_n  &  \\
 & \overline{a}_{\alpha} 
\big(
 \left[\begin{smallmatrix}
T_{\alpha}^* \\
0 
\end{smallmatrix}\right]\widetilde{Q}_{t+1}
\big)
 \otimes \mathbb{I}_n
\end{smallmatrix}\right]_{\alpha=1}^t
   \\
\left[
\begin{smallmatrix}
\widetilde{A}_{t+1}\widetilde{Q}_{t+1}^*
\left[
\begin{smallmatrix}
T_{\beta} & 0 
\end{smallmatrix}\right]\otimes \mathbb{I}_n  &  \\
 &  
 \widetilde{A}_{t+1}^*\widetilde{Q}_{t+1}^*
\left[
\begin{smallmatrix}
T_{\beta} & 0 
\end{smallmatrix}\right]\otimes \mathbb{I}_n 
\end{smallmatrix}\right]_{\beta=1}^t
&
\left[\begin{smallmatrix}
\widetilde{A}_{t+1}\widetilde{Q}_{t+1}^*\widetilde{Q}_{t+1}
\otimes \mathbb{I}_n &  \\
& \widetilde{A}_{t+1}^*\widetilde{Q}_{t+1}^*\widetilde{Q}_{t+1}
\otimes \mathbb{I}_n
\end{smallmatrix}\right]
\end{smallmatrix}\right].
\end{equation}
 Eliminate zeros  from   $\widetilde{L}'_2$ and do  simple matrix   operations,  we thus derive   from \eqref{L1 add L2hat}  that 
\begin{equation}\label{tildeL1 add L0}
D_Y
=\det\left(
\widetilde{L}_1+ \sqrt{\gamma_N} \widetilde{L}''_2
\right)\prod_{\alpha=1}^t
\big(\det(E_{\alpha}) \big)^{r_{\alpha}-n},
\end{equation}
where for $\alpha=1,\cdots,t$, $E_{\alpha}$ is defined in \eqref{A alpha},
\begin{equation}\label{tilde L1}
\widetilde{L}_1=
{\rm diag}\big(
B_{1,n},\cdots,B_{t,n},
B_{t+1}
\big),
\end{equation}
and
\begin{equation}\label{L0}
\widetilde{L}''_2=
\left[
\begin{smallmatrix}
%%%%%%%%%%%%%%%%%%%%%%%%%%%%%%%%%%%%%%%%%%%%%%%%%%%%%%%%
\left[
\begin{smallmatrix}
a_{\alpha}
(T_{\alpha}^*T_{\beta}) \otimes \mathbb{I}_n
 &  \\
& \overline{a}_{\alpha}
(T_{\alpha}^*T_{\beta}) \otimes \mathbb{I}_n
\end{smallmatrix}\right]_{\alpha,\beta=1}^t
&
\left[\begin{smallmatrix}
a_{\alpha}
(
T_{\alpha}^*
\widetilde{Q}_{t+1}
)
\otimes \mathbb{I}_n  &  \\
 & \overline{a}_{\alpha} 
(
T_{\alpha}^*\widetilde{Q}_{t+1}
)
 \otimes \mathbb{I}_n
\end{smallmatrix}\right]_{\alpha=1}^t
   \\
\left[
\begin{smallmatrix}
(\widetilde{A}_{t+1}\widetilde{Q}_{t+1}^*
T_{\beta})\otimes \mathbb{I}_n  &  \\
 &  
 (\widetilde{A}_{t+1}^*\widetilde{Q}_{t+1}^*
T_{\beta})\otimes \mathbb{I}_n 
\end{smallmatrix}\right]_{\beta=1}^t
&
\left[\begin{smallmatrix}
(\widetilde{A}_{t+1}\widetilde{Q}_{t+1}^*\widetilde{Q}_{t+1})
\otimes \mathbb{I}_n &  \\
& (\widetilde{A}_{t+1}^*\widetilde{Q}_{t+1}^*\widetilde{Q}_{t+1})
\otimes \mathbb{I}_n
\end{smallmatrix}\right]
%%%%%%%%%%%%%%%%%%%%%%%%%%%%%%%%%%%%%%%%%%%%%%%%%%%%%%%%%%%%%
\end{smallmatrix}\right].
\end{equation}

To extract the tensor product further, we need to change the order of tensor product.  By  Proposition \ref{tensorproperty}, 
 we have 
\begin{equation}\label{L1hat add L0hat}
D_Y
=
\det\left(
\widehat{L}_1+\sqrt{\gamma_N}\widehat{L}_2
\right)\prod_{\alpha=1}^t
\big( \det(E_{\alpha}) \big)^{r_{\alpha}-n},
\end{equation}
where $\widehat{L}_1$ and $\widehat{L}_2$ are defined in \eqref{L1hat} and \eqref{L0hat} respectively.

%%%%%%%%%%%%%%%%%%%%%%%%%%%%%%%%%%%%%%%%%%%%%%%%%%%%%%%%%%%%%%%%%%%%%%%%%%%%%%%%%%%%%%

 On the other hand, rewrite $T_{\alpha}$ as  sum of  a diagonal matrix $\sqrt{T_{{\rm d},\alpha}}$ and a strictly upper triangular  matrix ${T_{{\rm u},\alpha}}$, which is done in \eqref{Talpha 2} and \eqref{Tdalpha}.  
The Jacobian of the transform \eqref{Qalpha QR} reads (see e.g.   \cite[Proposition 3.2.5]{FR10})
\begin{equation}\label{Qalpha QR Jacobian}
{\rm d}Q_{\alpha}=
\frac{\pi^{n r_{\alpha}-\frac{n(n-1)}{2}}}{\prod_{j=1}^n(r_{\alpha}-j)!}
\prod_{j=1}^n \big( t_{j,j}^{(\alpha)} \big)^{r_{\alpha}-j}
{\rm d}T_{\alpha} {\rm d}U_{\alpha},
\end{equation}
where  the  induced measure   from the Haar measure  has been  chosen such that  
$$\int {\rm d} U_{\alpha}=1,\quad \alpha=1,\ldots,t.$$

Finally,  recalling  \eqref{algebraequa},   substituting the QR decomposition \eqref{Qalpha QR} into $h(Q)$,  
combine    \eqref{EGinUE dual}, \eqref{L1hat add L0hat} and \eqref{Qalpha QR Jacobian}, after integrating out $U_{\alpha}$ and with the above notations we  have   the desired result. 
\end{proof}

\subsection{Localization reduction}

 It's easy to see from  \eqref{A alpha} that  
\begin{equation} \label{Aa}
E_{\alpha}=
\begin{bmatrix}
\sqrt{ \gamma_N} (z_0-a_{\alpha}) \mathbb{I}_n &
-Y^* \\
Y &
\sqrt{ \gamma_N} \overline{z_0-a_{\alpha}}\mathbb{I}_n
\end{bmatrix}+
\sqrt{ \gamma_N}N^{-\frac{1}{4}}
\begin{bmatrix}
\hat{Z}  &   \\
&   \hat{Z}^*
\end{bmatrix},
\end{equation}
and 
\begin{equation}\label{logAalpha decompose}
\log\det(E_{\alpha})=\log\det\left(
\gamma_N f_{\alpha}\mathbb{I}_n+YY^*
\right)+\log\det\left(
\mathbb{I}_{2n}+\sqrt{ \gamma_N}N^{-\frac{1}{4}}
\widehat{E}_{\alpha}
\right),
\end{equation}
where, for $ \alpha=1,\cdots,t$, $f_{\alpha}$ is defined in \eqref{falpha}, and
\begin{equation}\label{HAalpha}
\begin{aligned}
&\widehat{E}_{\alpha}=
\left[\begin{smallmatrix}
\sqrt{\gamma_N}\overline{z_0-a_{\alpha}}\left(
 \gamma_N f_{\alpha}\mathbb{I}_n+Y^*Y
\right)^{-1}\hat{Z}
& Y^*\left(
\gamma_N f_{\alpha}\mathbb{I}_n+YY^*
\right)^{-1}\hat{Z}^*       \\
-Y\left(
\gamma_N f_{\alpha}\mathbb{I}_n+Y^*Y
\right)^{-1}\hat{Z}
& \sqrt{\gamma_N}(z_0-a_{\alpha})\left(
\gamma_Nf_{\alpha}\mathbb{I}_n+YY^*
\right)^{-1}\hat{Z}^*
\end{smallmatrix}\right].
\end{aligned}
\end{equation}

Take all the  leading terms of $f(\boldsymbol{T},Y,\boldsymbol{Q})$ in  \eqref{fTY} and put
\begin{equation}\label{f0TY}
\widetilde{f}(\boldsymbol{T},Y,\boldsymbol{Q})=\sum_{\alpha=1}^t\big(
c_{\alpha}\log\det(T_{\alpha}T_{\alpha}^*) 
+c_{\alpha}\log\det\big(
f_{\alpha}\mathbb{I}_n+YY^*
\big)
\big)   
 +
h_{0}(\boldsymbol{T},\boldsymbol{Q})-{\rm Tr}(YY^*)
,
\end{equation}
with
\begin{multline}\label{h0T}
h_{0}(\boldsymbol{T},\boldsymbol{Q})= 
-\sum_{\alpha=1}^t f_{\alpha}{\rm Tr}\big(
T_{\alpha}T_{\alpha}^*
\big)
-{\rm Tr}Q_{t+1}
\big(
z_0\mathbb{I}_{r_{t+1}}-A_{t+1}
\big)^*
\big(
z_0\mathbb{I}_{r_{t+1}}-A_{t+1}
\big)
Q_{t+1}^*               \\
+
|z_0|^2\bigg(
\sum_{\alpha=1}^t
{\rm Tr}\big(
T_{\alpha}T_{\alpha}^*
\big)+{\rm Tr}\big(
Q_0Q_0^*
\big)+{\rm Tr}\big(
Q_{t+1}Q_{t+1}^*
\big)
\bigg)         \\ 
+\sum_{1\leq i < j\leq n}
\bigg|
\bigg(
\sum_{\alpha=1}^t a_{\alpha}T_{\alpha}T_{\alpha}^*
+z_0Q_0Q_0^*+Q_{t+1}A_{t+1}Q_{t+1}^*
\bigg)_{i,j}
\bigg|^2.         
\end{multline}

By restricting the  integration region to 
\begin{equation}\label{regiondelta}
\Omega_{N,\delta}=\Omega
\cap
A_{N,\delta},
\end{equation}
where $\Omega$ is defined in 
\eqref{integrationregionprop} and
\begin{equation}\label{ANdelta}
A_{N,\delta}:=\bigg\{ (\boldsymbol{T},Y,\boldsymbol{Q})\Big|
\sum_{\alpha=1}^t
\left(
\big\|
T_{{\rm d},\alpha}- c_{\alpha} f_{\alpha}^{-1}
\mathbb{I}_n
\big\|^2
+\|T_{{\rm u},\alpha}\|^2
\right)  
+ 
\|Y\|^2+\|Q_0\|^2+\|Q_{t+1}\|^2
\leq \delta \bigg\},
\end{equation}
we have    
\begin{proposition}\label{RNn delta} 
Let $\tau=1+N^{-\frac{1}{2}}\hat{\tau}$ and assume that $P_{00}(z_0)=1$, then with the same notations  and assumptions as in  Proposition \ref{foranlysis}, for any positive constant $\delta$ there exists $\Delta>0$ such that
\begin{equation}\label{RNndelta}
R_N^{(n)}(X_0;z_1,\cdots,z_n)=
D_{N,n}\,\Big(
I_{N,\delta}+O\big(
e^{-\frac{1}{2}N\Delta}
\big)
\Big),
\end{equation}
where
\begin{equation}\label{INdelta}
I_{N,\delta}:=
  \int_{\Omega_{N,\delta}} 
   \det\Big(
\widehat{L}_1+\sqrt{\gamma_N}\widehat{L}_2
\Big)
  \exp\Big\{ N\Big(f(\boldsymbol{T},Y,\boldsymbol{Q})-\widetilde{f}
  \big(\big(\sqrt{ c_{\alpha}/f_{\alpha}}\mathbb{I}_n\big)_{\alpha},0,0\big)\Big)
 \Big\} {\rm d}V,
\end{equation}
  and 
\begin{multline}\label{DNn}
{\rm d}V:= {\rm d}Y
{\rm d}Q_0{\rm d}Q_{t+1} \prod_{\alpha=1}^t
{\rm d}T_{\alpha},\quad D_{N,n}:=
e^{nN(|z_0|^2-1)}\prod_{\alpha=1}^t
c_{\alpha}^{nNc_{\alpha}}
\\
\times \Big( \frac{N-n}{\pi\tau} \Big)^{n N}
\frac{N^{\frac{n(n+1)}{2}}}{\tau^{\frac{n(n+1)}{2}} \pi^n \prod_{i=1}^{n-1}i!}
e^{-\frac{N}{\tau}\sum_{k=1}^n|z_k|^2}
\prod_{1\leq i <j \leq n}|z_i-z_j|^2
\prod_{\alpha=1}^t
\frac{\pi^{nr_{\alpha}-\frac{n(n-1)}{2}}}
{\prod_{j=1}^n(r_{\alpha}-j)!}.
\end{multline}

\end{proposition}
\begin{proof}
By  Lemma \ref{maximumY}  and Lemma  \ref{maximum lemma}  in Section \ref{Sectmax} below,  
we see that $\widetilde{f}(\boldsymbol{T},Y,\boldsymbol{Q})$ in \eqref{f0TY}  attains its maximum  only at  
\begin{equation}\label{maximumpoint}
Y=0, \ Q_0=0, \ Q_{t+1}=0, \quad T_{\alpha}=\sqrt{ c_{\alpha}/f_{\alpha}}\mathbb{I}_n,\quad \alpha=1,\cdots,t.
\end{equation}
So  there exists $\Delta>0$ such that in  the domain $\Omega\cap 
A_{N,\delta}^{\complement}$,
\begin{equation}\label{f0TY minus f0}
\widetilde{f}(\boldsymbol{T},Y,\boldsymbol{Q})-
\widetilde{f}\big(\big(\sqrt{ c_{\alpha}/f_{\alpha}}\mathbb{I}_n\big)_{\alpha},0,0\big)
\leq -2\Delta.
\end{equation}

First, rewrite 
\begin{multline}\label{Concentration0}
N\Big(
 f(\boldsymbol{T},Y,\boldsymbol{Q})-\widetilde{f}
  \big(\big(\sqrt{ c_{\alpha}/f_{\alpha}}\mathbb{I}_n\big)_{\alpha},0,0\big)
\Big)=
\\
N\Big(
f(\boldsymbol{T},Y,\boldsymbol{Q})-
\widetilde{f}(\boldsymbol{T},Y,\boldsymbol{Q})
\Big)+
N\Big(
 \widetilde{f}(\boldsymbol{T},Y,\boldsymbol{Q})-\widetilde{f}
  \big(\big(\sqrt{ c_{\alpha}/f_{\alpha}}\mathbb{I}_n\big)_{\alpha},0,0\big)
\Big)
.
\end{multline}
On the one hand,  noting that  $$ \sum_{\alpha=1}^t
T_{\alpha}T_{\alpha}^*+Q_0Q_0^*+Q_{t+1}Q_{t+1}^*\leq \mathbb{I}_n$$
in the domain   $\Omega$ defined in \eqref{integrationregionprop}, by comparing \eqref{hQrewrite} with \eqref{h0T}
we know that for 
$\hat{Z}$
in a compact subset of $\mathbb{C}^n$,
\begin{equation}\label{error term1}
\frac{1}{\tau} h(\boldsymbol{T},\boldsymbol{Q})-h_0(\boldsymbol{T},\boldsymbol{Q})
=O(N^{-\frac{1}{4}}).
\end{equation}

On the other hand, the edge point $z_0$ satisfies $P_{00}(z_0)=1$. From \eqref{nu} and \eqref{parameter} we have 
$ \sum_{\alpha=1}^t c_{\alpha}/f_{\alpha}=1. $ Hence, we have $f_{\alpha}>0.$ By applying the singular value decomposition \eqref{singularcorre},  for any $Y$, we obtain 
\begin{equation*}
(\sqrt{\gamma_N}f_{\alpha}\mathbb{I}_n+Y^*Y)^{-1}=O(1),
\quad Y(\sqrt{\gamma_N}f_{\alpha}\mathbb{I}_n+Y^*Y)^{-1}=O(1),
\end{equation*}
and 
\begin{equation}\label{error term2}
\log\det\left(
\mathbb{I}_{2n}+\sqrt{\gamma_N} N^{-\frac{1}{4}}
\widehat{E}_{\alpha}
\right)=O(N^{-\frac{1}{4}})
\end{equation}
since $\widehat{E}_{\alpha}=O(1)$. 

Meanwhile,  we have 
\begin{equation*}
\log\det\left(
\gamma_N f_{\alpha}\mathbb{I}_{n}+
YY^*
\right)=\log\det\big(
f_{\alpha}\mathbb{I}_{n}+
YY^*
\big)+ O\big( N^{-1} \big),
\end{equation*}
and from   \eqref{logAalpha decompose}  that 
\begin{equation}\label{logAalpha decompose0}
\log\det\big(
E_{\alpha}
\big)=\log\det\big(
f_{\alpha}\mathbb{I}_{n}+
YY^*
\big)+ O\big( N^{-\frac{1}{4}} \big).
\end{equation}
Combination of  \eqref{fTY}, \eqref{f0TY} and \eqref{error term1} gives rise to 
\begin{multline}\label{Concentration1}
N\Big(
f(\boldsymbol{T},Y,\boldsymbol{Q})-
\widetilde{f}(\boldsymbol{T},Y,\boldsymbol{Q})
\Big)=
\\
\sum_{\alpha=1}^t
\sum_{j=1}^n (R_{\alpha,N}-j)\log\big(
t_{j,j}^{(\alpha)}
\big)+
\sum_{\alpha=1}^t
(R_{\alpha,N}-n)\log\det\big(
E_{\alpha}
\big)
+O\big(
N^{\frac{3}{4}}+\sqrt{N}{\rm Tr}(YY^*)
\big).
\end{multline}
Here, note that in \eqref{f0TY}, from \eqref{Talpha 2} and \eqref{Tdalpha}, we have $ \log\det(T_{\alpha}T_{\alpha}^*)=\sum_{j=1}^n \log\big(
t_{j,j}^{(\alpha)}
\big). $

Further, there exists a positive constant $C_0$ such that
\begin{multline}\label{Concentration2}
\widetilde{f}(\boldsymbol{T},Y,\boldsymbol{Q})-\widetilde{f}
  \big(\big(\sqrt{ c_{\alpha}/f_{\alpha}}\mathbb{I}_n\big)_{\alpha},0,0\big)
 \\ \leq
  C_0-{\rm Tr}(YY^*)
  +\sum_{\alpha=1}^t \big(
c_{\alpha}\log\det(T_{\alpha}T_{\alpha}^*) 
+c_{\alpha}\log\det\big(
f_{\alpha}\mathbb{I}_n+YY^*
\big)
\big) .
\end{multline}
Elementary analysis shows that there exists a positive constant $C_1$ such that
\begin{equation}\label{Concentration3}
-\frac{1}{2}{\rm Tr}(YY^*)
  +\sum_{\alpha=1}^t c_{\alpha}\log\det
  (f_{\alpha}\mathbb{I}_n+YY^*)\leq
  C_1.
\end{equation}

Now take a positive constant $\epsilon_0$ such that
 \begin{equation}\label{ConcentrationCondition}
0<\epsilon_0<\min\Big\{
\frac{\Delta}{\Delta+C_0+C_1},\frac{1}{3}
\Big\}.
\end{equation}
Further,  take a fixed positive integer $N_0$ such that $N_0\epsilon_0>\sum_{\alpha=1}^t |n-R_{\alpha,N}|/c_{\alpha}$. 
Combining \eqref{f0TY minus f0}, \eqref{Concentration2} and \eqref{Concentration3}, for sufficiently large $N$ we have
\begin{multline}\label{Concentration4}
N\Big(
 \widetilde{f}(\boldsymbol{T},Y,\boldsymbol{Q})-\widetilde{f}
  \big(\big(\sqrt{ c_{\alpha}/f_{\alpha}}\mathbb{I}_n\big)_{\alpha},0,0\big)
\Big)+O\big(
\sqrt{N}{\rm Tr}(YY^*)
\big)
\leq 
-2N(1-\epsilon_0)\Delta+N\epsilon_0C_0
\\
+
N\epsilon_0\Big(
-\frac{1}{2}{\rm Tr}(YY^*)
  +\sum_{\alpha=1}^t\big(
c_{\alpha}\log\det(T_{\alpha}T_{\alpha}^*) 
+c_{\alpha}\log\det\big(
f_{\alpha}\mathbb{I}_n+YY^*
\big)
\big)
\Big)
\\
\leq -\frac{2N \Delta}{3}
+N_0\epsilon_0\Big(
-\frac{1}{2}{\rm Tr}(YY^*)
  +\sum_{\alpha=1}^t\big(
c_{\alpha}\log\det(T_{\alpha}T_{\alpha}^*) 
+c_{\alpha}\log\det\big(
f_{\alpha}\mathbb{I}_n+YY^*
\big)
\big)
\Big).
\end{multline}

From \eqref{L1hat} and \eqref{L0hat}, we know that the determinant $ \det\Big(
\widehat{L}_1+\sqrt{\gamma_N}\widehat{L}_2
\Big) $ is a polynomial of the matrix variables with finite degree.
Hence, we have
\begin{multline}\label{Concentration5}
\int_{A_{N,\delta}^{\complement}\cap \Omega}
\Big| \det\Big(
\widehat{L}_1+\sqrt{\gamma_N}\widehat{L}_2
\Big) \Big|\exp\Big\{ N_0\epsilon_0\Big(
\sum_{\alpha=1}^t\big(
c_{\alpha}\log\det(T_{\alpha}T_{\alpha}^*) 
+c_{\alpha}\log\det\big(
f_{\alpha}\mathbb{I}_n+YY^*
\big)
\big)
\\
-\frac{1}{2}{\rm Tr}(YY^*)\Big)
+\sum_{\alpha=1}^t
\sum_{j=1}^n (R_{\alpha,N}-j)\log\big(
t_{j,j}^{(\alpha)}
\big)+
\sum_{\alpha=1}^t
(R_{\alpha,N}-n)\log\det\big(
E_{\alpha}
\big) \Big\}
{\rm d}V
=O(1).
\end{multline}
Here, by recalling \eqref{f0TY} and \eqref{integrationregionprop}, which are the definitions of $\widetilde{f}(\boldsymbol{T},Y,\boldsymbol{Q})$ and $\Omega$ respectively, we know that the matrix variable $Y$ is controlled by the factor $\exp(-\frac{N_0\epsilon_0}{2}{\rm Tr}(YY^*))$, and the remaining variables $\boldsymbol{T}$ and $\boldsymbol{Q}$ are restricted in a compact set. Hence, the integrability is guaranteed.

Finally, combining \eqref{Concentration0}, \eqref{Concentration1}, \eqref{Concentration4} and \eqref{Concentration5}  we have
\begin{multline*}
\left|
\int_{A_{N,\delta}^{\complement}\cap \Omega}
\det\Big(
\widehat{L}_1+\sqrt{\gamma_N}\widehat{L}_2
\Big) \exp\left\{ N\Big(
 f(\boldsymbol{T},Y,\boldsymbol{Q})-\widetilde{f}
  \big(\big(\sqrt{ c_{\alpha}/f_{\alpha}}\mathbb{I}_n\big)_{\alpha},0,0\big)
\Big) \right\}
{\rm d}V
\right|   \\
\leq e^{-\frac{2}{3}N\Delta+O\big(N^{\frac{3}{4}}\big)}
 O(1)    =O\big(
e^{-\frac{1}{2}N\Delta}
\big).
\end{multline*}
This  thus completes the proof.
\end{proof}

%%%%%%%%%%%%%%%%%%%%%%%%%%%%%%%%%%%%%%%%%%%%%%%%%%%%%%%%%%%%%%%%%%%%%%%%%%%55
%%%%%%%%%%%%%%%%%%%%%%%%%%%%%%%%%%%%%%%%%%%%%%%%%%%%%%%%%%%%%%%%%%%%%%%%%%%%%%%%%%%%%%%%%%%%%%%%%%%%%%%%%5

%$$\int \cdots \int $$

%%%%%%%%%%%%%%%%%%%%%%%%%%%%%%%%%%%%%%%%%%%%%%%%%%%%%%%%%%%%%%%%%%%%%%%%%%%%%%%%%%%%%%%%%%%%%%%%%%%%%%%%%%%%%%%%%%%%%%%%%%%%%%%%%%%%%%%%%%%%%%%%%%%%%%%%%%%%%%%%%%%%%%%%%%%%%%%%%%%%%%%%%%%%%%%%%%%%%%%%%%%%%%%%%%%%%%%%%%%%%%
%极值引理
\subsection{Maximum  lemmas} \label{Sectmax}
Given   complex numbers $z_0$,       $a_{\alpha}$  and   
$c_{\alpha}>0 $    ($\alpha=1,\cdots,t$),     assume that 
\begin{equation} \sum_{\alpha=1}^t c_{\alpha}=1, \quad  \sum_{\alpha=1}^t \frac{c_{\alpha}}{|z_0-a_{\alpha}|^2}\geq \frac{1}{\tau}. \end{equation}
Let $t_0$ be  the unique   non-negative  solution of the equation   \begin{equation}
\sum_{\alpha=1}^t \frac{c_{\alpha}}{|z_0-a_{\alpha}|^2+t_0}=\frac{1}{\tau}.
\end{equation}

The following two lemmas hold and play a  significant  role in investigating local eigenvalue statistics of the deformed GinUEs. 
\begin{lemma}\label{maximumY}
As a function of a non-negative definite matrix $H$,
\begin{equation}\label{Hpart}
 \sum_{\alpha=1}^t c_{\alpha}
\log\det\big(
|z_0-a_{\alpha}|^2 \mathbb{I}_n+H
\big)   -
\frac{1}{\tau}  {\rm Tr}(H)
\end{equation}
attains it maximum only at  $H=t_0 I_n$. \end{lemma}

\begin{proof}
Obviously, it's sufficient to check   the $n=1$ case. For this, take the first and second derivatives,  find the saddle point and we can  easily complete the proof.   
\end{proof}

\begin{lemma}\label{maximum lemma}
Assume that  $z_0$ is not  an eigenvalue of   the   $n\times n$   complex normal   matrix  $D$,  let 
\begin{equation}\label{Jn lemma}
\begin{aligned}
J_n&=\sum_{\alpha=1}^t\Big(
\tau c_{\alpha}\log\det\big(
T_{\alpha}T_{\alpha}^*
\big)-|z_0-a_{\alpha}|^2{\rm Tr}\big(
T_{\alpha}T_{\alpha}^*
\big)
\Big)
+|z_0|^2 {\rm Tr}\Big(
\sum_{\alpha=1}^t T_{\alpha}T_{\alpha}^*
+AA^*+BB^*
\Big)
\\&-{\rm Tr}\Big(
B\big(
z_0\mathbb{I}_{l_2}-D
\big)^*\big(
z_0\mathbb{I}_{l_2}-D
\big)B^*
\Big)
+\sum_{1\leq i<j\leq n}\left|
\Big(
\sum_{\alpha=1}^t a_{\alpha}T_{\alpha}T_{\alpha}^*
+z_0AA^*+BDB^*
\Big)_{i,j}
\right|^2,
\end{aligned}
\end{equation}
where  as matrix-valued  variables
$T_{\alpha}$'s  are  $n\times n$  upper-triangular matrices  with  positive diagonal elements,   
$A$ and $B$  are   $n\times l_1$ and  $n\times l_2$  matrices respectively. Then %\begin{equation}\label{Jn inequlity}
%J_n\leq n\sum_{\alpha=1}^t c_{\alpha}\Big(
%\log\frac{c_{\alpha}}{|z_0-a_{\alpha}|^2+t_0}
%-\frac{|z_0-a_{\alpha}|^2}{|z_0-a_{\alpha}|^2+t_0}
%\Big)+n|z_0|^2,
%\end{equation} 
\begin{equation}\label{Jn inequlity}
J_n\leq n\sum_{\alpha=1}^t \tau c_{\alpha}
\log\frac{\tau c_{\alpha}}{|z_0-a_{\alpha}|^2+t_0}
+n(t_0+|z_0|^2-\tau),
\end{equation} 
whenever  
\begin{equation}\label{lemma condition}
\sum_{\alpha=1}^t T_{\alpha}T_{\alpha}^*+AA^*+BB^*
\leq \mathbb{I}_n,.
\end{equation}
Moreover, the  equality holds  if and only if 
\begin{equation} 
A=0,\quad
B=0, \quad T_{\alpha}=\sqrt{\frac{\tau c_{\alpha}}{|z_0-a_{\alpha}|^2+t_0}}
\mathbb{I}_n, \quad \alpha=1,  \ldots, t.
\end{equation} 
\end{lemma}
\begin{proof}
We  proceed  by  induction on $n$.  Without loss of generality, choose $\tau=1$.  For simplicity,   let  
 $
f_{\alpha}=|z_0-a_{\alpha}|^2
$ as in  \eqref{falpha}.

{\bf (1) Base Step.}  When $n=1$,  rewrite $T_{\alpha}=\sqrt{t_{\alpha}}$, then
%$$ \sum_{\alpha=1}^t t_{\alpha}+AA^*+BB^*\leq 1 $$
\begin{small}
\begin{equation*}
\begin{aligned}
J_1=\sum_{\alpha=1}^t\big(c_{\alpha}\log(t_{\alpha})
-f_{\alpha}t_{\alpha}
\big)
+|z_0|^2\big(
\sum_{\alpha=1}^t t_{\alpha}+AA^*+BB^*
\big)-B\big(
z_0\mathbb{I}_{l_2}-D
\big)^*\big(
z_0\mathbb{I}_{l_2}-D
\big)B^*.
\end{aligned}
\end{equation*}
\end{small}

Note that the simple  inequality  
\begin{equation} \label{logine}
c_{\alpha}\log(t_{\alpha}) - f_{\alpha} t_{\alpha}  \leq  t_0 t_{\alpha} -c_{\alpha}+c_{\alpha}
\log\frac{c_{\alpha}}{f_{\alpha}+t_0}, 
\end{equation}
with   the  equality    if and only if $t_{\alpha}=c_{\alpha}/(f_{\alpha}+t_0)$, 
we have 
\begin{equation*}
\begin{aligned}
J_1\leq (t_0+|z_0|^2)
\sum_{\alpha=1}^t t_{\alpha} +
|z_0|^2\big(
AA^*+BB^*
\big)
-B\big(
z_0\mathbb{I}_{l_2}-D
\big)^*\big(
z_0\mathbb{I}_{l_2}-D
\big)B^* -1+\sum_{\alpha=1}^t  c_{\alpha}
\log\frac{c_{\alpha}}{f_{\alpha}+t_0}.
\end{aligned}
\end{equation*}
 Together with the constraint $\sum_{\alpha=1}^t t_{\alpha}+AA^*+BB^*\leq 1$, 
we further obtain
\begin{equation*}
\begin{aligned}
J_1\leq  t_0+|z_0|^2 -1+\sum_{\alpha=1}^t  c_{\alpha}
\log\frac{c_{\alpha}}{f_{\alpha}+t_0}.
\end{aligned}
\end{equation*}
with equality if and only if $t_{\alpha}=c_{\alpha}/(f_{\alpha}+t_0)$,
  $A=0$ and $B=0$. So  this completes  the $n=1$ case.

{\bf (2) Inductive Step.}  Assuming that the desired result  holds true for the $n-1$ case,  we  will  prove it for  the $n$ case. For this,   rewrite the relevant matrices in  $2\times 2$ block matrix form  \begin{equation}\label{prooflemma2}
T_{\alpha}=\begin{bmatrix}
\sqrt{t_{1,1}^{(\alpha)}} & T_{1,2}^{(\alpha)}
\\
& T_{2,2}^{(\alpha)}
\end{bmatrix},\quad
A=\begin{bmatrix}
A_1 \\ A_2
\end{bmatrix},\quad
B=\begin{bmatrix}
B_1 \\ B_2
\end{bmatrix},
\end{equation} 
  where $T_{1,2}^{(\alpha)}$, $A_1$ and $B_1$ are  of size $1\times (n-1)$, $1\times l_1$ and $1\times l_2$  respectively. 
Hence $J_n$ in \eqref{Jn lemma} can be divided in two parts  
\begin{equation}\label{prooflemma5}
J_n=J_{n-1}+\widetilde{J},
\end{equation} 
where $J_{n-1}$ has the same form as $J_{n}$  but in  terms of  $\{ T_{2,2}^{(\alpha)}\},A_2,B_2$  and
\begin{small}
\begin{equation}\label{prooflemma6}
\begin{aligned}
&\widetilde{J}=\sum_{\alpha=1}^t\Big(
c_{\alpha}\log\big(
t_{1,1}^{(\alpha)}
\big)-f_{\alpha}\big(
t_{1,1}^{(\alpha)}+T_{1,2}^{(\alpha)}{T_{1,2}^{(\alpha)}}^*
\big)
\Big)
\\
&+|z_0|^2 \Big(
\sum_{\alpha=1}^t \big(
t_{1,1}^{(\alpha)}+T_{1,2}^{(\alpha)}{T_{1,2}^{(\alpha)}}^*
\big)
+A_1A_1^*+B_1B_1^*
\Big)-
B_1\big(
z_0\mathbb{I}_{l_2}-D
\big)^*\big(
z_0\mathbb{I}_{l_2}-D
\big)B_1^*
\\
&+\Big(
\sum_{\alpha=1}^t a_{\alpha}
T_{1,2}^{(\alpha)}{T_{2,2}^{(\alpha)}}^*
+z_0A_1A_2^*+B_1DB_2^*
\Big)\Big(
\sum_{\alpha=1}^t a_{\alpha}
T_{1,2}^{(\alpha)}{T_{2,2}^{(\alpha)}}^*
+z_0A_1A_2^*+B_1DB_2^*
\Big)^*.
\end{aligned}
\end{equation}
\end{small}

By the   induction hypotheses
we have 
\begin{equation}\label{prooflemma7}
J_{n-1}\leq (n-1)\sum_{\alpha=1}^t c_{\alpha}
\log\frac{c_{\alpha}}{f_{\alpha}+t_0}+(n-1)(t_0+|z_0|^2-1),
\end{equation} 
with equality if and only if 
$$
A_2=0,\quad
B_2=0, \quad T_{2,2}^{(\alpha)}=\sqrt{\frac{c_{\alpha}}{f_{\alpha}+t_0}}
\mathbb{I}_{n-1},\quad \alpha=1,\cdots,t
.
$$
So the remaining task is to prove that $\widetilde{J}$  attains its maximum  only  when 
$$ A_1=0,\quad B_1=0, \quad T_{1,2}^{(\alpha)}=0, \quad 
t_{1,1}^{(\alpha)}=\frac{c_{\alpha}}{f_{\alpha}+t_0},\quad
\alpha=1,\cdots,t.
$$

%$\widetilde{J}$  attains its maximum. We see $\widetilde{J}$ as function of $t_{1,1}^{(\alpha)}$ for $\alpha=1,\cdots,t$ only, then 
%$$
%\frac{\partial \widetilde{J}}{\partial t_{1,1}^{(\alpha)}}
%=\frac{c_{\alpha}}{t_{1,1}^{(\alpha)}}-f_{\alpha}+|z_0|^2.
%$$
%If for $\alpha=1,\cdots,t$, $f_{\alpha}>|z_0|^2$, then
%$$
%\sum_{\alpha=1}^t 
%\frac{c_{\alpha}}{f_{\alpha}-|z_0|^2}\geq 1,
%$$
%also $\widetilde{J}\big|_{t_{1,1}^{(\alpha)}=0}=-\infty.$ So only if
%\begin{equation}\label{prooflemma8}
%S_{1,1}-1-S_{1,2}
%(S_{2,2}-\mathbb{I}_{n-1})^{-1}
%S_{1,2}^*=0,
%\end{equation} 
%$\widetilde{J}$ can attains its maximum.

Now set
\begin{equation}\label{prooflemma9}
T_{1,2}=\big(
T_{1,2}^{(1)},\cdots,T_{1,2}^{(t)},A_1,B_1
\big),\quad
\widetilde{T}=\big(
T_{2,2}^{(1)},\cdots,T_{2,2}^{(t)},A_2,B_2
\big),
\end{equation} 
then $\widetilde{J}$ in \eqref{prooflemma6} reduces to  
\begin{equation}\label{prooflemma10}
\begin{aligned}
\widetilde{J}&=\sum_{\alpha=1}^t\Big(
c_{\alpha}\log\big(
t_{1,1}^{(\alpha)}
\big)-f_{\alpha}
t_{1,1}^{(\alpha)}
\Big)
+|z_0|^2 \Big(
\sum_{\alpha=1}^t 
t_{1,1}^{(\alpha)}+T_{1,2}{T_{1,2}}^*
\Big)
\\&-
T_{1,2}\big(
z_0\mathbb{I}_{l}-L
\big)^*\big(
z_0\mathbb{I}_{l}-L
\big)T_{1,2}^*
+
T_{1,2}L\widetilde{T}^*
\widetilde{T}
L^*T_{1,2}^*,
\end{aligned}
\end{equation} 
where 
\begin{equation}\label{prooflemma11}
L=\mbox{diag}(a_1\mathbb{I}_{n-1},\ldots, a_t\mathbb{I}_{n-1}, z_0\mathbb{I}_{l_1},D)
%\begin{bmatrix}
%a_1\mathbb{I}_{n-1} &&&&\\
%& \ddots &&& \\
%&& a_t\mathbb{I}_{n-1} && \\
%&&& z_0\mathbb{I}_{l_1} & \\
%&&&& D
%\end{bmatrix}
\end{equation} 
has  size $l\times l$ with $l=t(n-1)+l_1+l_2$.  On the other hand,  the constraint condition in \eqref{lemma condition} reduces to  
 \begin{equation}  \label{pdef}
 \begin{bmatrix}
1- \sum_{\alpha=1}^t t_{1,1}^{(\alpha)}
-T_{1,2}^{} T_{1,2}^*  & -T_{1,2} \widetilde{T}^*
\\
-\widetilde{T} T_{1,2}^* &\mathbb{I}_{n-1}- \widetilde{T} \widetilde{T}^*
\end{bmatrix}\geq 0,
\end{equation} 
which is further equivalent to 

 \begin{equation}  \label{cc2}
 \mathbb{I}_{n-1}- \widetilde{T} \widetilde{T}^*
 \geq  0, \quad 
1- \sum_{\alpha=1}^t t_{1,1}^{(\alpha)}
-T_{1,2} \big((\mathbb{I}_{n-1}+\widetilde{T}^*(\mathbb{I}_{n-1}- \widetilde{T} \widetilde{T}^*)^{-1}\widetilde{T} \big)T_{1,2}^*    \geq 0,
\end{equation} 
if $\mathbb{I}_{n-1}- \widetilde{T} \widetilde{T}^*
 > 0$.
 
 When $\mathbb{I}_{n-1}- \widetilde{T} \widetilde{T}^*
 > 0$, set  \begin{equation}\label{prooflemma14}
F=\mathbb{I}_l-\widetilde{T}^*\widetilde{T},
\end{equation}
 it's easy to see that 
\begin{equation}\label{prooflemma13}
F^{-1}=\mathbb{I}_l+\widetilde{T}^*
\big(\mathbb{I}_{n-1}-\widetilde{T}\widetilde{T}^*\big)^{-1}
\widetilde{T}.
\end{equation}  
  Also set $T_{1,2}=\widetilde{T}_{1,2}F^{\frac{1}{2}}$, then the second condition in   \eqref{cc2}  becomes    
\begin{equation}\label{prooflemma15}
\sum_{\alpha=1}^t t_{1,1}^{(\alpha)}
+\widetilde{T}_{1,2}^{}\widetilde{T}_{1,2}^*\leq 1.
\end{equation}
Therefore,  $\widetilde{J}$ can be rewritten as 
\begin{equation}\label{prooflemma17}
\begin{aligned}
\widetilde{J}&=\sum_{\alpha=1}^t\Big(
c_{\alpha}\log\big(
t_{1,1}^{(\alpha)}
\big)-f_{\alpha}
t_{1,1}^{(\alpha)}
\Big)
+|z_0|^2\big(\sum_{\alpha=1}^t t_{1,1}^{(\alpha)}
+\widetilde{T}_{1,2}^{}\widetilde{T}_{1,2}^*\big)-
 \widetilde{T}_{1,2}^{}H\widetilde{T}_{1,2}^*,
\end{aligned}
\end{equation} 
where
\begin{equation}\label{prooflemma18}
H:=|z_0|^2\widetilde{T}^*\widetilde{T}+F^{\frac{1}{2}}
\big( z_0\mathbb{I}_l-L \big)^*\big( z_0\mathbb{I}_l-L \big)
F^{\frac{1}{2}}-F^{\frac{1}{2}}
L\widetilde{T}^*\widetilde{T}L^*
F^{\frac{1}{2}}.
\end{equation}
Since   $L^*L=LL^*$,   by \eqref{prooflemma14}  a direct calculation shows that 
\begin{equation*}
H 
=\big( z_0\mathbb{I}_l-F^{\frac{1}{2}}LF^{\frac{1}{2}} \big)
\big( z_0\mathbb{I}_l-F^{\frac{1}{2}}LF^{\frac{1}{2}} \big)^*
\geq 0.
\end{equation*}
Together with \eqref{prooflemma15} and \eqref{logine},  we  obtain 
\begin{equation} \widetilde{J} \leq   t_0+|z_0|^2 -1+\sum_{\alpha=1}^t  c_{\alpha}
\log\frac{c_{\alpha}}{f_{\alpha}+t_0}, \end{equation} 
from which 
using of  \eqref{prooflemma7}  gives rise to  
\begin{equation} \label{Jnless}
J_n<n\sum_{\alpha=1}^t c_{\alpha}
\log\frac{c_{\alpha}}{|z_0-a_{\alpha}|^2+t_0}
+n(t_0+|z_0|^2-1),
\end{equation} 
Here  the equality can not be taken since so does not  for  ${J_{n-1}}$  in the case that $\mathbb{I}_{n-1}- \widetilde{T} \widetilde{T}^*> 0$.

 When $\mathbb{I}_{n-1}- \widetilde{T} \widetilde{T}^*$ is just a  nonnegative but not positive  definite  matrix,  without loss of generality we assume that 
$$\mathbb{I}_{n-1}- \widetilde{T} \widetilde{T}^*= 
\mbox{diag}(P,  0_{l_0}), \quad l_0\geq 1,
$$
where  $P$ is a positive definite  matrix and  $0_{l_0}$ is a zero matrix of size $l_0$.  Then  the nonnegative definiteness in  \eqref{pdef}  implies that  all the last $l_0$ components of the  vector $T_{1,2} \widetilde{T}^*$ are zeros.  We consider  the matrix $P$ and  further proceed as like   $\mathbb{I}_{n-1}- \widetilde{T} \widetilde{T}^*>0$ to prove  that \\
 (i) \eqref{Jnless}  holds when $n-1>l_0\geq 1$;  (ii) \eqref{Jn inequlity}  holds and the equality  can be taken as required  when $l_0=n-1$.    
 
In summary, the desired result  immediately follows after { \bf  Base }  and { \bf   Inductive Steps}. 
\end{proof}

%%%%%%%%%%%%%%%%%%%%%%%%%%%%%%%%%%%%%%%%%%%%%%%%%%%%%%%%%%%%%%%%%%%%%%%%%%%%%%%%%%%%%%%%%%%%%%%%%%%%%%%%%%%%%%%%%%%%%%%%%%%%%%%%%%%%%%%%%%%%%%%%%%%%%%%%%%%%%%%%%%%%%%%%%%%%%%%%%%%%%%%%%%%%%%%%%%%%%%%%%%%%%%%%%%%%%%%%%%%%%%
%%%%%%%%%%%%%%%%%%%%%%%%%%%%%%%%%%%%%%%%%%%%%%%%%%%%%%%%%%%%%%%%%%%%%%%%%%%%%5

\section{Proof of Theorem \ref{2-complex-correlation critical2}} \label{sect3i}
The goal in this section is to complete the  proof of Theorem \ref{2-complex-correlation critical2}. 
Choose $\tau=1+N^{-\frac{1}{2}}\hat{\tau}$, then 
\begin{equation}\label{tau inverse expansion}
\frac{1}{\tau}
=1-\hat{\tau}N^{-\frac{1}{2}}+\hat{\tau}^2N^{-1}+
O\big( N^{-\frac{3}{2}} \big).
\end{equation} 

As mentioned in Section \ref{sec1.3},  the  primary  task  of this section  is to analyze  the matrix integral  $I_{N,\delta}$ in  Proposition \ref{RNn delta}.      To do   this, 
  we   need to  make  a series of changes of matrix variables
and  transform  the integration region    \eqref{integrationregionprop} to a  relative simple form with a few variables. 
   The most relevant variables  are  triangular matrices  $\{T_{\alpha}\}$, equivalently,    
    diagonal matrices  $\{T_{{\rm d},\alpha}\}$ and   strictly upper triangular  matrices  $\{T_{{\rm u},\alpha}\}$ as in \eqref{Talpha 2}.   We   proceed in    three subsections to  do  Taylor expansions  for    $f(\boldsymbol{T},Y,\boldsymbol{Q})$ in  \eqref{fTY}   in seven  steps  and for  $ \det\big(\widehat{L}_1+\sqrt{\gamma_N}\widehat{L}_2\big)$,  and then give   a final  summary.  
%  
%  
%   {\bf  Notation.} %$$\gamma_{N}=\frac{N}{N-n}.$$
%%$f_{\alpha}=|z_0-a_{\alpha}|^2$  in   \eqref{falpha}, 
%\begin{equation*}
%\| G \| :=\sum_{\alpha=1}^t
% \| 
%G_\alpha\|,
%\end{equation*}
%and 
%\begin{equation*} \| 
%  \widetilde{T}_{{\rm d}}\|_{2} :=\sum_{\alpha=2}^t
% \| 
%  \widetilde{T}_{{\rm d},\alpha}\|, \quad  \|\widetilde{\mathcal{T}}_{{\rm d}}\|_{2} :=\sum_{\alpha=2}^t
% \| \widetilde{\mathcal{T}}_{{\rm d},\alpha}\|,\quad   \| 
%G\|_{2} :=\sum_{\alpha=2}^t
% \| 
%G_\alpha\|,
%\end{equation*}  cf \eqref{2-t-sum}.

\subsection{Taylor expansion of $f(\boldsymbol{T},Y,\boldsymbol{Q})$} \label{sect3.1}

With the notation $f_{\alpha}$ in \eqref{falpha},   
 noticing the  integration domain \eqref{regiondelta} and the decomposition of     $T_{\alpha}$   as  sum of  a diagonal matrix $\sqrt{T_{{\rm d},\alpha}}$ and a strictly upper triangular  matrix ${T_{{\rm u},\alpha}}$ as in \eqref{Talpha 2}, 
  introduce new matrix variables 
\begin{equation}\label{Tdalpha1 critical non0}
T_{{\rm d},\alpha}=\frac{c_{\alpha}}{f_{\alpha}}
\mathbb{I}_n+\widetilde{T}_{{\rm d},\alpha},
\quad
T_{{\rm u},\alpha}=\frac{1}{\sqrt{f_{\alpha}}}\hat{T}_{{\rm u},\alpha},    \quad \alpha=1,\cdots,t.
\end{equation}
It's easy to obtain 
\begin{equation}\label{Talphaexpansion1 critical non00}
N^{-1}\sum_{\alpha=1}^t 
\sum_{j=1}^n (R_{\alpha,N}-j)\log\big(
t_{j,j}^{(\alpha)}
\big)=N^{-1}\sum_{\alpha=1}^t
\Big(
nR_{\alpha,N}-\frac{n(n+1)}{2}
\Big)\log\frac{c_{\alpha}}{f_{\alpha}}
+O\Big(
\sum_{\alpha=1}^t \|
\widetilde{T}_{{\rm d},\alpha}
\|
\Big),
\end{equation}
\begin{multline}\label{Talphaexpansion1 critical non0}
c_{\alpha}{\rm Tr}\log T_{{\rm d},\alpha}-f_{\alpha}{\rm Tr}(T_{{\rm d},\alpha})
=
nc_{\alpha}\big(
\log \frac{c_{\alpha}}{f_{\alpha}}-1
\big)-\frac{ f_{\alpha}^2}{2c_{\alpha}}{\rm Tr}\big(
\widetilde{T}_{{\rm d},\alpha}^2 \big)
\\
+\frac{f_{\alpha}^3}{3c_{\alpha}^2 }{\rm Tr}\big(
\widetilde{T}_{{\rm d},\alpha}^3 \big)
-\frac{ f^4_{\alpha}}{4c_{\alpha}^3}{\rm Tr}\big(
\widetilde{T}_{{\rm d},\alpha}^4 \big)
+O\big(
\|
\widetilde{T}_{{\rm d},\alpha}
\|^5
\big),
\end{multline}
\begin{equation}\label{Talphaexpansion2 critical non0}
\sqrt{T_{{\rm d},\alpha}}=\sqrt{\frac{c_{\alpha}}{f_{\alpha}}}
\Big(
\mathbb{I}_n+\frac{ f_{\alpha}}{2c_{\alpha}} \widetilde{T}_{{\rm d},\alpha}
-\frac{ f_{\alpha}^2}{8c_{\alpha}^2} \widetilde{T}_{{\rm d},\alpha}^2
+O\big(
\|
\widetilde{T}_{{\rm d},\alpha}
\|^3
\big)
\Big),
\end{equation}
\begin{multline}\label{Talphaexpansion3 critical non0}
T_{\alpha}T_{\alpha}^*=
\frac{c_{\alpha}}{f_{\alpha}}\mathbb{I}_n+\widetilde{T}_{{\rm d},\alpha}+
\frac{\sqrt{c_{\alpha}}}{f_{\alpha}}  \Big(
\mathbb{I}_n+\frac{ f_{\alpha}}{2c_{\alpha}} \widetilde{T}_{{\rm d},\alpha}
-\frac{ f_{\alpha}^2}{8c_{\alpha}^2} \widetilde{T}_{{\rm d},\alpha}^2
+O\big(
\|
\widetilde{T}_{{\rm d},\alpha}
\|^3
\big)
\Big)
\hat{T}_{{\rm u},\alpha}^*   
\\
+\frac{\sqrt{c_{\alpha}}}{f_{\alpha}}
\hat{T}_{{\rm u},\alpha}
 \Big(
\mathbb{I}_n+\frac{ f_{\alpha}}{2c_{\alpha}} \widetilde{T}_{{\rm d},\alpha}
-\frac{ f_{\alpha}^2}{8c_{\alpha}^2} \widetilde{T}_{{\rm d},\alpha}^2
+O\big(
\|
\widetilde{T}_{{\rm d},\alpha}
\|^3
\big)
\Big)
+\frac{1}{f_{\alpha}}  \hat{T}_{{\rm u},\alpha}\hat{T}_{{\rm u},\alpha}^*,
\end{multline}
and  for $i<j$,
\begin{multline}\label{Talphaexpansion4 critical non0}
\sum_{\alpha=1}^t \big(a_{\alpha}T_{\alpha}T_{\alpha}^*
\big)_{i,j}
=\sum_{\alpha=1}^t \frac{a_{\alpha}\sqrt{c_{\alpha}}}{f_{\alpha}}
\Big(
\hat{T}_{{\rm u},\alpha}
\big(
\mathbb{I}_n+\frac{ f_{\alpha}}{2c_{\alpha}} \widetilde{T}_{{\rm d},\alpha}
-\frac{ f_{\alpha}^2}{8c_{\alpha}^2} \widetilde{T}_{{\rm d},\alpha}^2
+O\big(
\|
\widetilde{T}_{{\rm d},\alpha}
\|^3
\big)
\big)
\Big)_{i,j}
\\
+\sum_{\alpha=1}^t \frac{a_{\alpha}}{f_{\alpha}}
\big(
\hat{T}_{{\rm u},\alpha}\hat{T}_{{\rm u},\alpha}^*
\big)_{i,j}.
\end{multline} 
Moreover, we see from the boundary condition $P_{00}(z_0)=1$ that  the restriction condition in  \eqref{integrationregionprop} becomes
\begin{equation}\label{integral region critical non0} 
 \sum_{\alpha=1}^t  \widetilde{T}_{{\rm d},\alpha}
+
\sum_{\alpha=1}^t \frac{\sqrt{c_{\alpha}}}{f_{\alpha}} \big(  
\hat{T}_{{\rm u},\alpha}
+ 
\hat{T}_{{\rm u},\alpha}^*\big)+S+R+R^*\leq 0,
\end{equation} 
where
\begin{equation}\label{S critical non0}
S=\sum_{\alpha=1}^t\frac{1}{f_{\alpha}}
\hat{T}_{{\rm u},\alpha}\hat{T}_{{\rm u},\alpha}^*
+Q_{t+1}Q_{t+1}^*+Q_{0}Q_{0}^*
\end{equation} 
and  a  strictly upper triangular matrix
\begin{equation}\label{R critical non0}
R=\sum_{\alpha=1}^t \frac{\sqrt{c_{\alpha}}}{f_{\alpha}}
\hat{T}_{{\rm u},\alpha}
\Big(
\frac{ f_{\alpha}}{2c_{\alpha}} \widetilde{T}_{{\rm d},\alpha}
-\frac{ f_{\alpha}^2}{8c_{\alpha}^2} \widetilde{T}_{{\rm d},\alpha}^2
+O\big(
\|
\widetilde{T}_{{\rm d},\alpha}
\|^3
\big)\Big).
\end{equation} 

On the other hand, we have  %In view of \eqref{Talphaexpansion3 critical non0},
\begin{multline}\label{Talphaexpansion6 critical non0}
N^{-\frac{1}{4}}
\sum_{\alpha=1}^t\Big(
\overline{a}_{\alpha}{\rm Tr}\big( \hat{Z}
T_{\alpha}T_{\alpha}^* \big)
+a_{\alpha}{\rm Tr}\big( \hat{Z}^*T_{\alpha}T_{\alpha}^* \big)
\Big)      
\\
=N^{-\frac{1}{4}}\sum_{\alpha=1}^t \Big(
\frac{c_{\alpha}}{f_{\alpha}}{\rm Tr}(H_{a_\alpha})
+
{\rm Tr}\big( H_{a_\alpha}\widetilde{T}_{{\rm d},\alpha} \big)
+\frac{1}{f_{\alpha}}
{\rm Tr}\big( H_{a_\alpha}
\hat{T}_{{\rm u},\alpha}\hat{T}_{{\rm u},\alpha}^* \big)
\Big),
\end{multline}
where
\begin{equation}\label{Halpha critical non0}
H_{a}:=\overline{a}\hat{Z}+
a\hat{Z}^*
%:={\rm diag}\big(H_1^{(\alpha)},\cdots,H_n^{(\alpha)}\big)
.
\end{equation}  

According to \eqref{tau inverse expansion}, we divide the relevant  analysis into  three main  parts: (i) {\bf Steps 1}-{\bf 5} deal with the  diagonal entries ({\bf Steps 1-2} ) and  with strictly upper triangular entries ({\bf Steps 3}-{\bf 5} ) in  $ \sum_{\alpha=1}^t
c_{\alpha}\log\det(T_{{\rm d},\alpha})+h(Q) $ in \eqref{fTY};   (ii) {\bf Step 6} is devoted to  $ (1/\tau-1)h(Q) $; (iii) {\bf Step 7} is devoted to matrix variable $Y$   and summary.  Part (i) seems quite complicated and a lot of notations are involved. % See \eqref{fTY expansion critical non0}-\eqref{f2 critical non0}.
%%%%%%%%%%%%%%%%%%%%%%%%%%%%%%%%%%%%%%%%%%%%%%%%%%%%%%%%%%%%%%%%%%%%%%%%%%%%%%%%%%%%%%%%%%%%%%%%%%%%%%%%%%%%%%%%%%%%%%%%%%%%%%%%%%%%%%%%%%%%5

{\bf Step 1: First transformation of  diagonal entries.}  Let  $S^{({\rm d})}$ be a  diagonal  matrix extracted   from the diagonal part of $S$ given  in \eqref{S critical non0},  introduce a new diagonal matrix   $\Gamma_1$ and make change of variables  from 
 $\widetilde{T}_{{\rm d},1}, \widetilde{T}_{{\rm d},2}, \cdots, \widetilde{T}_{{\rm d},t}$ to 
\begin{equation}\label{matrix transformations1 critical non0}
\begin{aligned}
&\Gamma_1:=\sum_{\alpha=1}^t    \widetilde{T}_{{\rm d},\alpha}+S^{({\rm d})},
\quad \widetilde{T}_{{\rm d},2},\, \ldots, \widetilde{T}_{{\rm d},t}.
%\alpha=2,\cdots,t.
\end{aligned}
\end{equation}
We will see that  the typical size for  $Q_{t+1}, \widetilde{T}_{{\rm d},2},\, \ldots, \widetilde{T}_{{\rm d},t}$ are $N^{-1/2}$   while it  is  $N^{-1}$ for  $\Gamma_1$. So we will do Taylor expansions up to some proper orders  accordingly.

With  $f(T,Y,Q)$ in \eqref{fTY},   combining  \eqref{Talphaexpansion1 critical non0} and  \eqref{Talphaexpansion6 critical non0}, we focus on  the power sums  
\begin{equation}\label{powersum}
\begin{aligned}
 PS_k:=\sum_{\alpha=1}^t\frac{f^k_{\alpha}}{k c^{k-1}_{\alpha}}{\rm Tr}\big(
\widetilde{T}_{{\rm d},\alpha}^k \big), \quad k=1,2,3,4.
\end{aligned}
\end{equation}
It's easy to see from \eqref{S critical non0}  that  
\begin{equation}\label{linear critical non0}
|z_0|^2\Big(
\sum_{\alpha=1}^t {\rm Tr}\big( T_{\alpha}T_{\alpha}^* \big)
+{\rm Tr}\big( Q_{t+1}Q_{t+1}^* \big)
+{\rm Tr}\big( Q_{0}Q_{0}^* \big)
\Big)=n|z_0|^2+|z_0|^2{\rm Tr}\big( \Gamma_1 \big),
\end{equation}
and 
\begin{multline}\label{Talphaexpansion5 critical non0}
-\sum_{\alpha=1}^t\frac{f_{\alpha}^2}{2c_{\alpha}}{\rm Tr}\big(
\widetilde{T}_{{\rm d},\alpha}^2 \big)
=-\frac{f_{1}^2}{2c_{1}}{\rm Tr}\big(
\sum_{\alpha=2}^t \widetilde{T}_{{\rm d},\alpha}+S^{({\rm d})}
 \big)^2-\sum_{\alpha=2}^t
 \frac{f_{\alpha}^2}{2c_{\alpha}}{\rm Tr}\big(
  \widetilde{T}_{{\rm d},\alpha}^2
 \big)
 \\
 +O\Big(
\| \Gamma_1 \|\Big( 
 \| \widetilde{T}_{{\rm d}}\|_{2} 
+\| S^{({\rm d})} \| +\| \Gamma_1 \|\Big)
 \Big),
\end{multline}
where 
\begin{equation} \label{2-t-sum}
 \| \widetilde{T}_{{\rm d}}\|_{2} :=\sum_{\alpha=2}^t\| 
  \widetilde{T}_{{\rm d},\alpha}\|. 
\end{equation}
 Rewriting the second term on the RHS of \eqref{Talphaexpansion6 critical non0}
\begin{equation}\label{Talphaexpansion7 critical non0}
N^{-\frac{1}{4}}\sum_{\alpha=1}^t
{\rm Tr}\big( H_{a_\alpha}\widetilde{T}_{{\rm d},\alpha} \big)
= N^{-\frac{1}{4}}\sum_{\alpha=2}^t 
{\rm Tr}\big( H_{a_\alpha-a_1}\widetilde{T}_{{\rm d},\alpha} \big)
-N^{-\frac{1}{4}}{\rm Tr}\big( H_{a_1} S^{({\rm d})} \big)
+O\big( N^{-\frac{1}{4}} \| \Gamma_1 \| \big),
\end{equation}  
together with \eqref{Talphaexpansion5 critical non0} 
we  obtain in a quadratic  form
\begin{multline}\label{Talphaexpansion8 critical non0}
-\sum_{\alpha=1}^t\frac{f_{\alpha}^2}{2c_{\alpha}}{\rm Tr}\big(
\widetilde{T}_{{\rm d},\alpha}^2 \big)
+N^{-\frac{1}{4}}\sum_{\alpha=1}^t
{\rm Tr}\big( H_{a_\alpha}\widetilde{T}_{{\rm d},\alpha} \big)
=-{\rm Tr}\big( (X-\frac{1}{2}B\Sigma^{-1})\Sigma (X-\frac{1}{2}B\Sigma^{-1})^t\big)+ \frac{1}{4}{\rm Tr}\big( B\Sigma^{-1}B^t\big)%-{\rm Tr}\big( X\Sigma X^t\big)+ {\rm Tr}\big( XB^t\big)+
\\
-N^{-\frac{1}{4}}{\rm Tr}(H_{a_1}S^{({\rm d})})
-\frac{f_{1}^2}{2c_{1}}{\rm Tr}\big( \big( S^{({\rm d})} \big)^2\big) +O\Big(
\| \Gamma_1 \|\Big( 
  \| \widetilde{T}_{{\rm d}}\|_{2}  +\| \Gamma_1 \|+N^{-1/4}\Big)
 \Big).
\end{multline}  
where
\begin{equation}  \label{Sigma}
\Sigma=[\Sigma_{\alpha,\beta}]_{\alpha,\beta=2}^t,  \quad \Sigma_{\alpha,\beta}=\frac{f_{\alpha}^2}{2c_{\alpha}}
\delta_{\alpha,\beta}+\frac{f_{1}^2}{2c_{1}},
\end{equation} 
and 
 two  $n\times (t-1)$ matrices 
\begin{equation}\label{XB}
X=\begin{bmatrix}
\big(
\widetilde{T}_{{\rm d},1+j} \big)_{i,i} 
\end{bmatrix}_{i=1, \ldots,n\atop
j=1,\ldots, t-1},  \quad B=\begin{bmatrix}
\big(
N^{-\frac{1}{4}}
H_{a_{1+j}-a_1}
-\frac{f_{1}^2}{c_{1}} S^{({\rm d})} \big)_{i,i} 
\end{bmatrix}_{i=1, \ldots,n\atop
j=1,\ldots, t-1}
\end{equation}
are chosen from  entries of certain diagonal  matrices.

Noticing 
\begin{equation}\label{Ainverse critical non0}
(\Sigma^{-1})_{\alpha,\beta}=
\frac{2c_{\alpha}\delta_{\alpha,\beta}}{f_{\alpha}^2}
-\frac{2c_{\alpha}c_{\beta}}{f_{\alpha}^2f_{\beta}^2P_1},
\quad
\alpha,\beta=2,\cdots,t,
\end{equation} 
where $P_1$ is given in  \eqref{parameter2}, 
we get 
\begin{multline}\label{bi A bit critical non0}
\frac{1}{4}{\rm Tr}\big( B\Sigma^{-1}B^t\big)
=\frac{1}{2}\sum_{\alpha=2}^t {\rm Tr}\Big( 
N^{-\frac{1}{4}}
H_{a_{\alpha}-a_1}
-\frac{f_{1}^2}{c_{1}} S^{({\rm d})} 
\Big)^2           \\
-\frac{1}{2P_1}{\rm Tr}\Big( 
N^{-\frac{1}{4}}\sum_{\alpha=1}^t 
\frac{c_{\alpha}}
{f_{\alpha}^2}\big(H_{a_{\alpha}-z_0}-H_{a_{1}-z_0}
\big) -\frac{f_{1}^2}{c_{1}}S^{({\rm d})}
\sum_{\alpha=2}^t\frac{c_{\alpha}}{f_{\alpha}^2} 
\Big)^2.
\end{multline}
Since $z_0$ is a critical point such that $P_{0}(z_0)=0$,  by definition of   $P_1$  in  \eqref{parameter2} and $H_a$ in \eqref{Halpha critical non0} 
\begin{align}\label{bi A bit1 critical non0}
\frac{1}{4}{\rm Tr}\big( B\Sigma^{-1}B^t\big)
&=\frac{1}{2} N^{-\frac{1}{2}}  \sum_{\alpha=2}^t 
{\rm Tr}\big( H_{a_{\alpha}-z_0}^2\big)         +\Big(\frac{f_{1}^2}{2c_{1}}-\frac{1}{2P_1} \Big)
{\rm Tr}\big(S^{({\rm d})} \big)^2+N^{-\frac{1}{4}}{\rm Tr}\big(H_{a_{1}-z_0}S^{({\rm d})}\big),
\end{align} 
from which simple calculation yields 
\begin{multline}\label{bi A bit3 critical non0}
\frac{1}{4}{\rm Tr}\big( B\Sigma^{-1}B^t\big)-N^{-\frac{1}{4}}{\rm Tr}\big(H_{a_{1}}S^{({\rm d})}\big)
-\frac{f_{1}^2}{2c_{1}} 
{\rm Tr}\big(S^{({\rm d})} \big)^2
\\
=\frac{1}{2}N^{-\frac{1}{2}}\sum_{\alpha=1}^t 
\frac{c_{\alpha}}{f_{\alpha}^2} 
\Big( (a_{\alpha}-z_0)^2{\rm Tr}\big( \hat{Z}^* \big)^2 
+\overline{(a_{\alpha}-z_0)^2}{\rm Tr}\big( \hat{Z}^2 \big)
\Big)
+N^{-\frac{1}{2}} {\rm Tr}\big( \hat{Z}\hat{Z}^* \big)
\\
-N^{-\frac{1}{4}}\Big(
z_0{\rm Tr}\big( \hat{Z}^*S^{({\rm d})} \big)
+\overline{z}_0{\rm Tr}\big( \hat{Z}S^{({\rm d})} \big)
\Big)
-\frac{1}{2P_1}{\rm Tr}\big( S^{({\rm d})} \big)^2.
\end{multline}
Note  that $S^{({\rm d})}$ is the diagonal part of $S$  defined in  \eqref{S critical non0}, adding 
  the third  term on the RHS of \eqref{Talphaexpansion6 critical non0} we get 
\begin{multline}\label{Tr hatZ Sd1}
-N^{-\frac{1}{4}}\Big(
z_0{\rm Tr}\big( \hat{Z}^*S^{({\rm d})} \big)
+\overline{z}_0{\rm Tr}\big( \hat{Z}S^{({\rm d})} \big)
\Big) +N^{-\frac{1}{4}}\sum_{\alpha=1}^t \frac{1}{f_{\alpha}}
{\rm Tr}\big( H_{a_{\alpha}}\hat{T}_{{\rm u},\alpha}\hat{T}_{{\rm u},\alpha}^* \big)     
   \\
=N^{-\frac{1}{4}}\sum_{\alpha=1}^t
\frac{1}{f_{\alpha}}
\Big(
\overline{a_{\alpha}-z_0}
{\rm Tr}\big( \hat{Z}\hat{T}_{{\rm u},\alpha}\hat{T}_{{\rm u},\alpha}^* \big)+(a_{\alpha}-z_0)
{\rm Tr}\big( \hat{Z}^*
\hat{T}_{{\rm u},\alpha}\hat{T}_{{\rm u},\alpha}^* \big)
\Big)
\\
-N^{-\frac{1}{4}}\Big(
z_0{\rm Tr}\big(
\hat{Z}^*Q_0Q_0^*
\big)+
\overline{z}_0 
{\rm Tr}\big(
\hat{Z} Q_0Q_0^*
\big)
\Big)
+O\big( \| Q_{t+1} \|^2 \big).
\end{multline}

Therefore, combination of  \eqref{Talphaexpansion8 critical non0}, \eqref{bi A bit3 critical non0} and \eqref{Tr hatZ Sd1} gives rise to 
\begin{multline}\label{2powerterm}
-\sum_{\alpha=1}^t\frac{f_{\alpha}^2}{2c_{\alpha}}{\rm Tr}\big(
\widetilde{T}_{{\rm d},\alpha}^2 \big)
+N^{-\frac{1}{4}}\sum_{\alpha=1}^t
{\rm Tr}\big( H_{a_\alpha}\widetilde{T}_{{\rm d},\alpha} \big)+N^{-\frac{1}{4}}\sum_{\alpha=1}^t \frac{1}{f_{\alpha}}
{\rm Tr}\big( H_{a_{\alpha}}\hat{T}_{{\rm u},\alpha}\hat{T}_{{\rm u},\alpha}^* \big)    
\\
=-{\rm Tr}\Big( (X-\frac{1}{2}B\Sigma^{-1})\Sigma (X-\frac{1}{2}B\Sigma^{-1})^t\Big)-\frac{1}{2P_1}{\rm Tr}\big( S^{({\rm d})} \big)^2+N^{-\frac{1}{2}} {\rm Tr}\big( \hat{Z}\hat{Z}^* \big)
\\+\frac{1}{2}N^{-\frac{1}{2}}\sum_{\alpha=1}^t 
\frac{c_{\alpha}}{f_{\alpha}^2} 
\Big( (a_{\alpha}-z_0)^2{\rm Tr}\big( \hat{Z}^* \big)^2 
+\overline{(a_{\alpha}-z_0)^2}{\rm Tr}\big( \hat{Z}^2 \big)
\Big)
+
N^{-\frac{1}{4}}\sum_{\alpha=1}^t
\frac{1}{f_{\alpha}}
{\rm Tr}\big( H_{a_\alpha-z_0} \hat{T}_{{\rm u},\alpha}\hat{T}_{{\rm u},\alpha}^* \big) 
   \\
+O\Big(
\| \Gamma_1 \|\Big( 
 \| \widetilde{T}_{{\rm d}}\|_{2}
+\| S^{({\rm d})} \| +\| \Gamma_1 \|+N^{-1/4}\Big)
 \Big)+O\big( \| Q_{t+1} \|^2 \big).
\end{multline}

{\bf Step 2: Second transformation of diagonal entries.}  With  two  $n\times (t-1)$ matrices $X$ and $B$  defined in \eqref{XB}, 
 let's  introduce     $t-1$ new diagonal matrices     such that 
$\widetilde{\mathcal{T}}_{\rm d,2},  \ldots, \widetilde{\mathcal{T}}_{\rm d,t}$ 
\begin{equation}\label{newT}
X-\frac{1}{2}B\Sigma^{-1}=\begin{bmatrix}
\big(
\widetilde{\mathcal{T}}_{{\rm d},1+j} \big)_{i,i} 
\end{bmatrix}_{i=1, \ldots,n\atop
j=1,\ldots, t-1}, \end{equation}
then it's easy to see from \eqref{XB} and \eqref{Ainverse critical non0} that  
\begin{equation}\label{newT-2}
\widetilde{\mathcal{T}}_{\rm d,\alpha}=  \widetilde{T}_{{\rm d},\alpha}  -
\frac{c_{\alpha}}{f_{\alpha}^2}
\Big(
N^{-\frac{1}{4}} H_{  a_{\alpha}-a_{1}} -\frac{f_{1}^2}{c_{1}}S^{({\rm d})}
-\frac{1}{P_1}\sum_{\beta=2}^t 
\frac{c_{\beta}}{f_{\beta}^2}
\big( N^{-\frac{1}{4}} H_{  a_{\beta}-a_{1}} -\frac{f_{1}^2}{c_{1}}S^{({\rm d})} \big)
\Big) \end{equation}  
for $\alpha=2. \ldots, t$, which further  reduces to  
\begin{equation}\label{newT-3}
\widetilde{\mathcal{T}}_{\rm d,\alpha}=  \widetilde{T}_{{\rm d},\alpha}  -
\frac{c_{\alpha}}{f_{\alpha}^2}
\Big(
N^{-\frac{1}{4}} H_{  a_{\alpha}-z_0} 
-\frac{1}{P_1} S^{({\rm d})} 
\Big), \end{equation}  
because of the assumption  $P_0(z_0)=0$.  Moreover,
\begin{equation}\label{newT-4}
\sum_{\alpha=2}^t \big(\widetilde{\mathcal{T}}_{\rm d,\alpha}-\widetilde{T}_{{\rm d},\alpha}\big)=
N^{-\frac{1}{4}}\frac{c_{1}}{f_{1}^2}
H_{a_{1}-z_0}+\Big(
1-\frac{c_1}{P_1f_1^2}\Big)S^{({\rm d})}. \end{equation}  

Next, we express the power sums   $PS_3$ and  $PS_4$ given by \eqref{powersum}  in terms of  new variables $\{\widetilde{\mathcal{T}}_{\rm d,\alpha}\}$ through the relations  \eqref{newT-3}.  For this, 
after doing  simple  expansion, 
we know that  
for $\alpha=2,\cdots,t$,
\begin{multline}\label{Talphaexpansion11 critical non0}
\big(
\widetilde{T}_{{\rm d},\alpha} \big)^3=3N^{-\frac{1}{2}}
\Big(
\widetilde{\mathcal{T}}_{\rm d,\alpha}-
\frac{c_{\alpha}}{P_1f_{\alpha}^2}S^{({\rm d})}
\Big)\frac{c_{\alpha}^2}{f_{\alpha}^4}
\big( H_{a_{\alpha}-z_0} \big)^2  +N^{-\frac{3}{4}}
\frac{c_{\alpha}^3}{f_{\alpha}^6} \big( H_{a_{\alpha}-z_0} \big)^3 
 \\
+O\Big(
\big( \|\widetilde{\mathcal{T}}_{{\rm d}}\|_{2}+\| S^{({\rm d})} \| \big)^2
\big( \|\widetilde{\mathcal{T}}_{{\rm d}}\|_{2}+\| S^{({\rm d})} \|
+N^{-\frac{1}{4}} \big)
\Big),
\end{multline}
and  by  \eqref{newT-4}
\begin{multline}\label{Talphaexpansion10 critical non0}
\big(
\sum_{\alpha=2}^t  \widetilde{T}_{{\rm d},\alpha} +S^{({\rm d})}\big)^3
=3N^{-\frac{1}{2}}
\Big(
\sum_{\alpha=2}^t \widetilde{\mathcal{T}}_{\rm d,\alpha}+
\frac{c_1}{P_1f_1^2}S^{({\rm d})}
\Big)\frac{c_{1}^2}{f_{1}^4}
\big(  H_{a_1-z_0}\big)^2    
\\
-N^{-\frac{3}{4}}
\frac{c_{1}^3}{f_{1}^6}\big(H_{a_1-z_0} \big)^3
+O\Big(
\big(\|\widetilde{\mathcal{T}}_{{\rm d}}\|_{2}+\| S^{({\rm d})} \| \big)^2
\big(\|\widetilde{\mathcal{T}}_{{\rm d}}\|_{2}+\| S^{({\rm d})} \|
+N^{-\frac{1}{4}} \big)
\Big),
\end{multline} 
where
 \begin{equation} \label{cal2-t-sum}  
 \|\widetilde{\mathcal{T}}_{{\rm d}}\|_{2} :=\sum_{\alpha=2}^t
 \| \widetilde{\mathcal{T}}_{{\rm d},\alpha}\|,  
\end{equation} 

  By definition of   $P_1$  in  \eqref{parameter2},  it's easy to see 
\begin{equation} \label{H2power}
\sum_{\alpha=1}^t\frac{c_{\alpha}}{f_{\alpha}^3}
\big( H_{a_{\alpha}-z_0}\big)^2=P_2\overline{\hat{Z}^2}
+2P_1 \hat{Z} \overline{Z} +\overline{P}_2\hat{Z}^2,
\end{equation} 
where  $P_2$ is defined in \eqref{parameter2}.
Again with  $P_{0}(z_0)=0$,  combination of   \eqref{Talphaexpansion11 critical non0} and  \eqref{Talphaexpansion10 critical non0} yields 
\begin{multline}\label{Talphaexpansion12 critical non0}
\sum_{\alpha=1}^t
\frac{f_{\alpha}^3}{3c_{\alpha}^2}{\rm Tr}\big(
\widetilde{T}_{{\rm d},\alpha}^3 \big)=
\frac{1}{3}N^{-\frac{3}{4}}\sum_{\alpha=1}^t
\frac{c_{\alpha}}{f_{\alpha}^3}\Big(
(a_{\alpha}-z_0)^3{\rm Tr}\big( \hat{Z}^* \big)^3+
\overline{a_{\alpha}-z_0}^3{\rm Tr}\big( \hat{Z}^3 \big)
\Big)        
\\
+N^{-\frac{1}{2}} {\rm Tr} \left(
\sum_{\alpha=2}^t  \widetilde{\mathcal{T}}_{\rm d,\alpha}
\Big(
\frac{1}{f_{\alpha}} \big( H_{a_{\alpha}-z_0} \big)^2
-\frac{1}{f_{1}}\big(   H_{a_{z}-z_0} \big)^2
\Big)-\frac{1}{P_1}{\rm Tr}
\Big(
S\big( P_2 (\hat{Z}^*)^2
+2P_1\hat{Z}\hat{Z}^*+\overline{P}_2\hat{Z}^2 \big)
\Big)
\right)         \\
+
O\Big(
\big( \|\widetilde{\mathcal{T}}_{{\rm d}}\|_{2}+ \| \Gamma_1  \| +N^{-\frac{1}{4}}\big)^2
 \| \Gamma_1  \|
\Big)+
O\Big(
\big( \|\widetilde{\mathcal{T}}_{{\rm d}}\|_{2}+\| S^{({\rm d})} \| \big)^2
\big(\|\widetilde{\mathcal{T}}_{{\rm d}}\|_{2} +\| S^{({\rm d})} \|
+N^{-\frac{1}{4}} \big)
\Big).
\end{multline}

Similarly,  we can obtain 
\begin{multline}\label{Talphaexpansion13 critical non0}
-\sum_{\alpha=1}^t
\frac{f_{\alpha}^4}{4c_{\alpha}^3}{\rm Tr}\big(
\widetilde{T}_{{\rm d},\alpha}^4 \big)=
-\frac{1}{4N}\sum_{\alpha=1}^t
\frac{c_{\alpha}}{f_{\alpha}^4}{\rm Tr}\big(
H_{a_{\alpha}-z_0}^4
 \big)+ 
O\Big(
\big(\|\widetilde{\mathcal{T}}_{{\rm d}}\|_{2}+\| S^{({\rm d})} \|+ \| \Gamma_1  \| +N^{-\frac{1}{4}}\big)^3
 \| \Gamma_1  \|
\Big)\\
+
O\Big(
\big(\|\widetilde{\mathcal{T}}_{{\rm d}}\|_{2}+\| S^{({\rm d})} \| \big)^3
\big(\|\widetilde{\mathcal{T}}_{{\rm d}}\|_{2}+\| S^{({\rm d})} \|
+N^{-\frac{1}{4}} \big)
\Big).
\end{multline}

%%%%%%%%%%%%%%%%%%%%%%%%%%%%%%%%%%%%%%%%%%%%%%%%%%%%%%%%%%%%%%%%%%%%%%%%%%%%%%%%%%%%%%%%%%%%%%%%%%%%%%%%%%%%%%%%%%%%%%%%%%%%%%%%%%%%%%%%%%%%%%%%%%%

{\bf Step 3: Transformation of strictly  upper  triangular entries.}  We deal with the  strictly   triangular entries of $\{T_{\alpha}\}$; cf \eqref{Talphaexpansion4 critical non0}. 
By  boundary condition   $P_{00}(z_0)=1$, the critical  condition   $P_{0}(z_0)=0$ and the  definition of $ P_1$, we can introduce a unitary matrix  with  two chosen  columns   
\begin{equation}\label{V critical non0}
V=[v_{\alpha,\beta}]_{\alpha,\beta=1}^t,\quad
v_{\alpha,1}=\frac{\overline{a_{\alpha}-z_0}\sqrt{c_{\alpha}}}{f_{\alpha}},
\quad
v_{\alpha,2}=\frac{1}{\sqrt{P_1}}\frac{\sqrt{c_{\alpha}}}{f_{\alpha}}, \quad \alpha=1,\cdots,t.
\end{equation} 
Therefore,   we can define a family of  strictly upper triangular matrices   $\{G_\alpha\}$  in terms of   $\{\hat{T}_{{\rm u},\alpha}\}$, that is,  for $\alpha=1,\cdots,t$, 
\begin{equation}\label{Talphahat galpha}
\hat{T}_{{\rm u},\alpha}=[ \hat{t}_{i,j}^{(\alpha)}]_{i,j=1}^n,   \quad  G_{\alpha}=[ g_{i,j}^{(\alpha)}]_{i,j=1}^n  ,\quad \hat{t}_{i,j}^{(\alpha)}=g_{i,j}^{(\alpha)}=0
 \ \mbox{whenever} \   i\leq j,
\end{equation} 
such that  
\begin{equation}\label{g hatT connection0}
\begin{bmatrix}
\hat{t}_{i,j}^{(1)} \\
    \vdots
    \\
\hat{t}_{i,j}^{(t)}
\end{bmatrix}=V\begin{bmatrix}
g_{i,j}^{(1)}\\
    \vdots
    \\ g_{i,j}^{(t)}
\end{bmatrix},\quad \forall\, i>j.
\end{equation} 
In matrix form this shows that 
\begin{equation}\label{g hatT connection1}
\hat{T}_{{\rm u},\alpha}=\sum_{\beta=1}^t
v_{\alpha,\beta} G_{\beta}.
\end{equation} 
So it immediately follows that the first term on the RHS of  \eqref{Talphaexpansion4 critical non0}  can be rewritten as \begin{equation}\label{change1}
\sum_{\alpha=1}^t \frac{a_{\alpha}\sqrt{c_{\alpha}}}{f_{\alpha}}
\hat{T}_{{\rm u},\alpha}
=G_1+z_0\sqrt{P_1}G_2,
\end{equation} and   the quadratic part  of $\hat{T}_{{\rm u},\alpha}$  in  the first summation  $-\sum_{\alpha} f_{\alpha}{\rm Tr}(T_{\alpha}T_{\alpha}^*)$ in  \eqref{hQrewrite} equals to 
\begin{equation}\label{unitarity critical non0}
-\sum_{\alpha=1}^t {\rm Tr}
\big( \hat{T}_{{\rm u},\alpha}\hat{T}_{{\rm u},\alpha}^* \big)
=-\sum_{\alpha=1}^t {\rm Tr}
\big( G_{\alpha}G_{\alpha}^* \big);
\end{equation} 
cf. \eqref{Talphaexpansion3 critical non0}.

%%%%%%%%%%%%%%%%%%%%%%%%%%%%%%%%%%%%%%%%%%%%%%%%%%%%%%%%%%%%%%%%%%%%%%%%%%%%%%%%%%%%%%%%%%%%%%%%%%%%%%%%%%%%%%%%%%%%%%%%%%%%%%%%%%%%%%%%%%%%%%%%%%%%
In order to  simplify the integration domain in \eqref{integral region critical non0}, let $S^{({\rm off},{\rm u})}$ denote the  strictly upper triangular  part of $S$  in \eqref{S critical non0} and put 
\begin{equation}\label{hat g2 critical non0}
\hat{G}_2=G_2+\frac{1}{\sqrt{P_1}}\big(
S^{({\rm off},{\rm u})}+R
\big),
\end{equation}
then the integration domain reduces to   \begin{equation}\label{region critical non0}
\Gamma_1+\sqrt{P_1}\hat{G}_2+\sqrt{P_1}\hat{G}_2^*
\leq 0;
\end{equation}  
cf.  \eqref{integral region critical non0} and \eqref{matrix transformations1 critical non0}.

On the other hand, noting that  the size of $\Gamma_1$ is  of order $N^{-1}$, combine  \eqref{R critical non0}, \eqref{g hatT connection1} and \eqref{matrix transformations1 critical non0} and we have 
\begin{multline}\label{R expansion1 critical non0}
R=\frac{\sqrt{c_1}}{f_1}\sum_{\beta=1}^t
v_{1,\beta}G_{\beta}\Big(-
\frac{f_1}{2c_1}\big( \sum_{\alpha=2}^t \widetilde{T}_{{\rm d},\alpha}+S^{({\rm d})}  \big)
-\frac{f_1^2}{8c_1^2}
\big( \sum_{\alpha=2}^t \widetilde{T}_{{\rm d},\alpha}+S^{({\rm d})}  \big)^2
\Big)
\\
+\frac{\sqrt{c_1}}{f_1}\sum_{\beta=1}^t
v_{1,\beta}G_{\beta}O\big( \| \Gamma_1 \|+( \|  \widetilde{T}_{{\rm d}}\|_2+\| S^{({\rm d})}  \|)^3 \big)
\\
+\sum_{\alpha=2}^t \frac{\sqrt{c_{\alpha}}}{f_{\alpha}}
\sum_{\beta=1}^t v_{\alpha,\beta}G_{\beta}\Big(
\frac{f_{\alpha}}{2c_{\alpha}}\widetilde{T}_{{\rm d},\alpha}-
\frac{f_{\alpha}^2}{8c_{\alpha}^2}\widetilde{T}_{{\rm d},\alpha}^2
+O\big( \| \widetilde{T}_{{\rm d},\alpha} \|^3 \big)\Big).
\end{multline}
Substitution of  \eqref{newT-3} into \eqref{R expansion1 critical non0},  simple calculation shows  
\begin{multline}\label{R expansion2 critical non0}
R=R_2+R_3+O\Big(  \| G_1 \| \| \Gamma_1 \|+
 \| G \|_2 \big( \| \widetilde{\mathcal{T}}_{{\rm d}} \|_2
+\| \Gamma_1 \|+\| S^{({\rm d})} \|+N^{-\frac{1}{2}}
 \big)
 \Big)        \\
+O\left( 
 \| G_1 \|\Big( N^{-\frac{3}{4}}
 +\big(  \|  \widetilde{\mathcal{T}}_{{\rm d}} \|_2
+\| S^{({\rm d})} \|
 \big)^2+N^{-\frac{1}{4}}\big(  \|  \widetilde{\mathcal{T}}_{{\rm d}} \|_2
+\| S^{({\rm d})} \|
 \big)
  \Big)
 \right),
\end{multline}
where
\begin{equation}\label{R2 critical non0}
R_2=N^{-\frac{1}{4}}G_1
\sum_{\alpha=1}^t 
\frac{c_{\alpha}\overline{a_{\alpha}-z_0}}{2 f_{\alpha}^3}
H_{a_{\alpha}-z_0}=
\frac{1}{2}N^{-\frac{1}{4}} G_1 \big( 
P_1\hat{Z}^*+\overline{P}_2\hat{Z}
 \big),
\end{equation} 
and
\begin{multline}\label{R3 critical non0}
R_3=
\sum_{\alpha=2}^t 
\frac{f_1-f_{\alpha}}{2f_{\alpha}f_1}
 \overline{a_{\alpha}-z_0}
G_{1} \widetilde{\mathcal{T}}_{{\rm d,\alpha}}-\sum_{\alpha=1}^t  \frac{\overline{a_{\alpha}-z_0}c_{\alpha}}{2f_{\alpha}^3}  \Big(
\frac{1}{P_1} G_{1} S^{({\rm d})}  +\frac{N^{-\frac{1}{2}} }{4f_{\alpha}}
G_1 H_{a_{\alpha}-z_0}^2
\Big)\\
+
\sum_{\beta=2}^t N^{-\frac{1}{4}}
\sum_{\alpha=1}^t
\frac{\sqrt{c_{\alpha}}v_{\alpha,\beta}}{2f_{\alpha}^2}
G_{\beta} H_{a_{\alpha}-z_0},
\end{multline} 
where 
\begin{equation} \label{G2-t-sum} 
\|G\|_2:=\sum_{\beta=2}^t \| G_{\beta} \|. 
\end{equation}

Similarly, we can obtain \begin{equation}\label{S expansion critical non0}
S=S_2+S_3+O\Big(
 \| G \|_{2}^2+\| Q_{t+1} \|^2
 \Big),      
\end{equation}
where
\begin{equation}\label{S2 critical non0}
S_2=
P_1 G_1G_{1}^*+Q_0Q_0^*
\end{equation} 
and
\begin{equation}\label{S3 critical non0}
S_3=
\sum_{\alpha=1}^t \sum_{\beta=2}^t   \frac{\sqrt{c_{\alpha}}}{f_{\alpha}^2}\big(\overline{a_{\alpha}-z_0}
\overline{v}_{\alpha,\beta}G_1G_{\beta}^*+
(a_{\alpha}-z_0) 
v_{\alpha,\beta}G_{\beta}G_1^*\big).
\end{equation} 

As to the change of variables   \eqref{hat g2 critical non0},  the above analysis shows that the Jacobian determinant \begin{equation}
\det\Big(\frac{\partial \hat{G}_2}{\partial G_2}\Big)=1+
O\big(\| G \|+\|  \widetilde{\mathcal{T}}_{{\rm d}} \|_2+\|\Gamma_1\| +N^{-\frac{1}{4}}\big),
\end{equation}
which implies \begin{equation}
\det\Big(\frac{\partial G_2}{\partial \hat{G}_2}\Big)=1+
O\big(\| G \|+\|  \widetilde{\mathcal{T}}_{{\rm d}} \|_2+\|\Gamma_1\| +N^{-\frac{1}{4}}\big).
\end{equation}
Combining  
\begin{equation}
\begin{aligned}
&\| G_2 \|\leq \| \hat{G}_2 \|+O\big( 
\|  S^{({\rm off},{\rm u})} \|+\| R \| 
\big)  
=O\big( \Omega_{{\rm error},1} \big)+O(\sqrt{\delta})\| G_2 \|,
\end{aligned}
\end{equation}
where
\begin{equation}\label{OmegaError1}
\Omega_{{\rm error},1}=\| \hat{G}_2 \|+
\big(\| G_1 \|+N^{-\frac{1}{4}}\big)^2+\sum_{\beta=3}^t
\| G_{\beta} \|+ \|  \widetilde{\mathcal{T}}_{{\rm d}} \|_2 
+\| Q_{t+1} \|+\| Q_0 \|^2+\|\Gamma_1\|,
\end{equation} 
we get 
\begin{equation}\label{g2 OmegaError1}
\| G_2 \|=O\big( \Omega_{{\rm error},1} \big),
\end{equation}
and further 
\begin{equation}\label{partial g2}
\det\Big(\frac{\partial G_2}{\partial \hat{G}_2}\Big)=1+
O\Big( \| G_1 \|+\| \hat{G}_2 \|+
\sum_{\beta=3}^t
\| G_{\beta} \|+\|\widetilde{\mathcal{T}}_{{\rm d}} \|_2+
\| Q_{t+1} \|+\|\Gamma_1\|+N^{-\frac{1}{4}}\Big),
\end{equation}
\begin{equation}\label{S error term0}
\begin{aligned}
S%&=O\big( \| g \|^2+\| Q_{t+1} \|^2 \big) \\
&=O\Big( \| G_1 \|^2+\| \hat{G}_2 \|+
\sum_{\beta=3}^t
\| G_{\beta} \|+\|\widetilde{\mathcal{T}}_{{\rm d}} \|_2+
\| Q_{t+1} \|+\|\Gamma_1\| \Big).
\end{aligned}
\end{equation}

Finally, re-expressing the error terms we see from 
  \eqref{R expansion2 critical non0} and  \eqref{S expansion critical non0} that  \begin{equation}\label{R error term}
\begin{aligned}
R=R_2+R_3
+O\big( \| G_1 \|\big( N^{-\frac{3}{4}}
+\| \Gamma_1 \|
 \big)+\Omega_{{\rm error},1}^2 \big),
\end{aligned}
\end{equation}
and  
\begin{equation}\label{S error term1}
\begin{aligned}
S=S_2+S_3
+O\big(\Omega_{{\rm error},1}^2 \big).
\end{aligned}
\end{equation}

%%%%%%%%%%%%%%%%%%%%%%%%%%%%%%%%%%%%%%%%%%%%%%%%%%%%%%%%%%%%%%%%%%%%%%%

{\bf Step 4:  Calculating  the square summation  %$\sum_{1\leq i<j\leq n}$
 in  $h(\boldsymbol{T},\boldsymbol{Q})$}.  %\eqref{hQrewrite} 
 Recalling  \eqref{Talphaexpansion4 critical non0} and \eqref{change1},  we have for $1\leq i<j\leq n$,
\begin{multline}\label{Triangular1 critical non0}
\left(
\sum_{\alpha=1}^t a_{\alpha}T_{\alpha}T_{\alpha}^*
+Q_{t+1}A_{t+1}Q_{t+1}^*+z_0 Q_0Q_0^*
\right)_{i,j}
=\big(G_1+z_0\sqrt{P_1}G_2\big)_{i,j}
+\sum_{\alpha=1}^t \frac{a_{\alpha}}{f_{\alpha}}
\big(
\hat{T}_{{\rm u},\alpha}\hat{T}_{{\rm u},\alpha}^*
\big)_{i,j}
\\
+z_0\big( Q_0Q_0^* \big)_{i,j}
+\sum_{\alpha=1}^t
\frac{a_{\alpha}\sqrt{c_{\alpha}}}{f_{\alpha}} 
\left(
\hat{T}_{{\rm u},\alpha}
\Big(
\frac{f_{\alpha}}{2c_{\alpha}}\widetilde{T}_{{\rm d},\alpha}
-\frac{f_{\alpha}^2}{8c_{\alpha}^2}\widetilde{T}_{{\rm d},\alpha}^2
+O\big(
\|
\widetilde{T}_{{\rm d},\alpha}
\|^3
\big)\Big)
\right)_{i,j}
+O\big( \| Q_{t+1} \|^2 \big).
\end{multline} 
Since combination of  \eqref{hat g2 critical non0}, \eqref{R error term} and \eqref{S error term1}  gives rise to 
\begin{equation}\label{g2expansion2}
G_2=-\frac{1}{\sqrt{P_1}}
\big( R_2+R_3+S_2^{({\rm off},{\rm u})}
+S_3^{({\rm off},{\rm u})} \big)
+O\Big( \| \hat{G}_2 \|+
\| G_1 \|\big( N^{-\frac{3}{4}}+\| \Gamma_1 \| \big)+
\Omega_{{\rm error},1}^2 \Big),
\end{equation}
where $S_{2}^{({\rm off},{\rm u})}$ denotes the strictly upper triangular part of $S_2$, 
together with  \eqref{R2 critical non0} and \eqref{R3 critical non0},  we proceed as in a similar  way   to get   \eqref{R expansion2 critical non0} and   arrive at  
\begin{equation}\label{Triangular3 critical non0}
\begin{aligned}
&-z_0(R_2+R_3)+
\sum_{\alpha=1}^t
\frac{a_{\alpha}\sqrt{c_{\alpha}}}{f_{\alpha}}
\hat{T}_{{\rm u},\alpha}
\Big(
\frac{f_{\alpha}}{2c_{\alpha}}\widetilde{T}_{{\rm d},\alpha}
-\frac{f_{\alpha}^2}{8c_{\alpha}^2}\widetilde{T}_{{\rm d},\alpha}^2
+O\big(
\|
\widetilde{T}_{{\rm d},\alpha}
\|^3
\big)\Big)           \\
&=\frac{N^{-\frac{1}{4}}}{2}G_1 \sum_{\alpha=1}^t
\frac{c_{\alpha}}{f_{\alpha}^2} H_{a_{\alpha}-z_0}
-\frac{1}{2}G_1\sum_{\alpha=1}^t
\frac{c_{\alpha}}{P_1f_{\alpha}^2}S^{({\rm d})}
-\frac{N^{-\frac{1}{2}}}{8}G_1
\sum_{\alpha=1}^t
\frac{c_{\alpha}}{f_{\alpha}^3}H_{a_{\alpha}-z_0}^2
\\
&+N^{-\frac{1}{4}}\sum_{\beta=2}^t G_{\beta}
\sum_{\alpha=1}^t
\frac{\sqrt{c_{\alpha}}(a_{\alpha}-z_0)}{2f_{\alpha}^2}
H_{a_{\alpha}-z_0}v_{\alpha,\beta}
+O\Big(  
\| G_1 \|\big( N^{-\frac{3}{4}}+\| \Gamma_1 \| \big)
+\Omega_{{\rm error},1}^2
\Big)\\
&=-\frac{1}{2}G_{1}S^{({\rm d})}
-\frac{N^{-\frac{1}{2}}}{8}G_{1}
\big( P_2\big( \hat{Z}^* \big)^2
+2P_1\hat{Z}\hat{Z}^*+\overline{P}_2\hat{Z}^2 \big)
+\frac{N^{-\frac{1}{4}}}{2}\sqrt{P_1}G_2\hat{Z}           \\
&+\frac{N^{-\frac{1}{4}}}{2}
\sum_{\beta=2}^t \eta_{\beta} G_{\beta}\hat{Z}^*
+O\Big(  
\| G_1 \|\big( N^{-\frac{3}{4}}+\| \Gamma_1 \| \big)
+\Omega_{{\rm error},1}^2
\Big),
\end{aligned}
\end{equation}
where  in the second identity  we have used  $P_{0}(z_0)=0$,   \eqref{H2power} and  the  identities 
\begin{equation}\label{identitiesHeta}
\sum_{\alpha=1}^t 
\frac{\sqrt{c_{\alpha}}(a_{\alpha}-z_0)}{2f_{\alpha}^2}
H_{a_{\alpha}-z_0}v_{\alpha,\beta}=\frac{1}{2}
\big( \eta_{\beta}\hat{Z}^*
+\sqrt{P_1}\delta_{2,\beta}\hat{Z}
 \big)
\end{equation}
with \begin{equation}\label{eta beta}
\eta_{\beta}:=\sum_{\alpha=1}^t
\frac{\sqrt{c_{\alpha}}(a_{\alpha}-z_0)^2}{f_{\alpha}^2}
v_{\alpha,\beta}.
\end{equation} 
\eqref{identitiesHeta}

In fact,   \eqref{identitiesHeta}  immediately follows from the definition and unitary invariance of $V$, which further implies   
\begin{equation}    \label{etarelation}
\eta_{1}=0,\quad
\sqrt{P_1} \eta_{2} =P_2%P_1 |\eta_{2}|^2=|P_2|^2
, \quad \sum_{\alpha=1}^t |\eta_{\alpha}|^2=P_1.
\end{equation}

As to  the  remaining terms in \eqref{g2expansion2},  
it may be   a coincidence that  
\begin{multline}\label{Triangular6 critical non0}
-z_0\big(  S_2^{({\rm off},{\rm u})}+S_3^{({\rm off},{\rm u})} \big)
+\sum_{\alpha=1}^t 
\frac{a_{\alpha}}{f_{\alpha}}
\big( \hat{T}_{{\rm u},\alpha}\hat{T}_{{\rm u},\alpha}^* \big)
^{({\rm off},{\rm u})} 
+z_0\big( Q_0Q_0^* \big)^{({\rm off},{\rm u})} 
          \\
=\sqrt{P_1}\big( G_1G_2^* \big)^{({\rm off},{\rm u})}
+\sum_{\beta=2}^t \eta_{\beta}
\big( G_{\beta}G_1^* \big)^{({\rm off},{\rm u})}
+O\big( \Omega_{{\rm error},1}^2 \big)
\end{multline}
after some overlaps are cancelled  out.

Finally, combining \eqref{Triangular1 critical non0}, \eqref{g2expansion2},  \eqref{Triangular3 critical non0} and  \eqref{Triangular6 critical non0},  for $1\leq i<j\leq n$ we derive 
\begin{equation}\label{Triangular7 critical non0}
\begin{aligned}
&\Big( \sum_{\alpha=1}^t 
a_{\alpha}T_{\alpha}T_{\alpha}^*+Q_{t+1}A_{t+1}Q_{t+1}^*
+z_0 Q_0Q_0^*
\Big)_{i,j}=\big(G_1+ \sqrt{P_1}G_1G_2^* \big)_{i,j}
+\sum_{\beta=2}^t \eta_{\beta} \big( G_{\beta} G_1^* \big)_{i,j}
\\
&-\frac{1}{2}\big(
G_1S^{({\rm d})} \big)_{i,j}-\frac{1}{8}
\big( G_1
\big( P_2\big( \hat{Z}^* \big)^2
+2P_1\hat{Z}\hat{Z}^*+\overline{P}_2\hat{Z}^2 \big)
 \big)_{i,j}+\frac{N^{-\frac{1}{4}}\sqrt{P_1}}{2}
 \big( G_2\hat{Z} \big)_{i,j}
 \\
 &+\frac{N^{-\frac{1}{4}}}{2}
 \sum_{\beta=2}^t \eta_{\beta}
 \big( G_{\beta}\hat{Z}^* \big)_{i,j}
+O\Big(  
\| G_1 \|\big( N^{-\frac{3}{4}}+\| \Gamma_1 \| \big)
+\Omega_{{\rm error},1}^2
\Big).
\end{aligned}
\end{equation}
So  in the trace form  this yields 
\begin{equation}\label{Triangular8 critical non0}
\begin{aligned}
&\sum_{1\leq i<j\leq n}\left|
\Big( \sum_{\alpha=1}^t 
a_{\alpha}T_{\alpha}T_{\alpha}^*+Q_{t+1}A_{t+1}Q_{t+1}^*
+z_0 Q_0Q_0^*
\Big)_{i,j}
\right|^2=         \\
&{\rm Tr}(G_1G_1^*)+
\sqrt{P_1}{\rm Tr}\big(G_1(G_2+G_2^*)G_1^*\big)
+\sum_{\beta=2}^t \overline{\eta}_{\beta}
{\rm Tr}\Big(\big( G_1^2+\frac{1}{2}N^{-\frac{1}{4}}
G_1\hat{Z} \big)G_{\beta}^*\Big)           \\
&+\sum_{\beta=2}^t \eta_{\beta}
{\rm Tr}\Big(G_{\beta}\big( G_1^2+\frac{ 1}{2}N^{-\frac{1}{4}}
G_1\hat{Z} \big)^*\Big)
+\frac{ 1}{2}N^{-\frac{1}{4}}\sqrt{P_1}\Big(
{\rm Tr}\big(G_1\hat{Z}^* G_2^*\big)
+{\rm Tr}\big(G_2\hat{Z} G_1^*\big)
\Big)
     \\
&-{\rm Tr}
\big( G_1S^{({\rm d})}G_1^* \big) 
-\frac{1}{4}N^{-\frac{1}{2}}{\rm Tr}\Big(
G_1\big( P_2\big( \hat{Z}^* \big)^2
+2P_1\hat{Z}\hat{Z}^*+\overline{P}_2\hat{Z}^2 \big)
G_1^*
\Big)      \\
&+O\Big(
\| G_1 \|^2\| \Gamma_1 \|+
 \big( \| G_1 \|+N^{-\frac{1}{4}} \big)
 \Omega_{{\rm error},1}^2+\big( \| G_1 \|+N^{-\frac{1}{4}} \big)^4
 \Omega_{{\rm error},1}
\Big).
\end{aligned}
\end{equation}

To simplify \eqref{Triangular8 critical non0} further,  using  the notations in \eqref{R expansion2 critical non0}-\eqref{S3 critical non0},  put 
\begin{equation}\label{K critical non0}
K=S_2+R_2+R_2^*.
\end{equation}
Note that it follows  from \eqref{hat g2 critical non0}, \eqref{R error term} and \eqref{S error term1} that \begin{equation}\label{g2expansion}
G_2=-\frac{1}{\sqrt{P_1}}
\big( R_2+S_2^{({\rm off},{\rm u})} \big)
+O\big( \Omega_{{\rm error},2} \big),
\end{equation}
where
\begin{equation}\label{Omega error2}
\begin{aligned}
\Omega_{{\rm error},2}=\| \hat{G}_2 \|+&
\Big( \sum_{\beta=3}^t\| G_{\beta} \|+
\big( \| G_1 \|+N^{-\frac{1}{4}} \big)^2+
\|\widetilde{\mathcal{T}}_{{\rm d}}\|_{2}+\| \Gamma_1 \|
+\| Q_{t+1} \|+\| Q_0 \|^2
 \Big)
 \\
&\times \Big( \sum_{\beta=3}^t\| G_{\beta} \|+
\| G_1 \|+N^{-\frac{1}{4}}+
\|\widetilde{\mathcal{T}}_{{\rm d}}\|_{2}+\| \Gamma_1 \|
+\| Q_{t+1} \|+\| Q_0 \|^2
 \Big),
\end{aligned}
\end{equation}
we see  from \eqref{S error term1} that 
\begin{multline}\label{Triangular9 critical non0}
\sqrt{P_1}{\rm Tr}\big(G_1(G_2+G_2^*)G_1^*\big)
-{\rm Tr}
\big( G_1S^{({\rm d})}G_1^* \big) \\
=-{\rm Tr}(G_1^*G_1K)
+O\Big(
\| G_1 \|^2\big( 
\Omega_{{\rm error},2}+\Omega_{{\rm error},1}^2+
\| G_1 \|\Omega_{{\rm error},1}
 \big)
 \Big),
\end{multline}
\begin{multline}\label{Triangular10 critical non0}
\frac{1}{2}N^{-\frac{1}{4}}\sqrt{P_1}\Big(
{\rm Tr}\big(G_1\hat{Z}^* G_2^*\big)
+{\rm Tr}\big(G_2\hat{Z} G_1^*\big)
\Big)         \\
=-\frac{1}{2}N^{-\frac{1}{4}}
{\rm Tr}\big((G_1\hat{Z}^*+\hat{Z} G_1^*) K\big)+O\big( N^{-\frac{1}{4}}
\| G_1 \|\Omega_{{\rm error},2}
 \big),
\end{multline}
\begin{multline}\label{Triangular11 critical non0}
\overline{\eta}_2{\rm Tr}
\Big(\big( G_1^2+\frac{1}{2}N^{-\frac{1}{4}}
G_1\hat{Z} \big)G_2^*\Big)+{\eta}_2{\rm Tr}
\Big(G_2\big( G_1^2+\frac{1}{2}N^{-\frac{1}{4}}
G_1\hat{Z} \big)^*\Big)
=
\\
-\frac{\overline{P}_2}{P_1}{\rm Tr}
\Big(\big( G_1^2+\frac{N^{-\frac{1}{4}}}{2}
G_1\hat{Z} \big)K\Big)-\frac{P_2}{P_1}{\rm Tr}
\Big(K\big( G_1^2+\frac{N^{-\frac{1}{4}}}{2}
G_1\hat{Z} \big)^*\Big)
+O\Big( \big(N^{-\frac{1}{4}}+
\| G_1 \|\big)^2\Omega_{{\rm error},2}
 \Big),
\end{multline}
 and  by adding partial sum on the RHS of  \eqref{unitarity critical non0}, in combination with \eqref{Tr hatZ Sd6},
\begin{multline}\label{Triangular13 critical non0}
\sum_{\beta=3}^t {\rm Tr}\Big(-G_{\beta}G_{\beta}^*
+\overline{\eta}_{\beta}
\big( 
(\hat{Z}G_1+\frac{G_1\hat{Z}}{2}
)N^{-\frac{1}{4}}+
G_1^2\big)G_{\beta}^*
+\eta_{\beta}
G_{\beta}\big( 
(\hat{Z}G_1+\frac{G_1\hat{Z}}{2}
)N^{-\frac{1}{4}}+
G_1^2\big)^*\Big)
\\
=-\sum_{\beta=3}^t {\rm Tr}\Big(
G_{\beta}-\overline{\eta}_{\beta}
\Big(
\big( \hat{Z}G_1+\frac{G_1\hat{Z}}{2} \big)N^{-\frac{1}{4}}
+G_1^2
\Big)
\Big)\Big(
G_{\beta}-\overline{\eta}_{\beta}
\Big(
\big( \hat{Z}G_1+\frac{G_1\hat{Z}}{2} \big)N^{-\frac{1}{4}}
+G_1^2
\Big)
\Big)^*     
 \\
+\big( P_1-\frac{|P_2|^2}{P_1} \big)
{\rm Tr}\Big(
\big( \hat{Z}G_1+\frac{G_1\hat{Z}}{2} \big)N^{-\frac{1}{4}}
+G_1^2
\Big)\Big(
\big( \hat{Z}G_1+\frac{G_1\hat{Z}}{2} \big)N^{-\frac{1}{4}}
+G_1^2
\Big)^*.
\end{multline}

%%%%%%%%%%%%%%%%%%%%%%%%%%%%%%%%%%%%%%%%%%%%%%%%%%%%%%%%%%%%%%%%%%%%%%%%%%%%%%%%%%%%%%%%%%%%%%%%%%%%%%%%%%%%%%%%%%%%%%%%%%%%%%%%%%%%%%%%%%%%%%%%%%%%%%%%%%%%%%%%%%%%%%%%%%%%%%%%%%%%%%%%%%%%%%%%%%%%%%%%%%%%%%%%%%%%%%%%%
%{\bf Step 5: Calculating the other terms in $h(T,Q)$}. \eqref{hQrewrite} With \eqref{S error term1} and \eqref{K critical non0}, continuing from \eqref{SP2 hatZ},

With  \eqref{g hatT connection1} in mind, 
 starting  from \eqref{Tr hatZ Sd1} we have 
\begin{multline}\label{Tr hatZ Sd2}
N^{-\frac{1}{4}}\sum_{\alpha=1}^t
\frac{a_{\alpha}-z_0}{f_{\alpha}}
{\rm Tr}\big( \hat{Z}^*
\hat{T}_{{\rm u},\alpha}\hat{T}_{{\rm u},\alpha}^* \big)=
N^{-\frac{1}{4}}\sum_{\alpha=1}^t
\frac{a_{\alpha}-z_0}{f_{\alpha}}|v_{\alpha,1}|^2
{\rm Tr}\big( \hat{Z}^*G_1G_1^* \big)          \\
+N^{-\frac{1}{4}}\sum_{\alpha=1}^t
\frac{a_{\alpha}-z_0}{f_{\alpha}}v_{\alpha,1}
\sum_{\gamma=2}^t \overline{v}_{\alpha,\gamma}
{\rm Tr}\big( \hat{Z}^*G_1 G_{\gamma}^* \big)  \\
+N^{-\frac{1}{4}}\sum_{\alpha=1}^t
\frac{a_{\alpha}-z_0}{f_{\alpha}}\overline{v}_{\alpha,1}
\sum_{\beta=2}^t v_{\alpha,\beta}
{\rm Tr}\big( \hat{Z}^*G_{\beta}G_1^* \big) 
+O\big( N^{-\frac{1}{4}} \| G\|_{2}^2 \big)
\\
=N^{-\frac{1}{4}}
\sqrt{P_1}{\rm Tr}\big( \hat{Z}^*G_1 G_2^* \big) 
+N^{-\frac{1}{4}}\sum_{\beta=2}^t
\eta_{\beta}{\rm Tr}\big( \hat{Z}^*G_{\beta} G_1^* \big) 
+O\big( N^{-\frac{1}{4}}
\Omega_{{\rm error},1}^2 \big),
\end{multline}
where we have used the facts 
\begin{equation}\label{Tr hatZ Sd mid}
\begin{aligned}
%&\sum_{\alpha=1}^t
%\frac{a_{\alpha}-z_0}{f_{\alpha}}|v_{\alpha,1}|^2=P_0=0,
%\\
\sum_{\alpha=1}^t 
\frac{a_{\alpha}-z_0}{f_{\alpha}}
v_{\alpha,1}\overline{v}_{\alpha,\gamma}
%&=
%\sum_{\alpha=1}^t \frac{\sqrt{c_{\alpha}}}{f_{\alpha}}
%\overline{v}_{\alpha,\gamma}
=\sqrt{P_1}\sum_{\alpha=1}^t
v_{\alpha,2}\overline{v}_{\alpha,\gamma}
=\sqrt{P_1}\delta_{2,\gamma},
\end{aligned}
\end{equation}
and
\begin{equation}\label{sum gbeta error}
  \| G \|_2=O\big( 
\Omega_{{\rm error},1}
 \big).
\end{equation}

Use of  \eqref{K critical non0} and \eqref{g2expansion}  immediately yields 
\begin{equation}\label{Tr hatZ Sd4}
N^{-\frac{1}{4}}\sqrt{P_1}
{\rm Tr}\big( \hat{Z}^*G_1 G_2^* \big)=
-N^{-\frac{1}{4}}
{\rm Tr}\big( \hat{Z}^*G_1 K \big)+
O\big( N^{-\frac{1}{4}}\| G_1 \|\Omega_{{\rm error},2} \big),
\end{equation}
and 
\begin{multline}\label{Tr hatZ Sd5}
N^{-\frac{1}{4}}\sum_{\beta=2}^t
\eta_{\beta}{\rm Tr}\big( \hat{Z}^*G_{\beta}G_1^* \big) 
=-N^{-\frac{1}{4}}\frac{P_2}{P_1}
{\rm Tr}\big( \hat{Z}^*K G_1^* \big)    +N^{-\frac{1}{4}}\sum_{\beta=3}^t \eta_{\beta}
{\rm Tr}\big( \hat{Z}^*G_{\beta} G_1^* \big) 
\\
+O\big( N^{-\frac{1}{4}}\| G_1 \|\Omega_{{\rm error},2} \big),
\end{multline}
 from which 
 we have
\begin{multline}\label{Tr hatZ Sd6}
N^{-\frac{1}{4}}\sum_{\alpha=1}^t
\frac{a_{\alpha}-z_0}{f_{\alpha}}
{\rm Tr}\big( \hat{Z}^*
\hat{T}_{{\rm u},\alpha}\hat{T}_{{\rm u},\alpha}^* \big)=
-N^{-\frac{1}{4}}
{\rm Tr}\big( \hat{Z}^* G_1 K \big)-N^{-\frac{1}{4}}\frac{P_2}{P_1}
{\rm Tr}\big( \hat{Z}^*K G_1^* \big)         
     \\
+N^{-\frac{1}{4}}\sum_{\beta=3}^t \eta_{\beta}
{\rm Tr}\big( \hat{Z}^*G_{\beta} G_1^* \big) 
+O\Big( N^{-\frac{1}{4}}\big(
\| G_1 \|\Omega_{{\rm error},2}+\Omega_{{\rm error},1}^2
 \big)\Big).
\end{multline}
Hence  combination of  \eqref{unitarity critical non0},   \eqref{g2expansion} and \eqref{K critical non0} gives rise to 
\begin{equation}\label{g2expansion1}
\begin{aligned}
-{\rm Tr}
\big( G_2 G_2^* \big)=-\frac{1}{2P_1}{\rm Tr}\Big( 
\big(  S_2^{({\rm off},{\rm u})}+R_2  \big)
K\Big)
-\frac{1}{2P_1}{\rm Tr}\Big( 
\big(  {S_2^{({\rm off},{\rm u})}}^*+R_2^*  \big)
K\Big)      \\
+O\Big( 
\big( \| G_1 \|+N^{-\frac{1}{4}} \big)^2\Omega_{{\rm error},2}
+\Omega_{{\rm error},2}^2
\Big).
\end{aligned}
\end{equation}
We also see from 
  \eqref{bi A bit3 critical non0} and \eqref{S error term1} that
\begin{equation}\label{Triangular14 critical non0}
-\frac{1}{2P_1}{\rm Tr}\big( S^{({\rm d})} \big)^2
=-\frac{1}{2P_1}{\rm Tr}\big( S_2^{({\rm d})}K \big)
 S_2^{({\rm d})} \big)^2
+O\Big(
\big( \Omega_{{\rm error},1}^2
+\| G_1 \|\Omega_{{\rm error},1} \big)^2
+\| G_1 \|^3\Omega_{{\rm error},1}
\Big).
\end{equation} 
So putting  the  $G_2$ term in \eqref{unitarity critical non0} and   the last term in  \eqref{bi A bit3 critical non0},     we   combine \eqref{g2expansion1} and  obtain 
\begin{multline}\label{Triangular15 critical non0}
-{\rm Tr}(G_2G_2^*)-\frac{1}{2P_1}
{\rm Tr}\big( S^{({\rm d})} \big)^2
=
-\frac{1}{2}{\rm Tr}(K^2)+O\Big(
\Omega_{{\rm error},2}\big( 
N^{-\frac{1}{4}}+\| G_1 \|
 \big)^2+\Omega_{{\rm error},2}^2
\Big)       
 \\
+O\Big(
\big( 
\Omega_{{\rm error},1}^2+
\| G_1 \|\Omega_{{\rm error},1}
 \big)^2+\| G_1 \|^3\Omega_{{\rm error},1}
\Big).
\end{multline}
Also,  we can rewrite   one term of  the RHS of \eqref{Talphaexpansion12 critical non0} as 
\begin{multline}\label{SP2 hatZ extension}
\frac{1}{P_1} N^{-\frac{1}{2}} {\rm Tr}
\Big(
S\big( P_2(\hat{Z}^*)^2
+2P_1\hat{Z}\hat{Z}^*+\overline{P}_2\hat{Z}^2 \big)
\Big)=
\\
\frac{1}{P_1}{\rm Tr}
\Big(
K\big( P_2(\hat{Z}^*)^2
+2P_1\hat{Z}\hat{Z}^*+\overline{P}_2\hat{Z}^2 \big)
\Big)+N^{-\frac{1}{2}}O\big(
\Omega_{{\rm error},1}^2+\| G_1 \|\Omega_{{\rm error},1}
\big).
\end{multline}

%%%%%%%%%%%%%%%%%%%%%%%%%%%%%%%%%%%%%%%%%%%%%%%%%%%%%%%%%%%%%%%%%%%%%%%%%%%%%%%%%%%%%%%%%%%%%%%%%%%%%%%%%%%%%%%%%%%%%%%%%%%%%%%%%%%%%%%%%%%%%%%%%%%%%%%%%%%%%%%%%%%%%%%%%%%%%%%%%%%%%%%%%%%%%%%%%%%%%%%%%%%%%%%%%%%%%%%%%%%%%%
{\bf Step 6: Effective terms in   $\tau^{-1}h(\boldsymbol{T},\boldsymbol{Q})$}.

Recalling $h(\boldsymbol{T},\boldsymbol{Q})$ in \eqref{hQrewrite} and  $1/\tau$ in \eqref{tau inverse expansion},  we  focus on all effective terms in $\hat{\tau}N^{-\frac{1}{2}}h(\boldsymbol{T},\boldsymbol{Q})$ and $\hat{\tau}^2 N^{-1}h(\boldsymbol{T},\boldsymbol{Q})$.

First, combining  \eqref{Tdalpha1 critical non0}, \eqref{g hatT connection1},
\eqref{OmegaError1} and \eqref{sum gbeta error}, we obtain
\begin{multline}\label{tau hQ expansion1}
-\sum_{\alpha=1}^t f_{\alpha}{\rm Tr}(T_{\alpha}T_{\alpha}^*)
=-\sum_{\alpha=1}^t 
\Big( 
nc_{\alpha}
+f_{\alpha}{\rm Tr}\big( \widetilde{T}_{{\rm d},\alpha} \big)
+{\rm Tr}\big( \hat{T}_{{\rm u},\alpha}\hat{T}_{{\rm u},\alpha}^* \big)
 \Big)
\\
=-n-\sum_{\alpha=1}^t 
f_{\alpha}{\rm Tr}\big( \widetilde{T}_{{\rm d},\alpha} \big)
-{\rm Tr}(G_1G_1^*)+O\big( \Omega_{{\rm error},1}^2 \big).
\end{multline}
Note that 
 by   \eqref{S critical non0} we have 
\begin{multline}\label{tau hQ expansion3}
{\rm Tr}(S)=\sum_{\alpha=1}^t \frac{1}{f_{\alpha}}
{\rm Tr}\big( \hat{T}_{{\rm u},\alpha}\hat{T}_{{\rm u},\alpha}^* \big)
+{\rm Tr}\big( Q_0Q_0^* \big)
+O\big( \| Q_{t+1} \|^2 \big)              \\
=P_1{\rm Tr}\big( G_1G_1^* \big)+{\rm Tr}\big( Q_0Q_0^* \big)+O\big( 
\| G_1 \|\Omega_{{\rm error},1}+\Omega_{{\rm error},1}^2
 \big),
\end{multline}
 and  further get  from   \eqref{matrix transformations1 critical non0} and \eqref{newT-3} that 
\begin{multline}\label{tau hQ expansion2}
\sum_{\alpha=1}^t 
f_{\alpha}{\rm Tr}\big( \widetilde{T}_{{\rm d},\alpha} \big)
=f_1{\rm Tr}\big( 
-\sum_{\alpha=2}^t \widetilde{T}_{\rm d,\alpha}-S^{({\rm d})}
 \big)+\sum_{\alpha=2}^t f_{\alpha}{\rm Tr}(\widetilde{T}_{\rm d,\alpha})
 +O(\| \Gamma_1 \|)
 \\
=\sum_{\alpha=2}^t (f_{\alpha}-f_1)
{\rm Tr}\big(\widetilde{\mathcal{T}}_{\rm d,\alpha}\big)-\frac{1}{P_1}
{\rm Tr}(S)+N^{-\frac{1}{4}}\sum_{\alpha=1}^t
\frac{c_{\alpha}}{f_{\alpha}}
{\rm Tr}\big(H_{a_{\alpha}-z_0}\big)
+O(\| \Gamma_1 \|).
\\
=\sum_{\alpha=2}^t (f_{\alpha}-f_1)
{\rm Tr}\big(\widetilde{\mathcal{T}}_{\rm d,\alpha}\big)
-{\rm Tr}\big(G_1G_1^*\big)-\frac{1}{P_1}{\rm Tr}\big( Q_0Q_0^* \big)
+N^{-\frac{1}{4}}\sum_{\alpha=1}^t
\frac{c_{\alpha}}{f_{\alpha}}
{\rm Tr}\big(H_{a_{\alpha}-z_0}\big)
\\
+O\big( \| \Gamma_1 \|+
\| G_1 \|\Omega_{{\rm error},1}+\Omega_{{\rm error},1}^2
 \big).
\end{multline}

It's easy to see from \eqref{S critical non0} and \eqref{matrix transformations1 critical non0}  that 
\begin{multline}\label{tau hQ expansion4}
|z_0|^2\Big(
\sum_{\alpha=1}^t {\rm Tr}\big(T_{\alpha}T_{\alpha}^*\big)
+{\rm Tr}\big(Q_{t+1}Q_{t+1}^*\big)+{\rm Tr}\big( Q_0Q_0^* \big)
\Big)
=|z_0|^2\big(n+
\sum_{\alpha=1}^t {\rm Tr}
\big(\widetilde{T}_{{\rm d},\alpha}\big)
+{\rm Tr}(S)
\big)        
\\
=n|z_0|^2+O(\| \Gamma_1 \|),
\end{multline}
\begin{multline}\label{tau hQ expansion5}
N^{-\frac{1}{4}}\sum_{\alpha=1}^t\Big(
\overline{a}_{\alpha}{\rm Tr}\big(
\hat{Z}
T_{\alpha}T_{\alpha}^*\big)
+a_{\alpha}{\rm Tr}\big(
\hat{Z}^*
T_{\alpha}T_{\alpha}^*\big)
\Big)=N^{-\frac{1}{4}}\sum_{\alpha=1}^t
\frac{c_{\alpha}}{f_{\alpha}} {\rm Tr}(H_{a_\alpha})
\\+
N^{-\frac{1}{4}}\sum_{\alpha=1}^t
{\rm Tr}
\big( H_{a_\alpha}\widetilde{T}_{{\rm d},\alpha} \big)
+O\Big( 
N^{-\frac{1}{4}}
\big( \| G_1 \|+\Omega_{{\rm error},1} \big)^2
 \Big),
\end{multline}
and
\begin{multline*}
N^{-\frac{1}{4}}\sum_{\alpha=1}^t
{\rm Tr}
\big( H_{\alpha}\widetilde{T}_{{\rm d},\alpha} \big)=
N^{-\frac{1}{4}}{\rm Tr}
\Big( H_1
\big( -\sum_{\alpha=2}^t \widetilde{T}_{{\rm d},\alpha}-S^{({\rm d})} \big) \Big)
+N^{-\frac{1}{4}}\sum_{\alpha=2}^t {\rm Tr}
(H_{a_\alpha}\widetilde{T}_{{\rm d},\alpha})+O\big( N^{-\frac{1}{4}}\| \Gamma_1 \| \big)
\\
=N^{-\frac{1}{2}}\sum_{\alpha=1}^t
\frac{c_{\alpha}}{f_{\alpha}^2}  
{\rm Tr}\big( H_{a_\alpha}H_{a_\alpha-z_0} \big)
+O\Big( 
N^{-\frac{1}{4}}
\big( \| \widetilde{\mathcal{T}}_{{\rm d}} \|_2+
\| \Gamma_1 \|+
 \| G_1 \|^2+\Omega_{{\rm error},1}^2
 \big)
 \Big).
\end{multline*}
By $P_{0}(z_0)=0$ we have  
$$
\sum_{\alpha=1}^t
\frac{c_{\alpha}}{f_{\alpha}^2}  
{\rm Tr}\big( H_{a_\alpha} H_{a_\alpha-z_0} \big)
=
\sum_{\alpha=1}^t
\frac{c_{\alpha}}{f_{\alpha}^2}  
{\rm Tr}\big( H_{a_\alpha-z_0} \big)^2
$$
from which 
\begin{equation}\label{tau hQ expansion6}
\begin{aligned}
&N^{-\frac{1}{4}}\sum_{\alpha=1}^t\Big(
\overline{a}_{\alpha}{\rm Tr}\big(
\hat{Z}
T_{\alpha}T_{\alpha}^*\big)
+a_{\alpha}{\rm Tr}\big(
\hat{Z}^*
T_{\alpha}T_{\alpha}^*\big)
\Big)=N^{-\frac{1}{4}}\sum_{\alpha=1}^t
\frac{c_{\alpha}}{f_{\alpha}} {\rm Tr}(H_{a_\alpha})
\\&+
N^{-\frac{1}{2}}\sum_{\alpha=1}^t
\frac{c_{\alpha}}{f_{\alpha}^2}  
{\rm Tr}\big( H_{a_\alpha-z_0} \big)^2
+O\Big( 
N^{-\frac{1}{4}}
\big( 
 \| \widetilde{\mathcal{T}}_{{\rm d}} \|_2+\| \Gamma_1 \|+\| G_1 \|^2
+\Omega_{{\rm error},1}^2 \big)
 \Big).
\end{aligned}
\end{equation}
By \eqref{Triangular8 critical non0}, we have
\begin{multline}\label{tau hQ expansion7}
\sum_{1\leq i<j\leq n}\Big|
\Big( \sum_{\alpha=1}^t 
a_{\alpha}T_{\alpha}T_{\alpha}^*+Q_{t+1}A_{t+1}Q_{t+1}^*
\Big)_{i,j}
\Big|^2= {\rm Tr}(G_1G_1^*)      
  \\
+O\Big(
\big( \| G_1 \|+N^{-\frac{1}{4}} \big)^2
\big( \Omega_{{\rm error},1}
+\| G_1 \|^2+\| \Gamma_1 \| \big)
+\big( \| G_1 \|+N^{-\frac{1}{4}} \big)\Omega_{{\rm error},1}^2
\Big).
\end{multline}

Therefore,  combination of  \eqref{tau hQ expansion1}, \eqref{tau hQ expansion2}, %\eqref{tau hQ expansion3}, 
\eqref{tau hQ expansion4}, \eqref{tau hQ expansion6} and \eqref{tau hQ expansion7} gives rise to
\begin{multline}\label{tau hQ expansion8}
-\hat{\tau}N^{-\frac{1}{2}}
h(\boldsymbol{T},\boldsymbol{Q})=n\hat{\tau}N^{-\frac{1}{2}}
-n\hat{\tau}|z_0|^2 N^{-\frac{1}{2}}
-\hat{\tau}N^{-\frac{3}{4}}
\big( z_0{\rm Tr}\big( \hat{Z}^* \big)
 +\overline{z}_0{\rm Tr}\big( \hat{Z} \big) \big)
 \\
 -\hat{\tau}N^{-1}
 \sum_{\alpha=1}^t 
 \frac{c_{\alpha}}{f_{\alpha}^2}
 {\rm Tr}\big( H_{a_\alpha-z_0} \big)^2  
 -\hat{\tau}N^{-\frac{1}{2}}{\rm Tr}( G_1G_1^* )
 +\hat{\tau}N^{-\frac{1}{2}}\sum_{\alpha=2}^t
 (f_{\alpha}-f_1){\rm Tr}\big( \widetilde{\mathcal{T}}_{{\rm d,\alpha}} \big)
 \\
 -\frac{\hat{\tau}}{P_1}N^{-\frac{1}{2}}{\rm Tr}\big( Q_0Q_0^* \big)
 +N^{-\frac{1}{2}}O\Big( 
 \Omega_{{\rm error},1}^2+
 \big( \|G_1\|+N^{-\frac{1}{4}} \big)\Omega_{{\rm error},1}
 +\|\Gamma_1\|
  \Big)
\end{multline}
and
\begin{equation}\label{tau hQ expansion9}
\hat{\tau}^2 N^{-1}
h(\boldsymbol{T},\boldsymbol{Q})=-n\hat{\tau}^2 N^{-1}+n|z_0|^2\hat{\tau}^2 N^{-1}
+N^{-1}O\big( 
\| G_1 \|^2+ \| \widetilde{\mathcal{T}}_{{\rm d}} \|_2+\| \Gamma_1 \|+
\Omega_{{\rm error},1}^2+N^{-\frac{1}{4}}
 \big).
\end{equation}

Finally, 
combining \eqref{Talphaexpansion1 critical non0}, \eqref{Talphaexpansion6 critical non0}, %\eqref{Talphaexpansion5 critical non0}, 
%\eqref{Talphaexpansion8 critical non0}, \eqref{bi A bit3 critical non0}, 
%\eqref{Tr hatZ Sd1}, \eqref{tilde xi bi}, \eqref{Talphaexpansion12 critical non0}, 
\eqref{2powerterm}, \eqref{Talphaexpansion12 critical non0},
%\eqref{SP2 hatZ},
 \eqref{Talphaexpansion13 critical non0}, \eqref{Triangular8 critical non0}-\eqref{Triangular13 critical non0}, %\eqref{Tr hatZ Sd6}-\eqref{unitarity critical non0}, 
 \eqref{Tr hatZ Sd6}, \eqref{Triangular15 critical non0}, \eqref{SP2 hatZ extension}, \eqref{tau hQ expansion8} and  \eqref{tau hQ expansion9}, we have
\begin{small}
\begin{equation}\label{tau hQ expansion10}
\begin{aligned}
&\frac{1}{\tau} h(\boldsymbol{T},\boldsymbol{Q})+\sum_{\alpha=1}^t 
c_{\alpha}{\rm Tr}\log (T_{{\rm d},\alpha})=
n\sum_{\alpha=1}^t c_{\alpha}\big(
\log \frac{c_{\alpha}}{f_{\alpha}}-1
\big)+n|z_0|^2    
+N^{-\frac{1}{4}}\sum_{\alpha=1}^t
\frac{c_{\alpha}}{f_{\alpha}}{\rm Tr}(H_{a_\alpha})
\\
&+\frac{N^{-\frac{1}{2}}}{2}\sum_{\alpha=1}^t 
\frac{c_{\alpha}}{f_{\alpha}^2} 
\Big( (a_{\alpha}-z_0)^2{\rm Tr}\big( \hat{Z}^* \big)^2 
+\overline{a_{\alpha}-z_0}^2{\rm Tr}\big( \hat{Z}^2 \big)
\Big)
+N^{-\frac{1}{2}} {\rm Tr}\big( \hat{Z}\hat{Z}^* \big)
\\
&+\frac{N^{-\frac{3}{4}}}{3}\sum_{\alpha=1}^t
\frac{c_{\alpha}}{f_{\alpha}^3}\Big(
(a_{\alpha}-z_0)^3{\rm Tr}\big( \hat{Z}^* \big)^3+
\overline{a_{\alpha}-z_0}^3{\rm Tr}\big( \hat{Z}^3 \big)
\Big)
-\frac{N^{-1}}{4}\sum_{\alpha=1}^t
\frac{c_{\alpha}}{f_{\alpha}^4}{\rm Tr}\big(
H_{a_\alpha-z_0}^4
 \big)                \\
& +(1-|z_0|^2)n\hat{\tau}N^{-\frac{1}{2}}
-\hat{\tau}N^{-\frac{3}{4}}
\big( z_0{\rm Tr}\big( \hat{Z}^* \big)
 +\overline{z}_0{\rm Tr}\big( \hat{Z} \big) \big)
 -\hat{\tau}N^{-1}
\sum_{\alpha=1}^t 
 \frac{c_{\alpha}}{f_{\alpha}^2}
 {\rm Tr}\big( H_{a_\alpha-z_0} \big)^2  
 \\
 &+(|z_0|^2-1)n\hat{\tau}^2 N^{-1}+ 
 |z_0|^2{\rm Tr}(\Gamma_1)     -{\rm Tr}\Big( (X-\frac{1}{2}B\Sigma^{-1})\Sigma (X-\frac{1}{2}B\Sigma^{-1})^t\Big)      \\
 &+
 N^{-\frac{1}{2}} {\rm Tr} \left(
\sum_{\alpha=2}^t  \widetilde{\mathcal{T}}_{\rm d,\alpha}
\Big(
\frac{1}{f_{\alpha}} H_{a_{\alpha}-z_0} ^2
-\frac{1}{f_{1}} H_{a_{1}-z_0} ^2+\hat{\tau}(f_{\alpha}-f_1)
\Big)\right)-\frac{\hat{\tau}}{P_1}N^{-\frac{1}{2}}{\rm Tr}\big( Q_0Q_0^* \big)
\\
 &-\sum_{\alpha=3}^t {\rm Tr}\Big(
 G_{\alpha}-\overline{\eta}_{\alpha}\Big(
 N^{-\frac{1}{4}}\big( \hat{Z}G_1+\frac{G_1\hat{Z}}{2} \big)
 +G_1^2
 \Big)
 \Big)\Big(
 G_{\alpha}-\overline{\eta}_{\alpha}\Big(
 N^{-\frac{1}{4}}\big( \hat{Z}G_1+\frac{G_1\hat{Z}}{2} \big)
 +G_1^2
 \Big)
 \Big)^*            \\
 &
 -{\rm Tr}\Big(\widetilde{K}_2(G_1,Q_0)
 \big( \frac{1}{2}\widetilde{K}_2(G_1,Q_0)+\widetilde{K}_1(G_1) \big)
 \Big)            -{\rm Tr}\Big( 
 Q_{t+1}\big(z_0\mathbb{I}_{r_{t+1}}-A_{t+1}\big)^*
 \big(z_0\mathbb{I}_{r_{t+1}}-A_{t+1}\big)Q_{t+1}^*
 \Big)    \\
 &+{\rm Tr}\big( \widetilde{K}_3(G_1)\big)+O\Big(
\big( \|G_1\|+N^{-\frac{1}{4}} \big)\Omega_{{\rm error},1}^2
+\big( \|G_1\|+N^{-\frac{1}{4}} \big)^2 \Omega_{{\rm error},2}
+\Omega_{{\rm error},2}^2+\|G_1\|^3\Omega_{{\rm error},1}
  \Big)    \\
 & +O\Big(
\big( \|\Gamma_1\|\big( \|G_1\|+N^{-\frac{1}{4}}+
\| \widetilde{\mathcal{T}}_{{\rm d}} \|_2+\| \Gamma_1 \| \big)+\Omega_{{\rm error},1}^3
+N^{-\frac{1}{2}}\big( \|G_1\|+N^{-\frac{1}{4}}\big)\Omega_{{\rm error},1}
  \Big),
\end{aligned}
\end{equation}
\end{small}
where
\begin{equation}\label{tildeK1g1 critical non0}
\begin{aligned}
&\widetilde{K}_1(G_1)=
\frac{\overline{P}_2}{\sqrt{P_1}}\Big(
G_1^2+\frac{1}{2}N^{-\frac{1}{4}}G_1\hat{Z}
\Big)+\frac{P_2}{\sqrt{P_1}}\Big(
G_1^2+\frac{1}{2}N^{-\frac{1}{4}}G_1\hat{Z}
\Big)^*
\\
&+\frac{1}{2}\sqrt{P_1}N^{-\frac{1}{4}}
\big( G_1\hat{Z}^*+\hat{Z}G_1^* \big)+\sqrt{P_1}G_1^*G_1
+\sqrt{P_1}N^{-\frac{1}{4}}
\big( \hat{Z}^*G_1+G_1^*\hat{Z} \big)          \\
&+N^{-\frac{1}{4}}\Big( 
\frac{\overline{P}_2}{\sqrt{P_1}}\hat{Z}G_1
+\frac{P_2}{\sqrt{P_1}}G_1^*\hat{Z}^*
 \Big)+N^{-\frac{1}{2}}
 \Big( \frac{P_2}{\sqrt{P_1}}\big( \hat{Z}^* \big)^2
 +2\sqrt{P_1}\hat{Z}\hat{Z}^*+\frac{\overline{P}_2}{\sqrt{P_1}}\hat{Z}^2
  \Big),
\end{aligned}
\end{equation}
\begin{equation}\label{tildeK2g1 critical non0}
\widetilde{K}_2(G_1,Q_0)=\frac{1}{\sqrt{P_1}}
\Big(
P_1 G_1G_1^*+Q_0Q_0^*+\frac{1}{2}N^{-\frac{1}{4}}G_1
\big( P_1\hat{Z}^*+\overline{P}_2\hat{Z} \big)
+\frac{1}{2}N^{-\frac{1}{4}}
\big( P_1\hat{Z}+P_2\hat{Z}^* \big)G_1^*
\Big)
\end{equation}
and
\begin{equation}\label{tildeK3g1 critical non0}
\begin{aligned}
\widetilde{K}_3(G_1)=
&-\frac{1}{4}N^{-\frac{1}{2}}G_1
\big( \overline{P}_2\hat{Z}^2+2P_1\hat{Z}\hat{Z}^*+P_2
\big( \hat{Z}^* \big)^2
 \big)G_1^*-\hat{\tau}N^{-\frac{1}{2}}G_1G_1^*
  \\
  &+\Big( P_1-\frac{|P_2|^2}{P_1} \Big)\Big( 
  N^{-\frac{1}{4}}\big( \hat{Z}G_1+\frac{1}{2}G_1\hat{Z} \big)
  +G_1^2
   \Big)\Big( 
  N^{-\frac{1}{4}}\big( \hat{Z}G_1+\frac{1}{2}G_1\hat{Z} \big)
  +G_1^2
   \Big)^*.
\end{aligned}
\end{equation}

Note that $\widetilde{K}_2(G_1)$ is just $K/\sqrt{P_1}$ given in \eqref{K critical non0}, and  by \eqref{Omega error2}
$$
\Omega_{{\rm error},2}=O\Big( \|\hat{G}_2\|+
\Omega_{{\rm error},1}\big( \|G_1\|+N^{-\frac{1}{4}}\big)
+\Omega_{{\rm error},1}^2
  \Big),
$$
so the error terms in \eqref{tau hQ expansion10} can be replaced by
\begin{equation}\label{tau hQ expansion10 error}
\begin{aligned}
&O\Big(
\big( \|G_1\|+N^{-\frac{1}{4}} \big)\Omega_{{\rm error},1}^2
+\|\hat{G}_2\|\big( \|G_1\|+N^{-\frac{1}{4}} \big)^2
+\Omega_{{\rm error},1}^2\big( \|G_1\|+N^{-\frac{1}{4}} \big)^3
  \Big)
\\&  +
  O\Big(
 \|\hat{G}_2\|^2+
  \|\Gamma_1\|\big( \|G_1\|+N^{-\frac{1}{4}}+\| \widetilde{\mathcal{T}}_{{\rm d}} \|_2+\| \widetilde{\mathcal{T}}_{{\rm d,1}} \| \big)+\Omega_{{\rm error},1}^3
  \Big).
\end{aligned}
\end{equation}

{\bf Step 7:  Taylor expansion  of  $Y$ and summary.}   
The part  relevant to $Y$ in $f(\boldsymbol{T},Y,\boldsymbol{Q})$ given in \eqref{fTY} is 
\begin{equation}
\begin{aligned}
-
\frac{N-n}{\tau N}
{\rm Tr}\big(
YY^*
\big)+\sum_{\alpha=1}^tc_{\alpha}\log\det(E_{\alpha})+
N^{-1}\sum_{\alpha=1}^t
(R_{\alpha,N}-n)\log\det\big(
E_{\alpha}
\big),
\end{aligned}
\end{equation}
where $E_{\alpha}$ is defined in  \eqref{Aa}.

Since $n$ is a fixed integer,  by  \eqref{tau inverse expansion} we have 
\begin{equation}\label{Y expansion0}
N^{-1}\sum_{\alpha=1}^t
(R_{\alpha,N}-n)\log\det\big(
E_{\alpha}
\big)=
N^{-1}\sum_{\alpha=1}^t n(R_{\alpha,N}-n)
\log(f_{\alpha})+N^{-1}O\big(
N^{-\frac{1}{4}}+\| Y \|
\big)
\end{equation}
and
\begin{equation}\label{Y expansion1}
\begin{aligned}
\frac{1}{\tau}\big(1-\frac{n}{N}\big){\rm Tr}(YY^*)=
{\rm Tr}(YY^*)-\hat{\tau}N^{-\frac{1}{2}}
{\rm Tr}(YY^*)+O\big( N^{-1}\| Y \|^2 \big).
\end{aligned}
\end{equation}
In  \eqref{logAalpha decompose}, we have
\begin{equation}\label{Y expansion2}
\begin{aligned}
\log\det\left(
\mathbb{I}_{2n}+\sqrt{\gamma_N}N^{-\frac{1}{4}}
\hat{E}_{\alpha}
\right)=
&N^{-\frac{1}{4}}{\rm Tr}\big( \hat{E}_{\alpha} \big)-
\frac{1}{2}N^{-\frac{1}{2}}{\rm Tr}\big(\hat{E}_{\alpha} \big)
\\
&+\frac{1}{3}N^{-\frac{3}{4}}{\rm Tr}\big( \hat{E}_{\alpha}^3 \big)
-\frac{1}{4}N^{-1}{\rm Tr}\big( \hat{E}_{\alpha}^4 \big)
+O\big( N^{-\frac{5}{4}} \big).
\end{aligned}
\end{equation}
By  the definition of  $\hat{E}_{\alpha}$ in \eqref{HAalpha}, noting that  
\begin{equation}
\begin{aligned}
&\big( \gamma_{N}f_{\alpha}\mathbb{I}_n+Y^*Y \big)^{-1}
=f_{\alpha}^{-1}\mathbb{I}_n-f_{\alpha}^{-2}Y^*Y+
O\big( N^{-1}+\| Y \|^4 \big),
\end{aligned}
\end{equation}
we arrive at 
\begin{equation}\label{TrHAalpha1}
\begin{aligned}
&{\rm Tr}\big( \hat{E}_{\alpha} \big)=
f_{\alpha}^{-1}\big( 
\overline{z_0-a_{\alpha}}{\rm Tr}(\hat{Z})
+(z_0-a_{\alpha}){\rm Tr}(\hat{Z}^*)
 \big)     \\
 &-f_{\alpha}^{-2}\Big(
\overline{z_0-a_{\alpha}}{\rm Tr}\big( Y^*Y\hat{Z} \big)
+(z_0-a_{\alpha}){\rm Tr}\big( YY^*\hat{Z}^* \big)
 \Big)+O\big( N^{-1}+\| Y \|^4 \big),
\end{aligned}
\end{equation}
\begin{equation}\label{TrHAalpha2}
\begin{aligned}
{\rm Tr}\big( \hat{E}_{\alpha}^2 \big)
&=
f_{\alpha}^{-2}\big( 
\overline{z_0-a_{\alpha}}^2{\rm Tr}(\hat{Z}^2)
+(z_0-a_{\alpha})^2{\rm Tr}(\hat{Z}^*)^2
 \big)     
-2\overline{z_0-a_{\alpha}}^2 f_{\alpha}^{-3}
 {\rm Tr}\big( Y^*Y\hat{Z}^2 \big)
\\ 
 &-2f_{\alpha}^{-2}{\rm Tr}\big( Y\hat{Z}Y^*\hat{Z}^* \big)
 -2(z_0-a_{\alpha})^2 f_{\alpha}^{-3}
 {\rm Tr}\big( YY^*\big(\hat{Z}^*\big)^2 \big)
+O\big( N^{-1}+\| Y \|^4 \big),
\end{aligned}
\end{equation}
\begin{equation}\label{TrHAalpha3}
\begin{aligned}
{\rm Tr}\big( {\hat{E}_{\alpha}}^3 \big)
=f_{\alpha}^{-3}\big( 
\overline{z_0-a_{\alpha}}^3{\rm Tr}(\hat{Z}^3)
+(z_0-a_{\alpha})^3{\rm Tr}(\hat{Z}^*)^3
 \big)     
+O\big( N^{-1}+\| Y \|^2 \big)
\end{aligned}
\end{equation}
and
\begin{equation}\label{TrHAalpha4}
\begin{aligned}
{\rm Tr}\big( \hat{E}_{\alpha} ^4 \big)
=
f_{\alpha}^{-4}\big( 
\overline{z_0-a_{\alpha}}^4{\rm Tr}(\hat{Z}^4)
+(z_0-a_{\alpha})^4{\rm Tr}(\hat{Z}^*)^4
 \big)     
+O\big( N^{-1}+\| Y \|^2 \big).
\end{aligned}
\end{equation}
Together with 
\begin{equation}\label{fin0Y critical non0}
\begin{aligned}
&{\rm Tr}(YY^*)-\sum_{\alpha=1}^t c_{\alpha}
\log\det\big( \gamma_N f_{\alpha}\mathbb{I}_n+YY^* \big)=
\\
&-\frac{n^2}{N}-n\sum_{\alpha=1}^t c_{\alpha} \log(f_{\alpha})
+\frac{P_1}{2}{\rm Tr}(YY^*)^2+O
\big(
\| Y \|^6+N^{-1}\| Y \|^2+N^{-2}
\big),
\end{aligned}
\end{equation}
we see from \eqref{logAalpha decompose} and  \eqref{Y expansion1}-\eqref{TrHAalpha4} that  
\begin{equation}\label{finY critical non0}
\begin{aligned}
&\frac{1}{\tau}\big(1-\frac{n}{N}\big){\rm Tr}(YY^*)-
\sum_{\alpha=1}^t c_{\alpha}
\log\det( E_{\alpha} )=
-\frac{n^2}{N}-n\sum_{\alpha=1}^t c_{\alpha} \log(f_{\alpha})
\\
&-N^{-\frac{1}{4}}\sum_{\alpha=1}^t
\frac{c_{\alpha}}{f_{\alpha}}\big( 
\overline{z_0-a_{\alpha}}{\rm Tr}(\hat{Z})
+(z_0-a_{\alpha}){\rm Tr}(\hat{Z}^*)
 \big)
 \\&+\frac{1}{2}N^{-\frac{1}{2}}\sum_{\alpha=1}^t
 \frac{c_{\alpha}}{f_{\alpha}^2}\big( 
\overline{z_0-a_{\alpha}}^2{\rm Tr}(\hat{Z}^2)
+(z_0-a_{\alpha})^2{\rm Tr}(\hat{Z}^*)^2
 \big)  
 \\&-\frac{1}{3}N^{-\frac{3}{4}}\sum_{\alpha=1}^t
 \frac{c_{\alpha}}{f_{\alpha}^3}\big( 
\overline{z_0-a_{\alpha}}^3{\rm Tr}(\hat{Z}^3)
+(z_0-a_{\alpha})^3{\rm Tr}(\hat{Z}^*)^3
 \big)     \\
 &+\frac{1}{4} N^{-1}\sum_{\alpha=1}^t
 \frac{c_{\alpha}}{f_{\alpha}^4}\big( 
\overline{z_0-a_{\alpha}}^4{\rm Tr}(\hat{Z}^4)
+(z_0-a_{\alpha})^4{\rm Tr}(\hat{Z}^*)^4
 \big)+\frac{1}{2}P_1{\rm Tr}(YY^*)^2
 \\&-N^{-\frac{1}{2}}\Big(
 \overline{P}_2{\rm Tr}( Y^*Y\hat{Z}^2 )
 +P_1{\rm Tr}( Y\hat{Z}Y^*\hat{Z}^* )
 +P_2{\rm Tr}( YY^*\big(\hat{Z}^*)^2 \big)
 +\hat{\tau}{\rm Tr}( YY^* )
 \Big)     \\
 &+O\big( \| Y \|^6+N^{-\frac{1}{4}}\| Y \|^4
+N^{-\frac{3}{4}}\| Y \|^2+N^{-\frac{5}{4}}
  \big).
\end{aligned}
\end{equation}

Recalling  $f(\boldsymbol{T},Y,\boldsymbol{Q})$ in \eqref{fTY},  combining \eqref{tau hQ expansion10} and \eqref{finY critical non0} we have
\begin{equation}\label{fTY expansion critical non0}
f(\boldsymbol{T},Y,\boldsymbol{Q})-\widetilde{f}
  \big(\big(\sqrt{ c_{\alpha}/f_{\alpha}}\mathbb{I}_n\big)_{\alpha},0,0\big)
=F_0+F_1+O(F_2),
\end{equation}  
where
\begin{multline}\label{f0 critical non0}
F_0=\frac{n^2}{N}
+N^{-1}\sum_{\alpha=1}^t \Big(
nR_{\alpha,N}-\frac{n(n+1)}{2}
\Big)\log c_{\alpha}
-N^{-1}\frac{n(n-1)}{2}\sum_{\alpha=1}^t\log f_{\alpha}
\\
+N^{-\frac{1}{4}}\big(
\overline{z}_0{\rm Tr}(\hat{Z})+z_0{\rm Tr}\big(\hat{Z}^*\big)
\big)+N^{-\frac{1}{2}}{\rm Tr}\big( \hat{Z}\hat{Z}^* \big)-\frac{N^{-1}}{4}
\sum_{\alpha=1}^t
 \frac{c_{\alpha}}{f_{\alpha}^4}{\rm Tr}\big( H_{a_\alpha-z_0}^4 \big)    
   \\
-\frac{N^{-1}}{4}\sum_{\alpha=1}^t
 \frac{c_{\alpha}}{f_{\alpha}^4}\Big( 
\overline{z_0-a_{\alpha}}^4{\rm Tr}\big(\hat{Z}^4\big)
+(z_0-a_{\alpha})^4{\rm Tr}\big(\hat{Z}^*\big)^4
 \Big)    -\hat{\tau}N^{-1}\sum_{\alpha=1}^t
\frac{c_{\alpha}}{f_{\alpha}^2}{\rm Tr}\big( H_{a_\alpha-z_0}^2
\big)  
 \\
 +n\hat{\tau}N^{-\frac{1}{2}}(1-|z_0|^2)
 -\hat{\tau}N^{-\frac{3}{4}}\big(
\overline{z}_0{\rm Tr}(\hat{Z})+z_0{\rm Tr}\big(\hat{Z}^*\big)
\big)    -n\hat{\tau}^2N^{-1}(1-|z_0|^2),
\end{multline}
\begin{small}
\begin{equation}\label{f1 critical non0}
\begin{aligned}
&F_1=
|z_0|^2{\rm Tr}(\Gamma_1)     -{\rm Tr}\Big( (X-\frac{1}{2}B\Sigma^{-1})\Sigma (X-\frac{1}{2}B\Sigma^{-1})^t\Big)
  -{\rm Tr}\Big(\widetilde{K}_2(G_1,Q_0)
 \big( \frac{1}{2}\widetilde{K}_2(G_1,Q_0)+\widetilde{K}_1(G_1) \big)
 \Big)
\\& +{\rm Tr}\big( \widetilde{K}_3(G_1)\big)+
 N^{-\frac{1}{2}} {\rm Tr} \bigg(
\sum_{\alpha=2}^t  \widetilde{\mathcal{T}}_{\rm d,\alpha}
\Big(
\frac{1}{f_{\alpha}} H_{a_{\alpha}-z_0} ^2
-\frac{1}{f_{1}} H_{a_{1}-z_0} ^2+\hat{\tau}(f_{\alpha}-f_1)
\Big)\bigg)
 \\
 &-\sum_{\alpha=3}^t {\rm Tr}\Big(
 G_{\alpha}-\overline{\eta}_{\alpha}\Big(
 N^{-\frac{1}{4}}\big( \hat{Z}G_1+\frac{G_1\hat{Z}}{2} \big)
 +G_1^2
 \Big)
 \Big)\Big(
 G_{\alpha}-\overline{\eta}_{\alpha}\Big(
 N^{-\frac{1}{4}}\big( \hat{Z}G_1+\frac{G_1\hat{Z}}{2} \big)
 +G_1^2
 \Big)
 \Big)^*           \\
 &
            -{\rm Tr}\Big( 
 Q_{t+1}\big(z_0\mathbb{I}_{r_{t+1}}-A_{t+1}\big)^*
 \big(z_0\mathbb{I}_{r_{t+1}}-A_{t+1}\big)Q_{t+1}^*
 \Big)  -\frac{P_1}{2}{\rm Tr}(YY^*)^2  \\
 &
 +N^{-\frac{1}{2}}\Big(
 \overline{P}_2{\rm Tr}\big( Y^*Y\hat{Z}^2 \big)
 +P_1{\rm Tr}\big( Y\hat{Z}Y^*\hat{Z}^* \big)
 +P_2{\rm Tr}\Big( YY^*\big(\hat{Z}^*\big)^2 \Big)
 +\hat{\tau}{\rm Tr}\big( YY^* \big)
 \Big),
\end{aligned}
\end{equation}
\end{small}
and
\begin{equation}\label{f2 critical non0}
\begin{aligned}
F_2&=\big( \|G_1\|+N^{-\frac{1}{4}} \big)\Omega_{{\rm error},1}^2
+\| \hat{G}_2 \|\big( \|G_1\|+N^{-\frac{1}{4}} \big)^2
+\| \hat{G}_2 \|^2  +\Omega_{{\rm error},1}^3          \\
&\ \ +\Omega_{{\rm error},1}
\big( \|G_1\|+N^{-\frac{1}{4}} \big)^3+\|\Gamma_1\|
\big( \|G_1\|+N^{-\frac{1}{4}}+\| \widetilde{\mathcal{T}}_{\rm d} \|_2+
\|\Gamma_1\| \big) \big)
\\
&\ \ +\| Y \|^6+N^{-\frac{1}{4}}\| Y \|^4
+N^{-\frac{3}{4}}\| Y \|^2+N^{-\frac{5}{4}}.
\end{aligned}
\end{equation}
Here $\Sigma, X,$ and   $B$  are given in   \eqref{Sigma}  and \eqref{XB}.

%%%%%%%%%%%%%%%%%%%%%%%%%%%%%%%%%%%%%%%%%%%%%%%%%%%%%%%%%%%%%%%%%%%%%%%%%%%%%%%%%%%%%%%%%%%%%%%%%%%%%%%%%%%%%%%%%%%%%%%%%%%%%%%%%%%%%%%%%%%%%%%%%%%%
\subsection{Taylor expansion of $ \det\big(\widehat{L}_1+\sqrt{\gamma_N}\widehat{L}_2\big)$.} \label{sect3.2}
Recalling $\hat{L}_1$ and $\hat{L}_2$ defined in    \eqref {L1hat} and \eqref{L0hat},
since  $z_0$ is not an eigenvalue of $A_{t+1}$ and all $a_{\alpha}\neq z_0$,  
$\widehat{L}_1$ is  invertible when $N$ is sufficiently large.  Together with the decomposition 
\begin{equation}\label{L6}
\widehat{L}_2
=\left[\begin{smallmatrix}
\left[\begin{smallmatrix}
a_1\mathbb{I}_n & \\ & \overline{a}_{1}\mathbb{I}_n
\end{smallmatrix}\right]\otimes T_1^*\\
\vdots \\
\left[\begin{smallmatrix}
a_t\mathbb{I}_n & \\ & \overline{a}_{t}\mathbb{I}_n
\end{smallmatrix}\right]\otimes T_t^* \\
 \left[\begin{smallmatrix}
\mathbb{I}_n \otimes  
\widetilde{A}_{t+1}\widetilde{Q}_{t+1}^*  & \\ & \mathbb{I}_n \otimes  
\widetilde{A}_{t+1}^* \widetilde{Q}_{t+1}^*
\end{smallmatrix}\right]
 \end{smallmatrix}
\right] \Big[ 
 \mathbb{I}_{2n}\otimes T_1  
,\cdots,  
 \mathbb{I}_{2n}\otimes T_t,  
\mathbb{I}_{2n}\otimes \widetilde{Q}_{t+1}
 \Big],
%\big[\begin{smallmatrix}
% \mathbb{I}_{2n}\otimes T_1 &
%,\cdots, &
% \mathbb{I}_{2n}\otimes T_t, &
%\mathbb{I}_{2n}\otimes \widetilde{Q}_{t+1}
%\end{smallmatrix}\big],
\end{equation} 
change the order of the above matrix product and  we obtain 
 \begin{multline}\label{detMrewrite0}
\det\!\bigg(
\widehat{L}_1+\sqrt{\gamma_N}\widehat{L}_2
\bigg)=\big|\det\big(z_0\mathbb{I}_{r_{t+1}}-A_{t+1}\big)\big|^{2n}
\Big( \prod_{\alpha=1}^t \det\big( E_{\alpha} \big) \Big)^n
\det\big( E_{0} \big)^{r_0-n}
\big( 1+O(\| Y \|+N^{-\frac{1}{4}}) \big) \\
\times \det\bigg(
E_0\otimes \mathbb{I}_{n}+\sqrt{\gamma_N}
 (E_0\otimes \mathbb{I}_{n})
 \begin{bmatrix}
Q_{1,1} & Q_{1,2} \\ Q_{2,1} & Q_{2,2}
\end{bmatrix} +\sqrt{\gamma_N} 
\begin{bmatrix}
z_0\mathbb{I}_n & \\ & \overline{z}_0\mathbb{I}_n
\end{bmatrix}\otimes (Q_0Q_0^*)
+O\big( \| Q_{t+1} \|^2 \big)
\bigg),
\end{multline}
where 
\begin{equation}\label{tildeA alpha inverse rewrite}
E_0:=\begin{bmatrix}
\sqrt{\gamma_N}N^{-\frac{1}{4}}\hat{Z}
 & -Y^*
 \\ Y & 
\sqrt{\gamma_N}N^{-\frac{1}{4}}\hat{Z}^*
\end{bmatrix},\quad
\begin{bmatrix}
Q_{1,1} & Q_{1,2} \\ Q_{2,1} & Q_{2,2}
\end{bmatrix}:=\sum_{\alpha=1}^t
\left( E_{\alpha}^{-1} \left[\begin{smallmatrix}
a_{\alpha}\mathbb{I}_n & \\ & \overline{a}_{\alpha}\mathbb{I}_n
\end{smallmatrix}\right] \right) \otimes
(T_{\alpha}T_{\alpha}^*).
\end{equation} 
Obviously, by \eqref{A alpha}, we have  
\begin{equation}\label{Qij}
\begin{aligned}
&Q_{1,1}=\sum_{\alpha=1}^t
a_{\alpha}\Big(
Z_{\alpha}+Y^*\big(Z_{\alpha}^*\big)^{-1}Y
\Big)^{-1}\otimes
(T_{\alpha}T_{\alpha}^*),\quad
Q_{2,2}=\sum_{\alpha=1}^t
\overline{a}_{\alpha} 
\big( Z_{\alpha}^*+YZ_{\alpha}^{-1}Y^* \big)^{-1} \otimes
(T_{\alpha}T_{\alpha}^*),
\\
&Q_{1,2}=\sum_{\alpha=1}^t \overline{a}_{\alpha}
Z_{\alpha}^{-1}Y^* \big( Z_{\alpha}^*+YZ_{\alpha}^{-1}Y^* \big)^{-1}
\otimes
(T_{\alpha}T_{\alpha}^*),   
\\
&
Q_{2,1}=-\sum_{\alpha=1}^t a_{\alpha}\big(Z_{\alpha}^*\big)^{-1}
Y \Big(
Z_{\alpha}+Y^*\big(Z_{\alpha}^*\big)^{-1}Y
\Big)^{-1}\otimes
(T_{\alpha}T_{\alpha}^*).
\end{aligned}
\end{equation}

In order to keep  the order of $Q_{i,j}$ up to $N^{-\frac{1}{4}}$, 
simple calculations show that 
$$  
Z_{\alpha}^{-1}=(z_0-a_{\alpha})^{-1}\mathbb{I}_n
-N^{-\frac{1}{4}}(z_0-a_{\alpha})^{-2}\hat{Z}+O\big( N^{-\frac{1}{2}} \big),
$$
\begin{small}
\begin{equation}\label{detM expansion1}
\begin{aligned}
\big( Z_{\alpha}+Y^*\big( Z_{\alpha}^* \big)^{-1}Y \big)^{-1}
=(z_0-a_{\alpha})^{-1}\mathbb{I}_n
-N^{-\frac{1}{4}}(z_0-a_{\alpha})^{-2}\hat{Z}+O\big( N^{-\frac{1}{2}}
+\| Y \|^2\big),   
\end{aligned}
\end{equation}
\end{small}
and 
\begin{equation*}
\begin{aligned}
T_{\alpha}T_{\alpha}^*
%&=
%T_{{\rm d},\alpha}+\sqrt{T_{{\rm d},\alpha}} T_{{\rm u},\alpha}^*
%+ T_{{\rm u},\alpha}\sqrt{T_{{\rm d},\alpha}}+T_{{\rm u},\alpha}T_{{\rm u},\alpha}^*
%\\
&=\frac{c_{\alpha}}{f_{\alpha}}\mathbb{I}_n+\widetilde{T}_{{\rm d},\alpha}+
\frac{\sqrt{c_{\alpha}}}{f_{\alpha}}
\big( \hat{T}_{{\rm u},\alpha}+\hat{T}_{{\rm u},\alpha}^* \big)
+O\big( \| \hat{T}_{\rm u} \|^2
+\| \hat{T}_{\rm u} \|\| \widetilde{T}_{\rm d} \| \big)\\
&=\frac{c_{\alpha}}{f_{\alpha}}
\mathbb{I}_n+N^{-\frac{1}{4}}
\frac{c_{\alpha}}{f_{\alpha}^2}H_{a_{\alpha}-z_0}
+\frac{c_{\alpha}}{f_{\alpha}^2}\big(\overline{a_{\alpha}-z_0}G_1
+ (a_{\alpha}-z_0)G_1^*)
+O\big( \Omega_{{\rm error},1} \big),
\end{aligned}
\end{equation*}
where $\Omega_{{\rm error},1}$  is defined in \eqref{OmegaError1}.
For the latter, we have used the facts 
%Recall  and \eqref{g hatT connection1}, we have
\begin{equation*}
\| \hat{T}_{\rm u} \|=O\big( \| G_1 \|+\Omega_{{\rm error},1} \big),\quad
\| \widetilde{T}_{\rm d} \|=O\big( \| G_1 \|^2
+\| \Gamma_1 \| +  \| \widetilde{\mathcal{T}}_{{\rm d}} \|_2+N^{-\frac{1}{4}}
+\Omega_{{\rm error},1}^2 \big),
\end{equation*} 
\begin{equation*}
\| \hat{T}_{\rm u} \|
\big( \| \hat{T}_{\rm u} \|+\| \widetilde{T}_{\rm d} \| \big)
=O\Big(
\big( \| G_1 \|+\Omega_{{\rm error},1} \big)
\big( \| G_1 \|+N^{-\frac{1}{4}}+\Omega_{{\rm error},1} \big)
\Big).
\end{equation*} 
 and 
$$  
\hat{T}_{{\rm u},\alpha} =\frac{\sqrt{c_{\alpha}}\overline{z_0-a_{\alpha}}}
{f_{\alpha}}G_1+O\big( \Omega_{{\rm error},1} \big), \ \widetilde{T}_{{\rm d},\alpha}=N^{-\frac{1}{4}}
\frac{c_{\alpha}}{f_{\alpha}^2} H_{a_{\alpha}-z_0}
+O\big(\| \Gamma_1 \| +  \| \widetilde{\mathcal{T}}_{{\rm d}} \|_2+\| G_1 \|^2+\Omega_{{\rm error},1}^2 \big).
$$
 
Therefore,   by  $P_0(z_0)=0$ we can obtain 
 
 \begin{equation}\label{detM expansion3}
Q_{1,2}=\overline{z}_0P_1 Y^*\otimes \mathbb{I}_n+\| Y \|
 O\big( 
 N^{-\frac{1}{4}}+\| G_1 \|
 \Omega_{{\rm error},1} \big).
\end{equation}
 
 \begin{equation}\label{detM expansion4}
Q_{2,1}=-z_0P_1 Y\otimes \mathbb{I}_n+\| Y \|
 O\big( 
 N^{-\frac{1}{4}}+\| G_1 \|
 \Omega_{{\rm error},1} \big),
\end{equation}
 \begin{equation}\label{detM expansion5}
 \begin{aligned}
\mathbb{I}_{n^2}+\sqrt{ \gamma_N}Q_{1,1}=
&-z_0\Big(
N^{-\frac{1}{4}}\mathbb{I}_{n}\otimes
\big( P_1\hat{Z}^*+\overline{P}_2\hat{Z} \big)
+\mathbb{I}_{n}\otimes
\big( P_1G_1^*+\overline{P}_2G_1 \big)
\Big)
\\&
-z_0N^{-\frac{1}{4}}\overline{P}_2 \hat{Z}\otimes \mathbb{I}_{n}
+O\big( 
\| Y \|^2+
 \Omega_{{\rm error},1} \big)
 \end{aligned}
\end{equation}
and
 \begin{equation}\label{detM expansion6}
 \begin{aligned}
\mathbb{I}_{n^2}+\sqrt{\gamma_N}Q_{2,2}=
&-\overline{z}_0\Big(
N^{-\frac{1}{4}}\mathbb{I}_{n}\otimes
\big( P_1\hat{Z}+P_2\hat{Z}^* \big)
+\mathbb{I}_{n}\otimes
\big( P_1 G_1+P_2G_1^* \big)
\Big)
\\&
-\overline{z}_0
N^{-\frac{1}{4}}P_2 \hat{Z}^*\otimes \mathbb{I}_{n}
+O\big( 
\| Y \|^2+
 \Omega_{{\rm error},1} \big).
 \end{aligned}
\end{equation}

Finally, combination of  \eqref{detMrewrite0}, \eqref{tildeA alpha inverse rewrite}, \eqref{detM expansion3}-\eqref{detM expansion6} gives rise to   
 \begin{multline}\label{detM expansion6 final}
\det\Big(
\widehat{L}_1+\sqrt{\gamma_N}\widehat{L}_2
\Big)=|z_0|^{2n^2} \big| \det\big(z_0\mathbb{I}_{r_{t+1}}-A_{t+1}\big) \big|^{2n}
\Big(\prod_{\alpha=1}^t f_{\alpha}^{n^2}
\Big)\det\big( E_{0} \big)^{r_0-n} \\
\times \bigg(
\det\begin{bmatrix}
P_1(Y^*Y)\otimes \mathbb{I}_n+\widetilde{F}_{1,1} 
& \widetilde{F}_{1,2}
+N^{-\frac{1}{4}}\big(
P_1\hat{Z}Y^*+P_2Y^*\hat{Z}^*
\big)\otimes \mathbb{I}_n      \\
-\widetilde{F}_{1,2}^*-N^{-\frac{1}{4}}\big(
P_1\hat{Z}^*Y+\overline{P}_2 Y\hat{Z}
\big)\otimes \mathbb{I}_n & 
\widetilde{F}_{1,1}^*+P_1(YY^*)\otimes \mathbb{I}_n
\end{bmatrix}
\\
+O\Big(
\| Y \|^{4n^2+1}+\| G_1 \|^{4n^2+1}+
N^{-n^2-\frac{1}{4}}+\Omega_{{\rm error},1}^{2n^2+\frac{1}{2}}
\Big)
\bigg),
\end{multline}
where $A_0$ and $\Omega_{{\rm error},1}$ are defined in \eqref{tildeA alpha inverse rewrite} and \eqref{OmegaError1} respectively, 
\begin{equation}\label{tildeK11}
\widetilde{F}_{1,1}=
\mathbb{I}_n\otimes (Q_0Q_0^*)-N^{-\frac{1}{2}}
\overline{P}_2\hat{Z}^2\otimes \mathbb{I}_n
-N^{-\frac{1}{4}}\hat{Z}\otimes 
\Big(
N^{-\frac{1}{4}}\big(
P_1\hat{Z}^*+\overline{P}_2\hat{Z}
\big)+P_1G_1^*+\overline{P}_2G_1 
\Big)
\end{equation} 
and
\begin{equation}\label{tildeK12}
\widetilde{F}_{1,2}=Y^*
\otimes 
\Big(
N^{-\frac{1}{4}}\big(
P_1\hat{Z}+P_2\hat{Z}^*
\big)+P_1G_1+P_2G_1^*
\Big).
\end{equation}

\subsection{Matrix integrals and final proof} \label{sect3.3}

 Finally, with those preparations  in Section \ref{sect3.1} and Section  \ref{sect3.2}, we will  give a  compete  proof of Theorem \ref{2-complex-correlation critical2}.

%Now we give a general argument to eliminate the error terms in $f(T,Y)$ and $g(Y,U,T)$, and only retain the terms of correct order. 
Recalling \eqref{INdelta}, \eqref{partial g2} and \eqref{fTY expansion critical non0}, we rewrite
 \begin{equation}\label{INdeltaexpan critical non0}
e^{-NF_0}I_{N,\delta}:=J_{1,N}+J_{2,N},
\end{equation}
 where

 \begin{multline}\label{J1N critical non0}
J_{1,N}=\prod_{\alpha=1}^t f_{\alpha}^{-\frac{n(n-1)}{2}}
\int_{\hat{\Omega}_{N,\delta}}
 \det\Big(
\widehat{L}_1+\sqrt{\gamma_N}\widehat{L}_2
\Big)
e^{NF_1} \Big( 
1+
\\O\Big( \| G_1 \|+\| \hat{G}_2 \|+
\sum_{\beta=3}^t
\| G_{\beta} \|+\|\widetilde{\mathcal{T}}_{{\rm d}} \|_2+
\| Q_{t+1} \|+
\| Q_{0} \|^2+\|\Gamma_1\|+N^{-\frac{1}{4}}\Big)
  \Big) {\rm d}\widetilde{V},
\end{multline}
and  \begin{multline}\label{J2N critical non0}
J_{2,N}=\prod_{\alpha=1}^t f_{\alpha}^{-\frac{n(n-1)}{2}}
 \int_{\hat{\Omega}_{N,\delta}}
 \det\Big(
\widehat{L}_1+\sqrt{\gamma_N}\widehat{L}_2
\Big) 
 e^{NF_1}
\big(
e^{O(NF_2)}-1
\big) 
 \Big( 
1+
\\
O\Big( \| G_1 \|+\| \hat{G}_2 \|+
\sum_{\beta=3}^t
\| G_{\beta} \|+\|\widetilde{\mathcal{T}}_{{\rm d}} \|_2+
\| Q_{t+1} \|+
\| Q_{0} \|^2+\|\Gamma_1\|+N^{-\frac{1}{4}}\Big)
  \Big) {\rm d}\widetilde{V},
\end{multline}
with 
\begin{equation}\label{volume element critical non0}
{\rm d}\widetilde{V}={\rm d}\Gamma_1{\rm d}G_1
{\rm d}\hat{G_2}{\rm d}Y {\rm d}Q_0{\rm d}Q_{t+1}
\prod_{\alpha=3}^t{\rm d}G_{\alpha}
\prod_{\alpha=2}^t{\rm d}  \widetilde{\mathcal{T}}_{{\rm d,\alpha}},
\end{equation} 
and 
\begin{equation}\label{hat Omega N delta critical non0}
\hat{\Omega}_{N,\delta}=A_{N,\delta}\cap \big\{
\Gamma_1+\sqrt{P_1}\big( \hat{G}_2+\hat{G}_2^* \big)\leq 0
\big\}; 
\end{equation}
 cf.  \eqref{ANdelta} for definition of $A_{N,\delta}$.

Under  the restriction condition  of $\Gamma_1+\sqrt{P_1}\big( \hat{G}_2+\hat{G}_2^* \big)\leq 0$,  every   principal minor  of order 2 is non-positive definite, so all diagonal entries of $\Gamma_1$  are zero or negative and  
\begin{equation}\label{g2hatU1 estimation}
{\rm Tr}\big(
\hat{G}_2\hat{G}_2^*
\big)=\sum_{1\leq i<j\leq n}\big| (\hat{G}_{2})_{i,j} \big|^2
\leq \frac{1}{P_1} \sum_{1\leq i<j\leq n} (\Gamma_1)_{i,i} (\Gamma_1)_{j,j}
\leq  \frac{1}{P_1}
  {\rm Tr}\big(
\Gamma_1  
\big)^2.
\end{equation}
Noticing the absence of $ \hat{G}_2$ in $F_1$ given  in \eqref{f1 critical non0}, for convergence we need to control  $\hat{G}_2$ by $
\Gamma_1$  in $F_2$,  just as shown in  \eqref{g2hatU1 estimation}. Since  $\| \hat{G}_2 \|= O(\| \Gamma_1 \|)$ from   \eqref{g2hatU1 estimation}, for sufficiently large $N$ and small $\delta$, there exists some $C>0$ such that 
\begin{equation*}
\begin{aligned}
&\frac{1}{C} NF_2 \leq N^{-\frac{1}{4}}+\sqrt{\delta}+
\sqrt{N\delta} \big(
\| G_1 \|^2+\sum_{\alpha=3}^t \| G_{\alpha} \|
+\|\widetilde{\mathcal{T}}_{{\rm d}} \|_2
+\| Y \|^2+\| Q_{t+1} \|+\| Q_0 \|^2
\big)            \\
&+N\sqrt{\delta}\big(
\| G_1 \|^4+\sum_{\alpha=3}^t \| G_{\alpha} \|^2
+\|\widetilde{\mathcal{T}}_{{\rm d}} \|_2^2
+\| Y \|^4+\| Q_{t+1} \|^2+\| Q_0 \|^4+\| \Gamma_1 \|
\big).
\end{aligned}
\end{equation*}

 Using the inequality \begin{equation}\label{J2N inequ critical non0}
\big|
e^{O(NF_2)}-1
\big|\leq O(N|F_2|)e^{O(N|F_2|)},
\end{equation}
 after change of  variables
\begin{equation}\label{change scale critical non0}
\begin{aligned}
&(G_1,Q_0,Y)\rightarrow N^{-\frac{1}{4}}(G_1,Q_0,Y),\quad (\hat{G}_2,\Gamma_1)\rightarrow N^{-1}(\hat{G}_2,\Gamma_1),\quad
Q_{t+1} \rightarrow N^{-\frac{1}{2}}Q_{t+1},
\\
&\widetilde{\mathcal{T}}_{{\rm d,\alpha}}
\rightarrow N^{-\frac{1}{2}}  \widetilde{\mathcal{T}}_{{\rm d,\alpha}},\quad
G_{\beta}\rightarrow N^{-\frac{1}{2}}G_{\beta},\quad \alpha=2,\cdots,t,
\beta=3,\cdots,t,
\end{aligned}
\end{equation}
the term $O(NF_2)$ in \eqref{J2N inequ critical non0} has an upper bound  by   $N^{-\frac{1}{4}}P(\tilde{V})$ for some polynomial of variables.  
since $F_1$ can control $F_2$  for sufficiently  small $\delta$, by the argument of  Laplace method and the dominant convergence theorem we know that 
$J_{2,N}$ is typically of order  $N^{-\frac{1}{4}}$ compared with $J_{1,N}$, that is,  
%  \eqref{J1N critical non0} and \eqref{J2N critical non0}
 \begin{equation}\label{J2Nestimation critical non0}
J_{2,N}=O\big(
N^{-\frac{1}{4}}
\big)J_{1,N}.
\end{equation}

For $J_{1,N}$,   
take  a large  $M_0>0$ such that
\begin{multline*}%\label{ANdelta complement critical non0}
A_{N,\delta}^{\complement}\subseteq  
  \bigcup_{\alpha=3}^t
\Big\{
{\rm Tr}\big(
G_{\alpha}G_{\alpha}^*
\big)>\frac{\delta}{M_0}
\Big\}
\bigcup
\bigcup_{\alpha=2}^t
\Big\{
{\rm Tr}\big(
\widetilde{\mathcal{T}}_{{\rm d,\alpha}}\widetilde{\mathcal{T}}_{{\rm d,\alpha}}^*
\big)>\frac{\delta}{M_0}
\Big\} \bigcup \Big\{{\rm Tr}\big(
G_{1}G_{1}^*
\big)>\frac{\delta}{M_0}\Big\}
        \\
\bigcup \Big\{{\rm Tr}\big(
\Gamma_{1}\Gamma_{1}^*
\big)>\frac{\delta}{M_0}\Big\} 
 \bigcup
\Big\{
{\rm Tr}\big(
YY^*
\big)>\frac{\delta}{M_0}
\Big\}
\bigcup
\Big\{
{\rm Tr}\big(
Q_{t+1}Q_{t+1}^*
\big)>\frac{\delta}{M_0}
\Big\}\bigcup
\Big\{
{\rm Tr}\big(
Q_0Q_0^*
\big)>\frac{\delta}{M_0}
\Big\}.
\end{multline*}
Here we have used \eqref{g2hatU1 estimation} to drop out the domain   $ \{{\rm Tr}\big(
\hat{G}_{2} \hat{G}_{2}^*
\big)>\delta/M_0\}$.   For each piece of  domain,  only  keep  the restricted matrix variable and let  the others free,  it's easy to prove that the corresponding  matrix integral  is exponentially  small, that is,  $O\big( e^{-\delta_1 N} \big)
$
for some $\delta_1>0$.  

So we can extend the integration  region from $\hat{\Omega}_{N,\delta}$ to  $\big\{ \Gamma_1+\sqrt{P_1}\big( \hat{G}_2+\hat{G}_2^* \big)\leq 0 \big\}$,  
and by  the change of  variables    
\eqref{change scale critical non0} we have
\begin{multline}\label{J1N change critical non0}
J_{1,N}=N^{-\frac{n(n+1)}{4}-\frac{n^2t}{2}-n(r_{t+1}+r_0+n)}
|z_0|^{2n^2}\big| \det\big( z_0\mathbb{I}_{r_{t+1}}-A_{t+1} \big) \big|^{2n}    \prod_{\alpha=1}^t
f_{\alpha}^{\frac{n(n+1)}{2}}    
   \\
\times
\Big( I_0+O\big( N^{-\frac{1}{4}} \big) \Big),
\end{multline}
where, with $\hat{H}=\Gamma_1+\sqrt{P_1}\big( \hat{G}_2+\hat{G}_2^* \big)$, 
\begin{multline}\label{I0 critical non0}
I_0=\int_{\hat{H}\leq 0}\bigg( \det\begin{bmatrix}
\hat{Z}
 & -Y^*
 \\ Y & 
\hat{Z}^*
\end{bmatrix}\bigg)^{r_0-n}
\\
\times
\det\begin{bmatrix}
P_1(Y^*Y)\otimes \mathbb{I}_n+\widehat{F}_{1,1} 
& \widehat{F}_{1,2}
+\big(
P_1\hat{Z}Y^*+P_2Y^*\hat{Z}^*
\big)\otimes \mathbb{I}_n      \\
-\widehat{F}_{1,2}^*-\big(
P_1\hat{Z}^*Y+\overline{P}_2 Y\hat{Z}
\big)\otimes \mathbb{I}_n & 
\widehat{F}_{1,1}^*+P_1(YY^*)\otimes \mathbb{I}_n
\end{bmatrix}
e^{F}{\rm d}\widetilde{V},
\end{multline} 
\begin{equation}\label{F critical non0}
\begin{aligned}
F=&
|z_0|^2{\rm Tr}(\Gamma_1)     -{\rm Tr}\Big( (X-\frac{1}{2}B\Sigma^{-1})\Sigma (X-\frac{1}{2}B\Sigma^{-1})^t\Big)        \\
&+
   {\rm Tr} \bigg(
\sum_{\alpha=2}^t  \widetilde{\mathcal{T}}_{\rm d,\alpha}
\Big(
\frac{1}{f_{\alpha}} H_{a_{\alpha}-z_0} ^2
-\frac{1}{f_{1}} H_{a_{1}-z_0} ^2+\hat{\tau}(f_{\alpha}-f_1) \mathbb{I}_n
\Big)\bigg)
 \\
 &-\sum_{\alpha=3}^t {\rm Tr}\Big(
 G_{\alpha}-\overline{\eta}_{\alpha}\big(
 \hat{Z}G_1+\frac{1}{2}G_1\hat{Z} 
 +G_1^2
 \big)
 \Big)\Big(
 G_{\alpha}-\overline{\eta}_{\alpha}\big(
  \hat{Z}G_1+\frac{1}{2}G_1\hat{Z} 
 +G_1^2
 \big)
 \Big)^*            \\
 &+{\rm Tr}\big( \hat{K}_3(G_1) \big)
 -{\rm Tr}\Big(\hat{K}_2(G_1,Q_0)
 \big( \frac{1}{2}\hat{K}_2(G_1,Q_0)+\hat{K}_1(G_1) \big)
 \Big)
 -\frac{\hat{\tau}}{P_1}{\rm Tr}(Q_0Q_0^*)
 \\
& -{\rm Tr}\Big( 
 Q_{t+1}\big(z_0\mathbb{I}_{r_{t+1}}-A_{t+1}\big)^*
 \big(z_0\mathbb{I}_{r_{t+1}}-A_{t+1}\big)Q_{t+1}^*
 \Big)-\frac{1}{2}P_1{\rm Tr}(YY^*)^2
 \\&+
 \overline{P}_2{\rm Tr}\big( Y^*Y\hat{Z}^2 \big)
 +P_1{\rm Tr}\big( Y\hat{Z}Y^*\hat{Z}^* \big)
 +P_2{\rm Tr}\Big( YY^*\big(\hat{Z}^*\big)^2 \Big)
 +\hat{\tau}{\rm Tr}\big( YY^* \big).
\end{aligned}
\end{equation}
Here 
\begin{equation}\label{hatK1g1 critical non0}
\begin{aligned}
&\hat{K}_1(G_1)=
\frac{\overline{P}_2}{\sqrt{P_1}}\Big(
G_1^2+\frac{G_1\hat{Z}}{2}
\Big)+\frac{P_2}{\sqrt{P_1}}\Big(
G_1^2+\frac{G_1\hat{Z}}{2}
\Big)^*
\\
&+\frac{1}{2}\sqrt{P_1}
\big( G_1\hat{Z}^*+\hat{Z}G_1^* \big)+\sqrt{P_1}G_1^*G_1
+\sqrt{P_1}
\big( \hat{Z}^*G_1+G_1^*\hat{Z} \big)          \\
&+ 
\frac{1}{\sqrt{P_1}}\Big( \overline{P}_2\hat{Z}G_1
+P_2 G_1^*\hat{Z}^*
 + P_2 \big( \hat{Z}^* \big)^2
 +\overline{P}_2 \hat{Z}^2+2 P_1 \hat{Z}\hat{Z}^*\Big)
  ,
\end{aligned}
\end{equation}
\begin{equation}\label{hatK2g1 critical non0}
\hat{K}_2(G_1,Q_0)=\frac{1}{\sqrt{P_1}}
\Big(
P_1 G_1G_1^*+Q_0Q_0^*+\frac{1}{2}G_1
\big( P_1\hat{Z}^*+\overline{P}_2\hat{Z} \big)
+\frac{1}{2}
\big( P_1\hat{Z}+P_2\hat{Z}^* \big)G_1^*
\Big),
\end{equation}
\begin{equation}\label{hatK3g1 critical non0}
\begin{aligned}
\hat{K}_3(G_1)=
&-\frac{1}{4}G_1
\big( \overline{P}_2\hat{Z}^2+2P_1\hat{Z}\hat{Z}^*+P_2
\big( \hat{Z}^* \big)^2
 \big)G_1^*-\hat{\tau}G_1G_1^*
  \\
  &+\Big( P_1-\frac{|P_2|^2}{P_1} \Big)\Big( 
   \hat{Z}G_1+\frac{ 1}{2}G_1\hat{Z} 
  +G_1^2
   \Big)\Big( 
   \hat{Z}G_1+\frac{1}{2}G_1\hat{Z} 
  +G_1^2
   \Big)^*,
\end{aligned}
\end{equation}
and 
\begin{equation}\label{tildeK11N}
\widehat{F}_{1,1}=
\mathbb{I}_n\otimes (Q_0Q_0^*)-
\overline{P}_2\hat{Z}^2\otimes \mathbb{I}_n
-\hat{Z}\otimes 
\Big(
\big(
P_1\hat{Z}^*+\overline{P}_2\hat{Z}
\big)+P_1G_1^*+\overline{P}_2G_1 
\Big),
\end{equation} 
\begin{equation}\label{tildeK12N}
\widehat{F}_{1,2}=Y^*
\otimes 
\Big(
\big(
P_1\hat{Z}+P_2\hat{Z}^*
\big)+P_1G_1+P_2G_1^*
\Big).
\end{equation} 

 To simplify    $I_0$ further, we first   integrate out  $Q_{t+1}, G_3, \ldots, G_t$  and obtain 
\begin{equation}\label{Qt1 integral critical non0}
\begin{aligned}
\int%_{\mathbb{C}^{n\times r_{t+1}}}
&e^{-{\rm Tr}\big( 
 Q_{t+1}(z_0\mathbb{I}_{r_{t+1}}-A_{t+1})^*
 (z_0\mathbb{I}_{r_{t+1}}-A_{t+1})Q_{t+1}^*
 \big)}{\rm d}Q_{t+1}= \big| \det\big( z_0\mathbb{I}_{r_{t+1}}
 -A_{t+1} \big) \big|^{-2n}\pi^{nr_{t+1}},
% \\& \big| \det\big( z_0\mathbb{I}_{r_{t+1}}
 %-A_{t+1} \big) \big|^{-2n}\pi^{nr_{t+1}}.
\end{aligned}
\end{equation}   
and \begin{equation}\label{galpha integral critical non0}
\begin{aligned}
\prod_{\alpha=3}^t \int
e^{- {\rm Tr}\big(
 G_{\alpha}-\overline{\eta}_{\alpha}\big(
 \hat{Z}G_1+\frac{1}{2}G_1\hat{Z} 
 +G_1^2
 \big)
 \big)\big(
 G_{\alpha}-\overline{\eta}_{\alpha}\big(
  \hat{Z}G_1+\frac{1}{2}G_1\hat{Z} 
 +G_1^2
 \big)
 \big)^*  }
 {\rm d}G_{\alpha}
 =\pi^{\frac{n(n-1)}{2}(t-2)}.
\end{aligned}
\end{equation}

Secondly,  we need to integrate out $ \widetilde{\mathcal{T}}_{\rm d,\alpha}$ for $\alpha=2,\cdots,t$. Recalling 
 $\Sigma, X,$ and   $B$   given in   \eqref{Sigma}  and \eqref{XB},   for $i=1,\cdots,n$, let 
\begin{equation}\label{tildeki}
\widetilde{k}_i=\big(
\widetilde{k}_i^{(2)},\cdots,
\widetilde{k}_i^{(t)}
\big),\quad
\widetilde{k}_i^{(\alpha)}=  \Big(
\frac{1}{f_{\alpha}} H_{a_{\alpha}-z_0} ^2
-\frac{1}{f_{1}} H_{a_{1}-z_0} ^2+\hat{\tau}(f_{\alpha}-f_1)
\Big)_{i,i}.
%\frac{\big(\widetilde{H}_i^{(\alpha)}\big)^2}{f_{\alpha}}
%-\frac{\big(\widetilde{H}_i^{(1)}\big)^2}{f_{1}}
%+\hat{\tau}(f_{\alpha}-f_1).
\end{equation} 
Noting the relation \eqref{newT} and  
\begin{equation}\label{Aalphabeta det}
\det(\Sigma)=2P_1\prod_{\alpha=1}^t 
\frac{f_{\alpha}^2}{2c_{\alpha}},%\quad
%\big( A^{-1} \big)_{\alpha,\beta}=
%\frac{2c_{\alpha}}{f_{\alpha}^2}\delta_{\alpha,\beta}
%-\frac{2c_{\alpha}c_{\beta}}{f_{\alpha}^2f_{\beta}^2P_1},\quad
%\alpha,\beta=2,\cdots,t.
\end{equation} 
   calculate the gaussian  integrals according to row vectors of $X$ and we obtain 
\begin{equation}\label{tildeUalpha integral1}
\begin{aligned}
\int  \cdots \int  &\exp
\Big\{  -{\rm Tr}\Big( (X-\frac{1}{2}B\Sigma^{-1})\Sigma (X-\frac{1}{2}B\Sigma^{-1})^t\Big)        
+\\
&
   {\rm Tr} \bigg(
\sum_{\alpha=2}^t  \widetilde{\mathcal{T}}_{\rm d,\alpha}
\Big(
\frac{1}{f_{\alpha}} H_{a_{\alpha}-z_0} ^2
-\frac{1}{f_{1}} H_{a_{1}-z_0} ^2+\hat{\tau}(f_{\alpha}-f_1)
\Big)\bigg)\Big\}\prod_{\alpha=2}^t {\rm d} \widetilde{\mathcal{T}}_{\rm d,\alpha}\\
%\prod_{i=1}^n\int 
%&e^{-\widetilde{x}_i A \widetilde{x}_i^t+
% \sum_{\alpha=2}^t 
% \widetilde{k}_i^{(\alpha)}\widetilde{U}_i^{(\alpha)}}
%\prod_{\alpha=2}^t {\rm d}\widetilde{U}_i^{(\alpha)}
&=
\Big( 2P_1\prod_{\alpha=1}^t 
\frac{f_{\alpha}^2}{2c_{\alpha}}  
\Big)^{-\frac{n}{2}}\pi^{\frac{t-1}{2}n}
e^{\frac{1}{4}\sum_{i=1}^n
\widetilde{k}_i \Sigma^{-1}\widetilde{k}_i^t}.
\end{aligned}
\end{equation}
Obviously, % $\widetilde{k}_1^{(\alpha)}=0$, so
$$
\sum_{i=1}^n \sum_{\alpha=2}^t \frac{c_{\alpha}}{f_{\alpha}^2} 
\widetilde{k}_i^{(\alpha)}
%=\sum_{\alpha=1}^t \frac{c_{\alpha}}{f_{\alpha}^2} \widetilde{k}_i^{(\alpha)}
=\sum_{\alpha=1}^t 
\frac{c_{\alpha}}{f_{\alpha}^3} 
 {\rm Tr}(H_{a_{\alpha}-z_0} ^2)
-\frac{P_1}{f_1}  {\rm Tr}(H_{a_{1}-z_0} ^2)-  n\hat{\tau}f_1P_1
+n\hat{\tau}
$$
and
\begin{equation*}
\begin{aligned}
\sum_{i=1}^n\sum_{\alpha=2}^t \frac{c_{\alpha}}{f_{\alpha}^2} 
\big(\widetilde{k}_i^{(\alpha)}\big)^2
%&=\sum_{\alpha=1}^t \frac{c_{\alpha}}{f_{\alpha}^2} 
%\Big(
%\frac{\big(\widetilde{H}_i^{(\alpha)}\big)^2}{f_{\alpha}}
%-\frac{\big(\widetilde{H}_i^{(1)}\big)^2}{f_{1}}
%+\hat{\tau}(f_{\alpha}-f_1)
%\Big)^2            \\
&=\sum_{\alpha=1}^t \frac{c_{\alpha}}{f_{\alpha}^2} 
 {\rm Tr} \Big( 
\frac{1}{f_{\alpha}} H_{a_{\alpha}-z_0} ^2-
\frac{1}{f_{1}} H_{a_{1}-z_0} ^2
\Big)^2+
n\hat{\tau}^2\sum_{\alpha=1}^t \frac{c_{\alpha}}{f_{\alpha}^2}
(f_{\alpha}-f_1)^2\\
& +2\hat{\tau}\sum_{\alpha=1}^t (f_{\alpha}-f_1)  \frac{c_{\alpha}}{f_{\alpha}^2} 
  {\rm Tr} \Big( 
\frac{1}{f_{\alpha}} H_{a_{\alpha}-z_0} ^2-
\frac{1}{f_{1}} H_{a_{1}-z_0} ^2
\Big).
\end{aligned}
\end{equation*}
By \eqref{Ainverse critical non0}, elementary calculation gives us 
\begin{equation}\label{tildeUalpha integral final}
\begin{aligned}
\frac{1}{4}\sum_{i=1}^n
\widetilde{k}_i \Sigma^{-1}\widetilde{k}_i^t
& =\frac{1}{2}\sum_{i=1}^n\Big(
\sum_{\alpha=2}^t \frac{c_{\alpha}}{f_{\alpha}^2} 
\big(\widetilde{k}_i^{(\alpha)}\big)^2-\frac{1}{P_1}
\Big(
\sum_{\alpha=2}^t \frac{c_{\alpha}}{f_{\alpha}^2} 
\widetilde{k}_i^{(\alpha)}
\Big)^2
\Big).\\
&=
\frac{1}{2}
\sum_{\alpha=1}^t \frac{c_{\alpha}}{f_{\alpha}^4}
{\rm Tr}\big(H_{a_{\alpha}-z_0}^4 \big)
-\frac{1}{2P_1}
{\rm Tr}\Big( 
\sum_{\alpha=1}^t \frac{c_{\alpha} }
{f_{\alpha}^3}H_{a_{\alpha}-z_0}^2
\Big)^2
\\&+
\hat{\tau}
\Big( 
\sum_{\alpha=1}^t \frac{c_{\alpha}}
{f_{\alpha}^2} {\rm Tr}( H_{a_{\alpha}-z_0}^2)- 
\frac{1}{P_1}\sum_{\alpha=1}^t
\frac{c_{\alpha}}
{f_{\alpha}^3}{\rm Tr} (H_{a_{\alpha}-z_0}^2)
\Big)+\frac{n\hat{\tau}^2}{2}\big( 1-\frac{1}{P_1} \big).
\end{aligned}
\end{equation}  

Thirdly, 
to  integrate out $\hat{G}_2$ and $\Gamma_1$, by  Lemma \ref{matrixintegral1}, we have
\begin{equation}\label{g2hat integration}
\int_{\Gamma_1+\sqrt{P_1}\big( \hat{G}_2+\hat{G}_2^* \big)\leq 0}
e^{|z_0|^2{\rm Tr}(\Gamma_1)}
{\rm d}\hat{G}_2{\rm d}\Gamma_1
=
\Big( \frac{\pi}{P_1} \Big)^{\frac{n(n-1)}{2}}
|z_0|^{-2n^2}
\prod_{i=1}^{n-1} i!
.
\end{equation}
So far, we have completed  a  simplification of $I_0$ in \eqref{I0 critical non0}. 

 Finally, with \eqref{norm-1} and \eqref{DNn} in mind,  by the  Stirling's formula we can obtain
%%%%%%%%%%%%%%%%%%%%%%%%%
\begin{equation}\label{eNf0 DNn critical non0}
\begin{aligned}
&e^{NF_0}D_{N,n}=\frac
{N^{\frac{n^2t}{2}+n(n+r_0+r_{t+1})
+\frac{n(n+1)}{2}-\frac{n(n-1)}{4}}}
{\pi^{n(n+1+r_{t+1}+r_{0})+\frac{n(n-1)}{2}t}
(2\pi)^{\frac{n t}{2}}
}
\frac{\prod_{\alpha=1}^t c_{\alpha}^{-\frac{n}{2}}
f_{\alpha}^{-\frac{n(n-1)}{2}}}
{\prod_{i=1}^{n-1} i!}      \\
&\times \prod_{1\leq i<j\leq n}
\big| \hat{z}_i-\hat{z}_j \big|^2
e^{-\frac{1}{4}
\sum_{\alpha=1}^t
 \frac{c_{\alpha}}{f_{\alpha}^4}{\rm Tr}\big(H_{a_{\alpha}-z_0}^4 \big)-\frac{1}{4}\sum_{\alpha=1}^t
 \frac{c_{\alpha}}{f_{\alpha}^4}\Big( 
\overline{z_0-a_{\alpha}}^4{\rm Tr}\big(\hat{Z}^4\big)
+(z_0-a_{\alpha})^4{\rm Tr}\big(\hat{Z}^*\big)^4
 \Big)}
 \\&
 \times e^{
 -\frac{n\hat{\tau}^2}{2}+\hat{\tau}
 {\rm Tr}\big( \hat{Z}\hat{Z}^* \big)-
 \hat{\tau}\sum_{\alpha=1}^t
 \frac{c_{\alpha}}{f_{\alpha}^2}{\rm Tr}\big( H_{a_{\alpha}-z_0}^2 \big)
 }\Big(
 1+O\big( N^{-\frac{1}{4}} \big)
 \Big).
\end{aligned}
\end{equation}

Starting from \eqref{RNndelta}, and combining \eqref{INdeltaexpan critical non0}, \eqref{J2Nestimation critical non0}, \eqref{J1N change critical non0}, \eqref{I0 critical non0}, \eqref{Qt1 integral critical non0}, \eqref{galpha integral critical non0}, \eqref{tildeUalpha integral1}, \eqref{tildeUalpha integral final}, \eqref{g2hat integration} and \eqref{eNf0 DNn critical non0},   we can arrive at 
\begin{small}
\begin{equation}\label{RNn pre critical non0}
\begin{aligned}
&\frac{1}{N^{\frac{n}{2}}}R_N^{(n)}\big(
X_0;z_0+N^{-\frac{1}{4}}\hat{z}_1,\cdots,
z_0+N^{-\frac{1}{4}}\hat{z}_n
\big)=
\\
&P_1^{-\frac{n^2}{2}}\frac{\prod_{1\leq i<j\leq n}
\big| \hat{z}_i-\hat{z}_j \big|^2}
{\pi^{n(n+1+r_0)+\frac{n(n-1)}{2}}(2\pi)^{\frac{n}{2}}}
e^{\hat{\tau}
 {\rm Tr}\big( \hat{Z}\hat{Z}^* \big)
 -\frac{\hat{\tau}}{P_1}\sum_{\alpha=1}^t
 \frac{c_{\alpha}}{f_{\alpha}^3}
 {\rm Tr}\big( H_{a_{\alpha}-z_0}^2 \big)
-\frac{n\hat{\tau}^2}{2P_1} 
 }          \\
 &\times
 e^{
 \frac{1}{4}
\sum_{\alpha=1}^t
 \frac{c_{\alpha}}{f_{\alpha}^4}{\rm Tr}\big( H_{a_{\alpha}-z_0}^4 \big)-\frac{1}{4}\sum_{\alpha=1}^t
 \frac{c_{\alpha}}{f_{\alpha}^4}\big( 
\overline{z_0-a_{\alpha}}^4{\rm Tr}\big(\hat{Z}^4\big)
+(z_0-a_{\alpha})^4{\rm Tr}\big(\hat{Z}^*\big)^4
 \big)
 -\frac{1}{2P_1}{\rm Tr}\big(
 \sum_{\alpha=1}^t
 \frac{c_{\alpha}}{f_{\alpha}^3}H_{a_{\alpha}-z_0}^2
 \big)^2
 }      \\
 &\times
 \int \bigg( \det\begin{bmatrix}
\hat{Z}
 & -Y^*
 \\ Y & 
\hat{Z}^*
\end{bmatrix}\bigg)^{r_0-n}
\det\begin{bmatrix}
P_1(Y^*Y)\otimes \mathbb{I}_n+\widehat{F}_{1,1} 
& \widehat{F}_{1,2}
+\big(
P_1\hat{Z}Y^*+P_2Y^*\hat{Z}^*
\big)\otimes \mathbb{I}_n      \\
-\widehat{F}_{1,2}^*-\big(
P_1\hat{Z}^*Y+\overline{P}_2 Y\hat{Z}
\big)\otimes \mathbb{I}_n & 
\widehat{F}_{1,1}^*+P_1(YY^*)\otimes \mathbb{I}_n
\end{bmatrix}
\\
&\times
e^{{\rm Tr}\big( \hat{K}_3(G_1) \big)
 -{\rm Tr}\big(\hat{K}_2(G_1,Q_0)
 \big( \frac{1}{2}\hat{K}_2(G_1,Q_0)+\hat{K}_1(G_1) \big)
 \big) -\frac{\hat{\tau}}{P_1}{\rm Tr}(Q_0Q_0^*)-\frac{P_1}{2}{\rm Tr}(YY^*)^2}
 \\&\times
 e^{\overline{P}_2{\rm Tr}\big( Y^*Y\hat{Z}^2 \big)
 +P_1{\rm Tr}\big( Y\hat{Z}Y^*\hat{Z}^* \big)
 +P_2{\rm Tr}\Big( YY^*\big(\hat{Z}^*\big)^2 \Big)
 +\hat{\tau}{\rm Tr}\big( YY^* \big)}
 {\rm d}G_1{\rm d}Y{\rm d}Q_0+O\big( N^{-\frac{1}{4}} \big).
\end{aligned}
\end{equation}
\end{small}

Noticing  the identity \eqref{H2power}  and the simple fact %Recall the definition of $\widetilde{H}_{\alpha}$ \eqref{tildeHalpha critical non0} we have
%\begin{equation}\label{finalTrace1}
%\sum_{\alpha=1}^t
% \frac{c_{\alpha}\widetilde{H}_{\alpha}^2}{f_{\alpha}^3}
% =P_2\big(\hat{Z}^*\big)^2
% +2P_1\hat{Z}\hat{Z}^*
% +\overline{P}_2\hat{Z}^2
%\end{equation}
\begin{equation}\label{finalTrace2}
\begin{aligned}
&\frac{1}{4}
\sum_{\alpha=1}^t
 \frac{c_{\alpha}}{f_{\alpha}^4}{\rm Tr}\big( H_{a_{\alpha}-z_0}^4 \big)-\frac{1}{4}\sum_{\alpha=1}^t
 \frac{c_{\alpha}}{f_{\alpha}^4}\Big( 
\overline{z_0-a_{\alpha}}^4{\rm Tr}\big(\hat{Z}^4\big)
+(z_0-a_{\alpha})^4{\rm Tr}\big(\hat{Z}^*\big)^4
 \Big)
 \\&=
 P_2{\rm Tr}\big(\hat{Z}\big(\hat{Z}^*\big)^3\big)
 +\frac{3P_1}{2}{\rm Tr}\big(\hat{Z}^2\big(\hat{Z}^*\big)^2\big)
 +\overline{P}_2{\rm Tr}\big(\hat{Z}^3\hat{Z}^*\big),
\end{aligned}
\end{equation}
  changing variables  in \eqref{RNn pre critical non0} like  $$\hat{Z}\rightarrow P_1^{-\frac{1}{4}}\hat{Z},\quad
\hat{\tau}\rightarrow \sqrt{P_1}\hat{\tau},\quad
Q_0\rightarrow P_1^{\frac{1}{4}}W,\quad
G_1\rightarrow P_1^{-\frac{1}{4}}T,\quad
Y\rightarrow P_1^{-\frac{1}{4}}Y,
$$
we thus complete the proof of Theorem \ref{2-complex-correlation critical2} after such a  long journey of  calculation!
%%%%%%%%%%%%%%%%%%%%%%%%%%%%%%%%%%%%%%%%%%%%%%%%%%%%%%%%%%%%%%%%%%%%%555
%%%%%%%%%%%%%%%%%%%%%%%%%%%%%%%%%%%%%%%%%%%%%%%%%%%%%%%%%%%%%%%%%%%%%%%%%%%%%%%%%55

\hspace*{\fill}
%%%%%%%%%%%%%%%%%%%%%%%%%%%%%%%%%%%%%%%%%%%%%%%%%%%%%%%%%%%%%%%%%%%%%%%%%%%%%%%%%%%%%%%%%%%%%%%%%%%%%%%%%%%%%%%%%%%%%%%%%%%%%%%%%%%%%%%%%%%%%%%%%%%%%%%%%%%%%%%%%%%%%%%%%%%%%%%%%%%%%%%%%%%%%%%%%%%%%%%%%%%%%%%%%%%%%%%%%%%%%%

 \noindent{\bf Acknowledgements}  
%We acknowledge support by the National Natural Science Foundation of China #11771417, the Youth Innovation Promotion Association CAS #2017491 (D.-Z. Liu), and by the National Natural Science Foundation of China #11901161 (Y. Wang).
% We  would like to thank     Elton P. Hsu for  his  encouragement and support,  L.  Erd\H{o}s  for  sharing his vision for  non-Hermitian random matrices,  and also  to thank Y.V. Fyodorov,  J. Grela and Y. Wang for useful comments on the  first arXiv version. 
We would like to thank  L. Erd\H{o}s for sharing his vision for non-Hermitian random matrices, and also to
thank  P.J. Forrester, Y. Xu and P. Zhong  for valuable  comments. 
This  work  was   supported by  the National Natural Science Foundation of China \#12371157  and \#12090012.
 %of D.-Z.~Liu  

%\hspace*{\fill}
%
% \noindent{\bf Conflict of interest}\   
% The authors do not have any potential conflicts of interests to disclose.
%
%\hspace*{\fill}
%
%  \noindent{\bf Data availability}\  
% Data sharing is not applicable to this article as no datasets were generated or analysed during the current study.
% 

  \appendix
 \section{Several  properties on matrices and integrals} \label{Appendix}

 The  following well-known property for tensor product (see e.g. \cite[eqn(6)]{HS81}).
\begin{proposition}\label{tensorproperty}
For  a $p\times q$  matrix  A and an $m\times n$ matrix   B,  there exist  permutation matrices $\mathbb{I}_{p,m}$,  which only depend on  $p, m$ and satisfy  $\mathbb{I}_{p,m}^{-1}=\mathbb{I}_{m,p}=\mathbb{I}_{p,m}^t$, such that 
%\begin{small}
%\begin{equation}\label{permutationinvrse}
%\mathbb{I}_{p,m}^{-1}=\mathbb{I}_{m,p}=\mathbb{I}_{p,m}^t.
%\end{equation}
%\end{small}

\begin{equation}\label{tensorpropertyequation}
\mathbb{I}_{p,m}\left( A\otimes B \right)\mathbb{I}_{q,n}^{-1}=B\otimes A.
\end{equation}
  \end{proposition}

 A partial   integral  over non-positive definite matrices  will be useful in this paper.
\begin{proposition}\label{matrixintegral1}
For  a  non-positive definite Hermitian matrix  $H_n=\left[ h_{i,j} \right]_{i,j=1}^n$,  given an integer  $ r_0\ge n$,  we have 
\begin{equation}\label{matrixintegral1equ}
\int_{H_n\leq 0} \big( \det(H_n) \big)^{r_0-n} \prod_{i<j}^n{\rm d}h_{i,j}
=
 \pi^{\frac{n(n-1)}{2}}  ( (r_0-1)! )^{-n} \prod_{k=1}^n   (r_0-k)!  \prod_{j=1}^n  (h_{j,j})^{r_0-1}.
\end{equation}
\end{proposition}

\begin{proof}
We prove it by induction.  Without loss of generality, we assume that  $H_n< 0$. Let
\begin{equation}
H_{n-k}=\begin{bmatrix}
H_{n-k-1}  &   E_{n-k}  \\
E_{n-k}^*   &  h_{n-k,n-k}
\end{bmatrix},
\quad
k=0,1,\cdots,n-2,
 \end{equation} and 
\begin{equation}
I_n:=\int_{H_n\leq 0}
\left( \det(H_n) \right)^{r_0-n} \prod_{i<j}^n{\rm d}h_{i,j}.
\end{equation}

By the matrix  identity 
\begin{small}
\begin{equation}
\begin{aligned}
&\begin{bmatrix}
\mathbb{I}_{n-k-1} & \\
-E_{n-k}^*H_{n-k-1}^{-1}  &    1
\end{bmatrix}H_{n-k}
\begin{bmatrix} \label{Schur}
\mathbb{I}_{n-k-1} &   - H_{n-k-1}^{-1}E_{n-k}        \\
  &    1
\end{bmatrix}       \\
&=
\begin{bmatrix}
H_{n-k-1} &  \\
&   h_{n-k,n-k}-E_{n-k}^*H_{n-k-1}^{-1} E_{n-k} 
\end{bmatrix},
\end{aligned}
\end{equation}
\end{small}
we have
\begin{equation}
\begin{aligned}
&I_{n-k}=\int_{H_{n-k-1}\leq 0}
\left( \det(H_{n-k-1}) \right)^{r_0-n+k}{\rm d}E_{n-k-1}\cdots{\rm d}E_2
\\
&\times
\int_{E_{n-k}^*\left(-H_{n-k-1}\right)^{-1} E_{n-k}\leq -h_{n-k,n-k}}
\left(
 h_{n-k,n-k}+E_{n-k}^*\left(-H_{n-k-1}\right)^{-1}  E_{n-k} 
\right)^{r_0-n+k}{\rm d}E_{n-k}.
\end{aligned}
\end{equation}
Make a change of   variables
\begin{equation}
\widetilde{E}_{n-k}=
\left(-H_{n-k-1}\right)^{-\frac{1}{2}}E_{n-k}, 
\end{equation}
we then  have 
\begin{equation}
I_{n-k}=(-1)^{n-k-1}I_{n-k-1}\widetilde{I}_{n-k},
\end{equation}
where
\begin{equation}
\begin{aligned}
\widetilde{I}_{n-k}:&=
\int_{\widetilde{E}_{n-k}^*\widetilde{E}_{n-k}\leq -h_{n-k,n-k}}
\left(
 h_{n-k,n-k}+\widetilde{E}_{n-k}^*\widetilde{E}_{n-k}
\right)^{r_0-n+k}{\rm d}\widetilde{E}_{n-k}      \\
&=\frac{2\pi^{n-k-1}}{\Gamma(n-k-1)}\int_0^{\sqrt{- h_{n-k,n-k}}}
\left(
h_{n-k,n-k}+r^2
\right)^{r_0-n+k}
r^{2n-2k-3}{\rm d}r        \\
&=\pi^{n-k-1}(h_{n-k,n-k})^{r_0-1}\frac{(r_0-n+k)!}{(r_0-1)!},
\end{aligned}
\end{equation}
in which the second equality follows from the  spherical polar coordinates and the third follows from the identity
\begin{equation} \label{sumid}
\sum_{k=0}^l{l \choose k}\frac{(-1)^k}{n+k-l-1}=
\frac{l!}{(n-l-1)\cdots (n-1)}.
\end{equation}
%for positive integers n and l, which can be proved by induction on n and l.

Therefore, put  the above derivations together and  we  obtain
\begin{equation}\label{iterated1}
I_{n-k}= \pi^{n-k-1}(h_{n-k,n-k})^{r_0-1}\frac{(r_0-n+k)!}{(r_0-1)!}I_{n-k-1}.
\end{equation}
By repeating \eqref{iterated1}, we thus  arrive at \eqref{matrixintegral1}.
\end{proof}

%%%%%%%%%%%%%%%%%%%%%%%%%%%%%%%%%%%%%%%%%%%%%%%%%%%%%%%%%%%%%%%%%%%%%%%%%%%%%%%%%%%%%%%%%%%%%%%%%%%%%%%%%%%%%%%%%%%%%%%%%%%%%%%%%%%%%%%%%%%%%%%%%%%%%%%%%%%%%%%%%%%%%%%%%%%%%%%%%%%%%%%%%%%%%%%%%%%%%%%%%%%%%%%%%%%%%%%%%%%%%%
%\section{%Appendix B: Confluent limits}   \label{AppendixB}

 %We also thank the referee for several other useful comments.

%\begin{Acknowledgments}
%%%The work of P.J.~Forrester was supported by the Australian Research Council, for the project DP140102613.  

%   and the Youth Innovation Promotion Association CAS \#2017491.
%%%%, the Fundamental Research Funds for the Central Universities \#WK0010450002, and Anhui Provincial Natural Science Foundation \#1708085QA03.
%%%The authors are grateful to Alexander Soshnikov for many useful discussions and Yan Fyodorov for references. They are particularly thankful to Terry Tao for helpful discussions and enthusiastic encouragement. The authors would also like to thank the anonymous referees for their valuable comments and corrections
% \end{Acknowledgments}
%

 %%%%%%%%%%%%%%%%%%%

%\providecommand{\bysame}{\leavevmode\hbox to3em{\hrulefill}\thinspace}
%\providecommand{\MR}{\relax\ifhmode\unskip\space\fi MR }
%% \MRhref is called by the amsart/book/proc definition of \MR.
%\providecommand{\MRhref}[2]{%
%  \href{http://www.ams.org/mathscinet-getitem?mr=#1}{#2}
%}
%\providecommand{\href}[2]{#2}

%\setlength{bibsep}{0em}

 \end{document}